\documentclass[11pt,reqno,a4paper]{article}
\usepackage{mathrsfs}
\usepackage{amsmath}
\usepackage{amsthm}
\usepackage{amsfonts}
\usepackage{amssymb}
\usepackage[dvips]{graphicx}
\usepackage[dvips]{color}
\usepackage{latexsym}
\usepackage{enumerate}
\usepackage{xspace}

\newtheorem{thm}{Theorem}[section]
\newtheorem{cor}[thm]{Corollary}
\newtheorem{lemma}[thm]{Lemma}
\newtheorem{prop}[thm]{Proposition}

\newtheorem{defn}[thm]{Definition}
\newtheorem{example}[thm]{Example}

\numberwithin{equation}{section}

\def \and {\, \mbox{\rm and}\, }

\def \span {\,{\rm span}\,}
\def \supp {\,{\rm supp}\,}

\begin{document}

\title{\bf Operator Reproducing Kernel Hilbert Spaces}

\author{Rui Wang\thanks{School of Mathematics, Jilin University, Changchun 130012, P. R. China. E-mail address: {\it rwang11@jlu.edu.cn}. } \quad and \ Yuesheng Xu\thanks{Corresponding author. Guangdong Provincial Key Lab of Computational Science, School of Data
and Computer Science, Sun Yat-sen University, Guangzhou 510275, P. R. China and Department of Mathematics, Syracuse University, Syracuse, NY 13244, USA. E-mail address: {\it yxu06@syr.edu} } }
\date{}

\maketitle{}

\begin{abstract}
Motivated by the need of processing functional-valued data, or more general, operator-valued data, we introduce the notion of the operator reproducing kernel Hilbert space (ORKHS). This space admits a unique operator reproducing kernel which reproduces a family of continuous linear operators on the space. The theory of ORKHSs and the associated operator reproducing kernels are established. A special class of ORKHSs, known as the perfect ORKHSs, are studied, which reproduce the family of the standard point-evaluation operators and at the same time another different family of continuous linear operators. The perfect ORKHSs are characterized in terms of features, especially for those with respect to integral operators. In particular, several specific examples of the perfect ORKHSs  are presented. We apply the theory of ORKHSs to sampling and regularized learning, where operator-valued data are considered. Specifically, a general complete reconstruction formula from linear operators values is established in the framework of ORKHSs. The average sampling and the reconstruction of vector-valued functions are considered in specific ORKHSs. We also investigate in the ORKHSs setting the regularized learning schemes, which learn a target element from operator-valued data. The desired representer theorems of the learning problems are established to demonstrate the key roles played by the ORKHSs and the operator reproducing kernels in machine learning from operator-valued data. We finally point out that the continuity of linear operators, used to obtain the operator-valued data, on an ORKHS is necessary for the stability of the numerical reconstruction algorithm using the resulting data.
\end{abstract}



\section{Introduction}

The main purpose of this paper is to introduce the notion of the operator reproducing kernel Hilbert space and discuss the essential ideas behind such spaces along with their applications to sampling and machine learning. The introduction of such a notion is motivated by the need of processing various data (not necessarily the point-evaluations of a function) emerging in practical applications.

Many scientific problems come down to understanding a function from a finite set of its sampled data. Commonly used sampled data of a function to be constructed are its function values. Through out this paper, we reserve the term ``function value" for the evaluation of a function at a point in its domain. We call such data the point-evaluation data of the function. The stability of the sampling process requires that the point-evaluation functionals must be continuous. This leads to the notion of reproducing kernel Hilbert spaces (RKHSs) \cite{Ar} which has become commonly used spaces for processing point-evaluation data. There is a bijective correspondence between a RKHS and its reproducing kernel \cite{Mercer}. Moreover, reproducing kernels are used to measure the similarity between elements in an input space. In machine learning, many effective schemes based upon reproducing kernels were developed and proved to be useful \cite{CS,PS,SS,SC,V}. The theory of vector-valued RKHSs \cite{P} received considerable attention in multi-task learning which concerns estimating a vector-valued function from its function values. In the framework of vector-valued RKHSs, kernel methods were proposed to learn multiple related tasks simultaneously, \cite{EMP,MP05}. Recently, to learn a function in a Banach space which has more geometric structures than a Hilbert space, the notion of reproducing kernel Banach spaces (RKBSs) was introduced in \cite{ZXZ} and further developed in \cite{SZH,XQ,ZZ12,ZZ13}. The celebrated Shannon sampling formula can be regarded as an orthonormal expansion of a function in the Paley-Wiener space, which is a typical RKHS. Motivated by this fact, many general sampling theorems were established in RKHSs through frames or Riesz bases composed by reproducing kernels, \cite{H07,H09,HKK,MNS,NW}.

There are many practical occasions where point-evaluation data are not readily available. Sampling is to convert an analog signal into a sequence of numbers, which can then be processed digitally on a computer. These numbers are expected to be the values of the original signal at points in the domain of the signal. However, in most physical circumstances, one cannot measure the value $f(x)$ of a signal $f$ at the location $x$ exactly, due to the non-ideal acquisition device at the sampling location. A more suitable form of the sampled data should be the weighted-average value of $f$ in the neighborhood of the location $x$. The average functions may reflect the characteristic of the sampling devices. Such a sampling process is called an average sampling process \cite{A,S}. In signal analysis, sometimes what is known is the frequency information of a signal. These data, represented mathematically by the Fourier transform or wavelet transform, are clearly not the function values. Another example of non-point-evaluation sampled data concerns EXAFS (extended x-ray absorption fine structure) spectroscopy, which is a useful atomic-scale probe for studying the environment of an atomic species in a chemical system. One key problem in EXAFS is to obtain the atomic radial distribution function, which describes how density varies  as a function of the  distance from a reference particle to the points in the space. The measured quantity used to reconstruct the atomic radial distribution function is the x-ray absorption coefficient as a function of the wave vector modulus for the photoelectron. Theoretically, these observed information is determined by an integral of the product of the atomic radial distribution function and functions reflecting the characteristics of the sample and the physics of the electron scattering process \cite{CNSU,MP04}. In the areas of machine learning and functional data analysis, we are required to process measured information rather
than function values such as curves, surfaces and other geometric shapes \cite{RS}. All above mentioned non-point-evaluation sampled data can be formulated mathematically as linear functional values or linear operator values of a function. The term ``functional value" is reserved in this paper for the value of a functional applied to a function  in its domain. Function values are a special case of functional values when the linear functional is the point-evaluation functional. Likewise, we also reserve  the term ``operator value" for the value of an operator at a function in its domain. Large amount of such sampled data motivate us to consider constructing a function from its linear functional values or its linear operator values other than its function values.

There is considerable amount of work on processing non-point-evaluation sampled data in the literature. The reconstruction of a function in the band-limited space, the finitely generated shift-invariant spaces and more generally the spaces of signals with finite rate of innovation from its local average values were investigated, respectively, in \cite{G,SZ}, \cite{A,AST} and \cite{S}. Sampling and reconstruction of vector-valued functions were considered in \cite{AI08}, where the sampled data were defined as the inner product of the function values and some given vectors. It is clear that these sampled data can also be treated as the linear functional values of the vector-valued function. A number of alternative estimators were introduced for the functional linear regression based upon the functional principal component analysis \cite{CH,RS,YMW} or a regularization method in RKHSs \cite{YC}. Optimal approximation of a function $f$ from a general RKHS by algorithms using only its finitely many linear functional
values was widely studied in the area of information-based complexity \cite{NoWo,TWW}. In machine learning, regularized learning schemes were proposed for functional-valued data. For example, learning a function from local average integration functional values was investigated in \cite{V} in the framework of the supported vector machine. Regularized learning from general linear functional values in Banach spaces was considered in \cite{MP04}. Regularized learning in the context of RKBSs was studied in \cite{ZXZ,ZZ12} from both function-valued data and functional-valued data.

Reconstructing a function in a usual RKHS (where only the point-evaluation functionals are continuous) from its non-point-evaluation data is not appropriate in general since the non-point-evaluation functionals used in the processing may not be continuous in the underlying RKHS. When constructing a function from its linear functional values, the sampled data are generally obtained with noise. To reduce the effects of the noise, we require that the sampling process must be stable. Clearly, a usual RKHS that enjoys only the continuity of point-evaluation functionals is not a suitable candidate for processing a class of non-point-evaluation functionals. This demands the availability of a Hilbert space where the specific non-point-evaluation functionals used in the processing are continuous. Such spaces are presently not available in the literature. Therefore, it is desirable to construct a Hilbert space where a specific class of linear functionals or linear operators on it are continuous.

The goal of this paper is to study the operator reproducing kernel Hilbert space. We shall introduce the notion of the operator reproducing kernel Hilbert space by reviewing examples of data forms available in applications. Such a space ensures the {\it continuity} of a family of prescribed linear operators on it. An operator reproducing kernel Hilbert space is expected to admit a reproducing kernel, which reproduces the linear operators defining the space. We shall show that every operator reproducing kernel Hilbert space is isometrically isomorphic to a usual vector-valued RKHS. Although this does not mean that the study of operator reproducing kernel Hilbert spaces can be trivially reduced to that of vector-valued RKHSs, several important results about vector-valued RKHSs can be transferred to operator reproducing kernel Hilbert spaces by the isometric isomorphism procedure. In the literature, the average sampling problem was usually investigated in spaces which are typical RKHSs. We shall point out that these spaces are in fact special operator reproducing kernel Hilbert spaces. We shall call such a space a {\it perfect} operator reproducing kernel Hilbert space, since it admits two different families of continuous linear operators, among which one is the set of the standard point-evaluation operators. The perfect operator reproducing kernel Hilbert spaces, especially the ones with respect to integral operators, bring us special interest. The existing results about average sampling are special cases of the sampling theorem which we establish in the operator reproducing kernel Hilbert space setting, with additional information on the continuity of the used integral functionals on the space. In machine learning, learning a function from its functional values was also considered in RKHSs or RKBSs. In the framework of the new spaces, we study the regularized learning schemes for learning an element from non-point-evaluation operator values. By establishing the representer theorem, we shall show that the operator reproducing kernel Hilbert spaces are more suitable for learning from non-point-evaluation operator values than the usual RKHSs which are used in the existing literature. The theoretic results and specific examples in this paper all demonstrate that operator reproducing kernel Hilbert spaces provide a right framework for constructing an element from its non-point-evaluation continuous linear functional or operator values.

This paper is organized in nine sections. We review in section 2 several non-point-evaluation data often used in practical applications. Motivated by processing such forms of data, we introduce the notion of the operator reproducing kernel Hilbert space in section 3. We also identify an operator reproducing kernel with each operator reproducing kernel Hilbert space and reveal the relation between the new type of spaces and the usual vector-valued RKHSs. In section 4, we characterize the perfect operator reproducing kernel Hilbert spaces and the corresponding operator reproducing kernels and study the universality of these kernels. Motivated by the local averages of functions, we investigate in section 5 the perfect operator reproducing kernel Hilbert space with respect to a family of integral operators and present two specific examples. Based upon the framework of operator reproducing kernel Hilbert spaces, we establish a general reconstruction formula of an element from its operator values in section 6 and develop the regularized learning algorithm for learning from non-point-evaluation operator values in section 7. In section 8, we comment on stability of a numerical reconstruction algorithm using operator values in an operator reproducing kernel Hilbert space. Finally, we draw a conclusion in section 9.

\section{Non-Point-Evaluation Data}

Point-evaluation function values are often used to construct a function in the classical RKHS setting. We encounter in both theoretical analysis and practical applications huge amount of data which are not point-evaluation function values such as integral values of lines and surfaces, geometric objects, and operator values. The purpose of this section is to review several such forms of data, which serve as motivation of introducing the notion of operator reproducing kernel Hilbert space.

Classical methods to construct a function are based on
its finite samples. The finite empirical data of a function $f$ from points in its domain $\mathcal{X}$ to $\mathbb{R}$ or $\mathbb{C}$ are usually taken as the values of $f$ on the finite set $\{x_j:j\in\mathbb{N}_m\}\subseteq\mathcal{X}$, where $\mathbb{N}_m:=\{1,2,\dots, m\}$. We call them the
point-evaluation data. Various approaches were developed for processing the point-evaluation data in the literature. The classical RKHS is the function space suitable for the point-evaluation data. However, because of a variety of reasons, the available data in practical applications are not always in the form of point evaluations. We shall call data that are not point-evaluation function values the non-point-evaluation data.

The first type of non-point-evaluation data concerns the functional-valued data. In most physical circumstances, although the values of a function $f$ are expected to be sampled and processed on computers, one can not measure the value $f(x)$ at the location $x$ exactly because of the non-ideal acquisition device \cite{A,S}. Instead of the function values at the sample points, one obtains the finite local average values of $f$ in the neighborhood of the sample points. Specifically, we suppose that $\{x_j:j\in\mathbb{N}_m\}$ is a finite set of $\mathbb{R}$ and $\{u_j: j\in\mathbb{N}_m\}$ is a set of nonnegative functions on $\mathbb{R}$, where for each $j\in\mathbb{N}_m$, the function $u_j$ supports in a neighborhood of $x_j$. The sampled data have the form
$\{L_j(f):j\in\mathbb{N}_m\}$, where for each $j\in\mathbb{N}_m$, $L_j$ is defined by
$$
L_{j}(f):=\int_{\mathbb{R}}f(x)u_{j}(x)dx,
$$
which can be viewed as a functional of $f$. We note that when the support of $u_j$ is small, the functional value $L_j(f)$ can be treated as an approximation of the point-evaluation function value of $f$ at the location $x_j$.

There are other practical occasions, where the available information is not point-evaluation data or the local-average-valued data. The frequency information of a signal are such non-point-evaluation data. Specifically, suppose that $f$ is a $2\pi$-periodic signal. The Fourier coefficients of $f$ are defined by
$$
c_j(f):=\frac{1}{\sqrt{2\pi}}\int_{0}^{2\pi}f(x)e^{-ijx}dx,
\ j\in\mathbb{Z}.
$$
These coefficients, taken as the functional-valued data, contain the information of the signal in the frequency domain. Such sampled data are usually used to recover the signal in the time domain.

Another example concerns the x-ray absorption coefficients of an atomic radial distribution function in EXAFS spectroscopy \cite{CNSU,MP04}. Within the EXAFS theory accepted generally, the x-ray absorption coefficient of the atomic radial distribution function $f$ is described by
\begin{equation*}\label{EXAFS}
\chi(\kappa):=\int_{a}^{b}G(\kappa,r)f(r)dr,
\end{equation*}
where $G$ is the kernel defined by
$$
G(\kappa,r):=\displaystyle{4\pi\rho\phi(\kappa) e^{-\frac{2r}{\lambda(\kappa)}}\sin(2\kappa r
+\psi(\kappa))}.
$$
Here, the constant $\rho$ and the functions $\lambda, \phi, \psi$ are used to describe physical parameters of the Radon transform. Specifically, $\rho$ is the sample density, $\lambda$ is the mean-free-path of the photo-electron, $\phi$ is the scattering amplitude of the electron scattering process and $\psi$ is the phase-shifts of the electron scattering process. To determine the atomic radial distribution function $f$, measurements are made giving values $\chi(\kappa_j)$, $j\in\mathbb{N}_m$. There were many methods developed in the literature to learn the atomic radial distribution function from such sampled data. It is clear that for each $\kappa$, the quantity $\chi(\kappa)$ can be viewed as a functional value of the atomic radial distribution function $f$.

The sampled data emerge in the cases discussed above are not the point-evaluation data but the functional-valued ones. This will motivate us to consider estimating the target function $f$ from the finite information $\{L_j(f): j\in\mathbb{N}_m\}$, where $L_j$, $j\in\mathbb{N}_m,$ are linear functionals of $f$.

Operator-valued data often appear in applications. In machine learning, multiple related learning tasks need to be learned simultaneously. For example, multi-modal human computer interface requires modeling of both speech and vision. This can improve performance relative to learning each task independently, which can not capture the relations between these tasks. Empirical studies have shown the benefits of such multi-task learning \cite{BH,Ca,EMP}. This leads us to learning functions taking their range in a finite-dimensional Euclidean space $\mathbb{C}^n$, or more generally, a Hilbert space $\mathcal{Y}$. In this case, the measured information to be processed is the finite function values $f(x_j)$, $j\in\mathbb{N}_m,$ on a discrete set in the domain. However, they are vectors in a Hilbert space $\mathcal{Y}$ rather than scalars in $\mathbb{C}$. The functional data, considered in the area of functional data analysis, consist of functions. Examples of this type range from the shapes of bones excavated by archaeologists, to economic data collected over many years, to the path traced out by a juggler¡¯s finger \cite{RS}. In these occasions, the known information can be seen as the operator-valued data. This will lead us to a study of estimating  a function $f$ from finite sampled data $\{L_j(f): j\in\mathbb{N}_m\}$, where $L_j$, $j\in\mathbb{N}_m,$ are linear operators of $f$, taking values in a Hilbert space $\mathcal{Y}$.

Finally, we describe types of data used in sampling. Motivated by the average sampling, one may consider reconstruction of a function $f$ from functional-valued data $\{L_j(f): j\in\mathbb{J}\}$, where $\mathbb{J}$ is an index set and $L_j$, $j\in\mathbb{J}$, is a sequence of linear functionals of $f$. For vector-valued functions, there are two types of data used in the reconstruction formula. One concerns the point-evaluation data
$\{f(x_j): j\in\mathbb{J}\}$. Another one consists linear combinations of the components of each point-evaluation data $f(x_j)$. Following \cite{AI08}, we shall consider reconstructing a function $f$ from the finite sampled data, defined as the inner product
$$
\langle L_j(f),\xi_j\rangle_{\mathcal{Y}},
\ \ j\in\mathbb{J},
$$
where $\mathbb{J}$ is an index set,
$L_j$, $j\in\mathbb{J},$ are linear operators of $f$ taking values in a Hilbert space $\mathcal{Y}$ and
$\{\xi_j: j\in\mathbb{J}\}$ is a subset in $\mathcal{Y}$. In fact, we shall only consider the sampling and the reconstruction problem with respect to the functional values. This is because if we introduce a family of functionals $\widetilde{L}_{j}, j\in\mathbb{J},$ as
$$
\widetilde{L}_{j}(f)=\langle L_{j}(f), \xi_j
\rangle_{\mathcal{Y}},\ j\in\mathbb{J},
$$
the sampling problem with respect to the operator values can also be taken as reconstructing an element $f$ from its functional values $\widetilde{L}_{j}(f),
\ j\in\mathbb{J}$.

It is known that RKHSs are ideal spaces for processing the point-evaluation data. However, these spaces are no longer suitable for processing non-point-evaluation data. Hence, it is desirable to introduce spaces suitable for processing non-point-evaluation data. This is the focus
of this paper.

\section{The Notion of the Operator Reproducing Kernel Hilbert Space}

We introduce in this section the notion of the operator reproducing kernel Hilbert space. We identify the classical scalar-valued and vector-valued RKHSs as its special examples. We also present nontrivial examples of operator reproducing kernel Hilbert spaces with respect to integral operators. In a way similar to RKHSs, each operator reproducing kernel Hilbert space has exactly one operator reproducing kernel to reproduce the operator values of elements in the space. Based upon the isometric isomorphism between an operator reproducing kernel Hilbert space and a vector-valued RKHS, we characterize an operator reproducing kernel by its features.

Let $\mathcal{H}$ denote a Hilbert space with inner product $\left<\cdot, \cdot\right>_\mathcal{H}$ and norm  $\|\cdot\|_{\mathcal{H}}:=\left<\cdot, \cdot
\right>_\mathcal{H}^{1/2}$. For a Hilbert space $\mathcal{Y}$ and a family $\mathcal{L}$ of linear operators from $\mathcal{H}$ to $\mathcal{Y}$, we say that the norm of $\mathcal{H}$ is compatible with $\mathcal{L}$ if for each $f\in \mathcal{H}$, $L(f)=0$ for all
$L\in\mathcal{L}$ if and only if $\|f\|_{\mathcal{H}}=0$.

\begin{defn}\label{def_ORKHS}
Let $\mathcal{H}$ be a Hilbert space and $\mathcal{L}$ a family of linear operators from $\mathcal{H}$ to a Hilbert space $\mathcal{Y}$. Space $\mathcal{H}$ is called an
{\it operator reproducing kernel Hilbert space} with respect to $\mathcal{L}$ if the norm of $\mathcal{H}$ is compatible with $\mathcal{L}$ and each linear operator in $\mathcal{L}$ is continuous on $\mathcal{H}$.
\end{defn}

In this paper, we shall use the abbreviation ORKHS to represent an operator reproducing kernel Hilbert space. We remark that unlike the typical RKHS, ORKHS is not always composed of functions, but general vectors. We shall see that an ORKHS is associated with a kernel which reproduces the linear operators in $\mathcal{L}$. It is convenient to associate the family $\mathcal{L}$ of linear operators with an index set $\Lambda$, mainly, $\mathcal{L}:=\{L_{\alpha}:\alpha\in \Lambda\}$.

In the scalar case when $\mathcal{Y}=\mathbb{C}$, where $\mathbb{C}$ denotes the set of complex numbers, if a Hilbert space $\mathcal{H}$ and a family $\mathcal{F}$ of linear functionals on $\mathcal{H}$ satisfy Definition \ref{def_ORKHS}, we shall call $\mathcal{H}$ a
{\it functional reproducing kernel Hilbert space} with respect to $\mathcal{F}$. We shall use the abbreviation FRKHS to represent a functional reproducing kernel Hilbert space. We remark that for any Hilbert space $\mathcal{H}$ there always exists a family $\mathcal{F}$ of linear functionals such that $\mathcal{H}$ becomes an FRKHS with respect to $\mathcal{F}$. In fact, we can choose the family $\mathcal{F}$ of linear functionals as the set of all the continuous linear functionals on $\mathcal{H}$. That is, $\mathcal{F}=\mathcal{H}^{*}$, the dual space of $\mathcal{H}$. Since $\mathcal{H}^{*}$ is isometrically anti-isomorphic to $\mathcal{H}$, we can easily obtain that the norm of $\mathcal{H}$ is compatible with $\mathcal{F}$ and then $\mathcal{H}$ is an FRKHS with respect to $\mathcal{F}$. Although any Hilbert space is an FRKHS with respect to some family of linear functionals, we are  interested in finding an FRKHS with respect to a specifically given family of linear functionals. This is nontrivial and useful in practical applications.

We begin with presenting several specific examples of ORKHSs and demonstrating that the notion of ORKHSs is not a trivial extension, but a useful development of the concept of RKHSs. Two typical RKHSs are special examples of ORKHSs. Suppose that $\mathcal{X}$ is a set and $\mathcal{Y}$ is a Hilbert space. The space $\mathcal{H}$ is chosen as a Hilbert space of functions from $\mathcal{X}$ to $\mathcal{Y}$ and the family $\mathcal{L}$ of linear operators from $\mathcal{H}$ to $\mathcal{Y}$ is chosen as the set of the point-evaluation operators defined by
$L_x(f):=f(x),\ x\in \mathcal{X}$. Then it is clear that the norm of $f\in\mathcal{H}$ vanishes if and only if $f$, as a function, vanishes everywhere on $\mathcal{X}$, which indicates that the norm of $\mathcal{H}$ is compatible with $\mathcal{L}$. If each point-evaluation operator in $\mathcal{L}$ is continuous on $\mathcal{H}$, the corresponding ORKHS $\mathcal{H}$ reduces to the vector-valued RKHS. In the special case when $\mathcal{Y}=\mathbb{C}$, we choose $\mathcal{H}$ as a Hilbert space of functions on a set $\mathcal{X}$ and family $\mathcal{F}$ of linear functionals as the collection of the point-evaluation functionals, the corresponding FRKHS $\mathcal{H}$ also reduces to the classical scalar-valued RKHS.

We are particularly interested in ORKHSs with respect to a family of linear operators not restricted to point-evaluation operators. The next three examples concerning integral operators show that the notion of ORKHSs is a useful development of the concept of classical RKHSs. To this end, we recall the Bochner integral of a vector-valued function \cite{DU}. Let $(\mathcal{X},\mathcal{M},\mu)$ be a measure space and $\mathcal{Y}$ a Hilbert space. To define the integral of a function from $\mathcal{X}$ to $\mathcal{Y}$ with respect to $\mu$, we first introduce the integral of a simple function with the form
$f:=\sum_{j\in \mathbb{N}_n}\chi_{E_j}\xi_j,$ where $\xi_j,j\in \mathbb{N}_n,$ are distinct elements of $\mathcal{Y}$, $E_j$, $j\in \mathbb{N}_n$, are pairwise disjoint measurable sets and $\chi_{E_j}$ denotes the characteristic function of $E_j$. If $\mu(E_j)$ is finite whenever $\xi_j\neq0$, the integral of $f$ on a measurable set $E$ is defined by
$$
\int_{E}f(x)d\mu(x):=\sum_{j\in \mathbb{N}_n}
\mu(E_j\cap E)\xi_j.
$$
A function $f:\mathcal{X}\rightarrow \mathcal{Y}$ is called measurable if there exists a sequence of simple functions $f_n:\mathcal{X}\rightarrow \mathcal{Y}$ such that
$$
\lim_{n\rightarrow \infty}\|f_n(x)-f(x)\|_\mathcal{Y}=0,
\ \  \mbox{almost everywhere with respect to} \ \mu\ \mbox{on}\ \mathcal{X}.
$$
A measurable function $f:\mathcal{X}\rightarrow \mathcal{Y}$ is called Bochner integrable if
there exists a sequence of simple functions
$f_n:\mathcal{X}\rightarrow \mathcal{Y}$
such that
$$
\lim_{n\rightarrow \infty}\int_{\mathcal{X}}
\|f_n(x)-f(x)\|_\mathcal{Y}d\mu(x)=0.
$$
One can get by the above equation for any measurable set $E$ that the sequence $\int_{E}f_n(x)d\mu(x)$ is a Cauchy sequence in $\mathcal{Y}$. This together with the completeness of $\mathcal{Y}$ allows us to define the integral of $f$ on $E$ by
$$
\int_{E}f(x)d\mu(x):=\lim_{n\rightarrow\infty}
\int_{E}f_n(x)d\mu(x).
$$
It is known that a measurable function $f:\mathcal{X}\rightarrow \mathcal{Y}$ is Bochner integrable if and only if there holds
$$
\int_{\mathcal{X}}\|f(x)\|_{\mathcal{Y}}d\mu(x)<+\infty.
$$
It follows that
$$
\left|\int_{\mathcal{X}}\langle \xi,f(x)
\rangle_{\mathcal{Y}}d\mu(x)\right|
\leq\|\xi\|_{\mathcal{Y}}\int_{\mathcal{X}}
\|f(x)\|_{\mathcal{Y}}d\mu(x),
\ \mbox{for any}\ \xi\in \mathcal{Y},
$$
yielding that the functional
$\int_{\mathcal{X}}\langle\cdot,f(x)\rangle_{\mathcal{Y}} d\mu(x)$ is continuous on $\mathcal{Y}$. Consequently,
we can comprehend the integral $\int_{\mathcal{X}}f(x)d\mu(x)$
as an element of $\mathcal{Y}$ satisfying
$$
\left\langle \xi, \int_{\mathcal{X}}f(x)d\mu(x)
\right\rangle_\mathcal{Y}=\int_{\mathcal{X}}
\langle \xi,f(x)\rangle_\mathcal{Y} d\mu(x),
\ \mbox{for any}\ \xi\in \mathcal{Y}.
$$
For $1\leq p<+\infty$, we denote by $L^p(\mathcal{X},\mathcal{Y},\mu)$ the space
of all the measurable functions
$f:\mathcal{X}\rightarrow \mathcal{Y}$ such that
$$
\|f\|_{L^p(\mathcal{X},\mathcal{Y},\mu)}:=
\left(\int_{\mathcal{X}}\|f(x)\|_{\mathcal{Y}}^pd\mu(x)
\right)^{1/p}<+\infty.
$$
In the case that $\mu$ is the Lebesgue measure on $\mathbb{R}^d$ and $\mathcal{X}$ is a measurable set of $\mathbb{R}^d$, $L^p(\mathcal{X},\mathcal{Y},\mu)$ will be abbreviated as $L^p(\mathcal{X},\mathcal{Y})$. Moreover, if $\mathcal{Y}=\mathbb{C}$, we also abbreviate $L^p(\mathcal{X},\mathcal{Y})$ as $L^p(\mathcal{X})$. Our interest in integral operators is motivated by a practical scenario that the point-evaluation function values can not be measured exactly in practice due to inevitable measurement errors. Alternatively, it may be of practical advantage to use local averages of $f$ as observed information  \cite{A,MP04,S,SZ,V}.

The first two concrete examples of ORKHSs with respect to integral operators is related to
$L^2(\mathcal{X},\mathbb{C}^{n})$, where $\mathcal{X}$ is a measurable set of $\mathbb{R}^d$ and $n$ is a positive integer. Let $\Lambda$ be an index set. We suppose that a family of functions $\{u_{\alpha}:=[u_{\alpha,j}:j\in\mathbb{N}_{n}]: \alpha
\in\Lambda\}\subset L^2(\mathcal{X},\mathbb{C}^{n})$ has the property that for each $j\in\mathbb{N}_{n}$ the linear span of $\{u_{\alpha,j}: \alpha\in\Lambda\}$ is dense in $L^2(\mathcal{X})$. For each $\alpha\in \Lambda$ we define a linear operator $L_{\alpha}$ for
$f\in L^2(\mathcal{X},\mathbb{C}^{n})$ by
\begin{equation}\label{example_integration_functional}
L_{\alpha}(f):=\int_{\mathcal{X}}f(x)\circ
\overline{u_{\alpha}(x)}dx,
\end{equation}
and set $\mathcal{L}:=\{L_{\alpha}: \alpha\in\Lambda\}$. Here, we denote by $u\circ v$ the Hadamard product of two functions $u$ and $v$ from $\mathcal{X}$ to $\mathbb{C}^{n}$. We first show that for each $\alpha\in \Lambda$, $L_{\alpha}$ is well-defined. Note that $f=[f_j:j\in\mathbb{N}_{n}]\in L^2(\mathcal{X},\mathbb{C}^{n})$ if and only if $f_j\in L^2(\mathcal{X})$ for all $j\in\mathbb{N}_{n}$. For each $\alpha\in\Lambda$ and each
$f\in L^2(\mathcal{X},\mathbb{C}^{n})$ there holds
$$
\int_{\mathcal{X}}\|f(x)\circ\overline{u_{\alpha}(x)}
\|_{\mathbb{C}^{n}}dx=\int_{\mathcal{X}}
\left(\sum_{j\in\mathbb{N}_{n}}|f_j(x)u_{\alpha,j}(x)|^2
\right)^{1/2}dx\leq\sum_{j\in\mathbb{N}_{n}}
\int_{\mathcal{X}}|f_j(x)u_{\alpha,j}(x)|dx.
$$
Since $f_j, u_{\alpha,j}\in L^2(\mathcal{X})$ for each $j\in\mathbb{N}_{n}$,  by the H\"{o}lder's inequality,
we have that
$$
\int_{\mathcal{X}}\|f(x)\circ\overline{u_{\alpha}(x)}
\|_{\mathbb{C}^{n}}dx\leq\sum_{j\in\mathbb{N}_{n}}
\|f_j\|_{L^2(\mathcal{X})}\|u_{\alpha,j}
\|_{L^2(\mathcal{X})}<+\infty.
$$
This shows that $f\circ\overline{u}_{\alpha}$ is Bochner integrable, which guarantees the well-definedness of the operator $L_{\alpha}$. We next verify that $L^2(\mathcal{X},\mathbb{C}^{n})$ is an ORKHS with respect to $\mathcal{L}$. By the definition of the Bochner integral, we have for each $\alpha\in\Lambda$ and $f=[f_j:j\in\mathbb{N}_{n}]\in L^2(\mathcal{X},\mathbb{C}^{n})$ that
$$
L_{\alpha}(f)=\left[\int_{\mathcal{X}}f_j(x)
\overline{u_{\alpha,j}(x)}dx:j\in\mathbb{N}_{n}\right].
$$
This together with the density of the linear span of $u_{\alpha,j}\in L^2(\mathcal{X}),\ \alpha\in\Lambda$, leads to the statement that $L_{\alpha}(f)=0$ for all $\alpha\in\Lambda$ if and only if $\|f\|_{L^2(\mathcal{X},\mathbb{C}^{n})}=0$. Hence, the norm of $L^2(\mathcal{X},\mathbb{C}^{n})$ is compatible with $\mathcal{L}$. Moreover, with the help of the H\"{o}lder's inequality, we have for each $\alpha\in\Lambda$ and
$f\in L^2(\mathcal{X},\mathbb{C}^{n})$ that
$$
\|L_{\alpha}(f)\|_{\mathbb{C}^{n}}
\leq\left(\sum_{j\in\mathbb{N}_{n}}
\|f_j\|^2_{L^2(\mathcal{X})}
\|u_{\alpha,j}\|^2_{L^2(\mathcal{X})}
\right)^{1/2}\leq c_{\alpha}
\|f\|_{L^2(\mathcal{X},\mathbb{C}^{n})},
$$
where $c_{\alpha}:=\max\{\|u_{\alpha,j}
\|_{L^2(\mathcal{X})}:j\in\mathbb{N}_{n}\}$.
This ensures that the operator $L_{\alpha}$ defined by (\ref{example_integration_functional}) is continuous on $L^2(\mathcal{X},\mathbb{C}^{n})$. Consequently, by Definition \ref{def_ORKHS} we conclude the following result.

\begin{prop}\label{ORKHS1}
If the family of functions $\{u_{\alpha}:=[u_{\alpha,j}:j\in\mathbb{N}_{n}]: \alpha
\in\Lambda\}\subset L^2(\mathcal{X},\mathbb{C}^{n})$ has the property that for each $j\in\mathbb{N}_{n}$ the linear span of $\{u_{\alpha,j}: \alpha\in\Lambda\}$ is dense in $L^2(\mathcal{X})$, then the space $L^2(\mathcal{X},\mathbb{C}^{n})$ is an
ORKHS with respect to $\mathcal{L}$ defined
as in \eqref{example_integration_functional}
in terms of $u_{\alpha}$.
\end{prop}

Two specific choices of the family of functions $\{u_{\alpha}:\alpha\in\Lambda\}$ lead to two ORKHSs which may be potentially important in practical applications. The first one is motivated by reconstructing a signal from its frequency components, which is one of the main themes in signal analysis. Fourier analysis has been proved to be an efficient approach for providing global energy-frequency distributions of signals. Because of its prowess and simplicity, Fourier analysis is considered as a popular tool in signal analysis and other areas of applications. In this example, we restrict ourselves to the case of periodic functions on $\mathbb{R}$, where the frequency information of such a signal is represented by its Fourier coefficients. Specifically, we consider the Hilbert space $L^2([0,2\pi])$. Set $\Lambda:=\mathbb{Z}$ and choose the functions $u_j,\ j\in\Lambda,$ as the Fourier basis functions. That is,
$$
u_j:=\frac{1}{\sqrt{2\pi}}e^{ij(\cdot)},\ \ j\in\Lambda.
$$
Associated with these functions, we introduce a family of linear functionals on $L^2([0,2\pi])$ as
\begin{equation}\label{integration_functional_Fourier}
L_j(f):=\int_{0}^{2\pi}f(x)\overline{u_j(x)}dx,\ f\in L^2([0,2\pi]),\ j\in\Lambda,
\end{equation}
and set $\mathcal{L}:=\{L_j:j\in\Lambda\}$. It follows from Proposition \ref{ORKHS1} and the density of
the Fourier basis functions that the space $L^2([0,2\pi])$ is an FRKHS with respect to $\mathcal{L}$.

Wavelet analysis is an important mathematical tool in various applications, such as signal processing, communication and numerical analysis. Various wavelet functions have been constructed for practical purpose in the literature \cite{CMX,D,Mallat,M,MX}.
Particularly, wavelet analysis provides an alternative representation of signals. The wavelet transform contains the complete information of a signal in both time and frequency domains at the same time. Our next example comes from the problem of recovering a signal from finite wavelet coefficients, which can be seen as the linear functional values of the signal. Suppose that $\psi\in L^2(\mathbb{R}^d)$ is a wavelet function such that the sequence $\psi_{j,k},\ j\in\mathbb{Z},\ k\in\mathbb{Z}^d$, defined by
$$
\psi_{j,k}(x):=2^{\frac{jd}{2}}\psi(2^jx-k), \ \ x\in\mathbb{R}^d, \ \ j\in\mathbb{Z},\ k\in\mathbb{Z}^d,
$$
constitutes an orthonormal basis for $L^2(\mathbb{R}^d)$.
We consider the space $L^2(\mathbb{R}^d)$. To introduce a family of linear functionals, we let $\Lambda:=\mathbb{Z}\times\mathbb{Z}^{d}$ and for each $(j,k)\in\Lambda$ set $u_{(j,k)}:=\psi_{j,k}$. Accordingly, the linear functionals are defined by
\begin{equation}\label{integration_functional_wavelet}
L_{(j,k)}(f):=\int_{\mathbb{R}^d}f(x)\overline{u_{(j,k)}(x)}
dx,\ f\in L^2(\mathbb{R}^d),\ (j,k)\in\Lambda.
\end{equation}
By Proposition \ref{ORKHS1}, we obtain directly that the space $L^2(\mathbb{R}^d)$ is an FRKHS with respect to the family $\mathcal{L}:=\{L_{(j,k)}:\ (j,k)\in\Lambda\}$ of continuous linear functionals.

As a third concrete example of ORHKSs with respect to integral operators, we consider the Paley-Wiener space
of functions from $\mathbb{R}^d$ to $\mathbb{C}^{n}$. For a set of positive numbers $\Delta:=\{\delta_j:j \in\mathbb{N}_d\}$, we denote by $\mathcal{B}_{\Delta}(\mathbb{R}^d,\mathbb{C}^{n})$ the closed subspaces of functions in $L^2(\mathbb{R}^d,\mathbb{C}^{n})$ which are bandlimited to the region
$$
I_{\Delta}:=\prod_{j\in\mathbb{N}_d}[-\delta_j,\delta_j].
$$
That is,
$$
\mathcal{B}_{\Delta}(\mathbb{R}^d,\mathbb{C}^{n}):=
\left\{f\in L^2(\mathbb{R}^d,\mathbb{C}^{n}): \ {\rm supp}\hat f\subseteq I_{\Delta}\right\},
$$
with the inner product
$$
\langle f,g\rangle_{\mathcal{B}_{\Delta}
(\mathbb{R}^d,\mathbb{C}^{n})}:=
\sum_{j\in\mathbb{N}_{n}}\langle f_j,g_j\rangle_{L^{2}(\mathbb{R}^d)}.
$$
Here, we define the Fourier transform $\hat{h}$ and the inverse Fourier transform $\check{h}$ of $h\in L^1(\mathbb{R}^d,\mathbb{C}^n)$, respectively, by
$$
\hat h(\omega):=\int_{\mathbb{R}^d}
h(t)e^{-i(t,\omega)}dt,\ \omega\in\mathbb{R}^d,
$$
and
$$
\check{h}(\omega):=\frac{1}{(2\pi)^d}\int_{\mathbb{R}^d}
h(t)e^{i(t,\omega)}dt,\ \omega\in\mathbb{R}^d,
$$
with $(\cdot,\cdot)$ being the standard inner product on $\mathbb{R}^d$. By standard approximation process, the Fourier transform and its inverse can be extended to the functions in $L^2(\mathbb{R}^d,\mathbb{C}^n)$. It is well-known that the space $\mathcal{B}_{\Delta}(\mathbb{R}^d,\mathbb{C}^{n})$ is a vector-valued RKHS. In other words, it is an ORKHS with respect to the family of the point-evaluation operators. We shall show that it is also an ORKHS with respect to a family of the integral operators to be defined next. Specifically, we suppose that $u\in L^2(\mathbb{R}^d)$ satisfies the property that
$\hat{u}(\omega)\neq0, \ {\rm a.e.}\ \omega\in I_{\Delta}$. For each $x\in\mathbb{R}^d$, we set $u_x:=u(\cdot-x)$ and introduce the operator
\begin{equation}\label{example_integration_functional1}
L_x(f):=\int_{\mathbb{R}^d}f(t)\overline{u_x(t)}dt
\end{equation}
and define $\mathcal{L}:=\{L_x: x\in \Lambda\}$, with $\Lambda:=\mathbb{R}^d$. We first show that the norm of  $\mathcal{B}_{\Delta}(\mathbb{R}^d,\mathbb{C}^{n})$ is compatible with $\mathcal{L}$. It suffices to verify for each $f\in\mathcal{B}_{\Delta}
(\mathbb{R}^d,\mathbb{C}^{n})$ that if
$L_x(f)=0$ for all $x\in\mathbb{R}^d$, then $\|f\|_{L^2(\mathbb{R}^d, \mathbb{C}^{n})}=0$. By the Plancherel identity we have for any $x\in \mathbb{R}^d$ and any $\xi\in \mathbb{C}^n$ that
\begin{equation*}
\langle L_x(f),\xi\rangle_{\mathbb{C}^{n}}=
\int_{\mathbb{R}^d}\langle f(t),\xi
\rangle_{\mathbb{C}^{n}}\overline{u_x(t)}dt
=\int_{I_{\Delta}}\sum_{j\in\mathbb{N}_{n}}
\overline{\xi_j}\hat{f}_j(\omega)
\overline{\hat{u}(\omega)}e^{i(x,\omega)}d\omega.
\end{equation*}
The density $\overline{\span}\{e^{i(x,\cdot)}:
x\in\mathbb{R}^d\}=L^2(I_{\Delta})$
ensures that the condition $L_x(f)=0$ for all $x\in\mathbb{R}^d$ is equivalent to that
\begin{equation}\label{new_sum}
\sum_{j\in\mathbb{N}_{n}}\overline{\xi_j}
\hat{f}_j(\omega)\overline{\hat{u}(\omega)}=0,\  \mbox{a.e.}  \ \omega\in I_{\Delta},
\ \mbox{for all}\ \ \xi\in\mathbb{C}^{n}.
\end{equation}
By the fact that $\hat{u}(\omega)\neq0, \ {\rm a.e.}\  \omega\in I_{\Delta}$, \eqref{new_sum} holds if and only if $\|f_j\|_{L^2(\mathbb{R}^d)}=0,j\in\mathbb{N}_{n}$, which is equivalent to
$\|f\|_{L^2(\mathbb{R}^d, \mathbb{C}^{n})}=0$.
Thus, we have proved that the norm of  $\mathcal{B}_{\Delta}(\mathbb{R}^d,\mathbb{C}^{n})$ is compatible with $\mathcal{L}$. Moreover, applying the H\"{o}lder's inequality to the definition of operator $L_x$ yields that for each $x\in\mathbb{R}^d$ the operator $L_x$ is continuous on $\mathcal{B}_{\Delta}
(\mathbb{R}^d,\mathbb{C}^{n})$. According to Definition \ref{def_ORKHS}, we obtain the following result regarding space $\mathcal{B}_{\Delta}(\mathbb{R}^d,\mathbb{C}^{n})$.

\begin{prop}\label{Paley}
If $u\in L^2(\mathbb{R}^d)$ satisfies  $\hat{u}(\omega)\neq0, \ {\rm a.e.}\  \omega\in I_{\Delta}$, then $\mathcal{B}_{\Delta}
(\mathbb{R}^d, \mathbb{C}^{n})$ is an ORKHS with respect to $\mathcal{L}$ defined as in \eqref{example_integration_functional1}.
\end{prop}

Finally, we describe a finite-dimensional ORKHS. Let $\mathcal{H}$ be a Hilbert space of dimension $n$ and   $\mathcal{L}:=\{L_{\alpha}:\alpha\in\Lambda\}$ a family of linear operators from $\mathcal{H}$ to a Hilbert space $\mathcal{Y}$. We suppose that there exists a finite set  $\{\alpha_j:j\in\mathbb{N}_n\}$ of $\Lambda$ such that $L_{\alpha_j},\ j\in\mathbb{N}_n,$ are linearly independent. That is, if for $\{\xi_j:j\in\mathbb{N}_n\}\subseteq\mathcal{Y}$ there holds
$$
\sum_{j\in\mathbb{N}_n}\langle L_{\alpha_j}(f),\xi_j\rangle_{\mathcal{Y}}=0,\ \mbox{for all} \ f\in\mathcal{H},
$$
then $\xi_j=0, j\in\mathbb{N}_n.$ It is known that any linear operator from a finite-dimensional normed space to a normed space is continuous. Hence, the continuity of the linear operators in $\mathcal{L}$ is clear. It suffices to verify that the norm of $\mathcal{H}$ is compatible with $\mathcal{L}$. Let $\xi$ be a nonzero element in $\mathcal{Y}$. As $L_{\alpha_j}, j\in\mathbb{N}_n,$ are continuous on $\mathcal{H}$, there exists a finite set $\{\phi_j: j\in\mathbb{N}_n\}\subseteq\mathcal{H}$ such that
\begin{equation}\label{finitebasis}
\langle L_{\alpha_j}(f),\xi\rangle_{\mathcal{Y}}=
\langle f,\phi_j\rangle_{\mathcal{H}}, \ \mbox{for all}\ f\in\mathcal{H}.
\end{equation}
We shall show that $\phi_j, j\in\mathbb{N}_n$, are linear independent. Assume that there holds for some scalars $c_j,j\in\mathbb{N}_n$, that $\sum_{ j\in\mathbb{N}_n}
c_j\phi_j=0$. This together with (\ref{finitebasis}) leads to
$$
\sum_{j\in\mathbb{N}_n}\langle L_{\alpha_j}(f), \overline{c_j}\xi\rangle_{\mathcal{Y}}=\left< f,\sum_{ j\in\mathbb{N}_n}c_j\phi_j\right>_{\mathcal{H}}=0, \ \mbox{for all}\ f\in\mathcal{H}.
$$
By the linear independency of $L_{\alpha_j}, j\in\mathbb{N}_n$, we have that $c_j=0$ for all $j\in\mathbb{N}_n.$ That is, $\phi_j, j\in\mathbb{N}_n$, are linear independent, which implies $\span\{\phi_j: j\in\mathbb{N}_n\}=\mathcal{H}$. Suppose that $f\in\mathcal{H}$ satisfies $L_{\alpha_j}(f)=0$ for all $j\in\mathbb{N}_n$. We then get by (\ref{finitebasis}) that $\langle f,\phi_j\rangle_{\mathcal{H}}=0, j\in\mathbb{N}_n$, which yields $\|f\|_{\mathcal{H}}=0$.
Consequently, we have proved that the norm of $\mathcal{H}$ is compatible with $\mathcal{L}$. By Definition \ref{def_ORKHS}, we conclude that $\mathcal{H}$ is an ORKHS with respect to $\mathcal{L}$, which is stated in the following proposition.
\begin{prop}\label{Finitedimension}
Let $\mathcal{H}$ be a Hilbert space of dimension $n$ and   $\mathcal{L}:=\{L_{\alpha}:\alpha\in\Lambda\}$ a family of linear operators from $\mathcal{H}$ to a Hilbert space $\mathcal{Y}$. If there exist $n$ linearly independent linear operators in $\mathcal{L}$, then $\mathcal{H}$ is an ORKHS with respect to $\mathcal{L}$.
\end{prop}

Observing from the examples presented above, we
point out that the spaces $L^2([0,2\pi])$ and $L^2(\mathbb{R}^d)$ are FRKHSs but not classical
RKHSs as they consist of equivalent classes of functions
with respect to the Lebesgue measure. However, the space $\mathcal{B}_{\Delta}(\mathbb{R}^d,\mathbb{C}^n)$
admits two different families of linear operators
from itself to $\mathbb{C}^n$, one being the
standard point-evaluation operators and the other
being the integral operators  having the form of (\ref{example_integration_functional1}). If the
finite-dimensional ORKHS $\mathcal{H}$ consists of functions from  an input set $\mathcal{X}$ to $\mathcal{Y}$, then all the point-evaluation operators are continuous on $\mathcal{H}$. Hence, the finite dimensional space $\mathcal{H}$ also admits the family of point-evaluation operators, besides the family $\mathcal{L}$. The ORKHSs having this property are especially desirable, which we shall study in the next two sections.

We now turn to investigating the properties of general ORKHSs. A core in the theory of RKHSs lies in the celebrated result that there is a bijective correspondence between the reproducing kernel and the corresponding RKHS. We next establish a similar bijection in the framework of ORKHSs.

We first identify the operator reproducing kernel to each ORKHS. Recall that a vector-valued RKHS admits a reproducing kernel and the functional values $\langle f(x), \xi\rangle_{\mathcal{Y}}$, $x\in \mathcal{X}, \xi\in\mathcal{Y}$, of a function $f$ in the vector-valued RKHS can be reproduced via its inner product with the kernel. In a manner similar to the vector-valued RKHS, we show that for a general ORKHS there exists a kernel which can be used to reproduce the functional values $\langle L_{\alpha}(f),\xi\rangle_{\mathcal{Y}}$, $\alpha\in \Lambda,\xi\in\mathcal{Y}$. For two Hilbert spaces $\mathcal{H}_1,\mathcal{H}_2$, we denote by $\mathcal{B}(\mathcal{H}_1,\mathcal{H}_2)$ the set of bounded linear operators from $\mathcal{H}_1$ to $\mathcal{H}_2$. For each
$L\in \mathcal{B}(\mathcal{H}_1,\mathcal{H}_2)$, we let $\|L\|_{\mathcal{B}(\mathcal{H}_1, \mathcal{H}_2)}$
denote the operator norm of $L$ in $\mathcal{B}(\mathcal{H}_1,\mathcal{H}_2)$ and
$L^{*}$ denote the adjoint of $L$.

\begin{thm}\label{property}
Suppose that $\mathcal{H}$ is a Hilbert space and $\mathcal{L}:=\{L_{\alpha}: \alpha\in \Lambda\}$ is a family of linear operators from $\mathcal{H}$ to a Hilbert space $\mathcal{Y}$. If $\mathcal{H}$ is an ORKHS with respect to $\mathcal{L}$, then there exists a unique operator $K:\Lambda\rightarrow
\mathcal{B}(\mathcal{Y},\mathcal{H})$ such that the following statements hold.

\indent (1) For each $\alpha\in \Lambda$ and each $\xi\in\mathcal{Y}$, $K(\alpha)\xi\in \mathcal{H}$ and
\begin{equation}\label{reproducing_property}
\langle L_{\alpha}(f),\xi\rangle_{\mathcal{Y}}
=\langle f,K(\alpha)\xi\rangle_\mathcal{H},
\ \ \mbox{for all}\ \ f\in \mathcal{H}.
\end{equation}

\indent (2) The linear span  $\mathcal{S}_{K}:=\span\{K(\alpha)\xi:
\alpha\in\Lambda,\xi\in\mathcal{Y}\}$
is dense in $\mathcal{H}$.

\indent (3) For each $n\in\mathbb{N}$, each pair of
finite sets $\Lambda_n:=\{\alpha_j:j\in\mathbb{N}_n\}
\subseteq\Lambda$ and $\mathcal{Y}_n:=\{\xi_j:
j\in\mathbb{N}_n\}\subseteq\mathcal{Y}$ there holds
$$
\sum_{j\in\mathbb{N}_n}\sum_{k\in\mathbb{N}_n}
\langle L_{\alpha_k}(K(\alpha_j)\xi_j),\xi_k
\rangle_{\mathcal{Y}}\geq0.
$$
\end{thm}
\begin{proof}
We construct the operator
$K:\Lambda\rightarrow\mathcal{B}(\mathcal{Y},\mathcal{H})$
which satisfies (1)-(3). For each $\alpha\in \Lambda$, since $L_{\alpha}\in\mathcal{B}(\mathcal{H},\mathcal{Y})$, the adjoint $L_{\alpha}^{*}$ is the unique operator in $\mathcal{B}(\mathcal{Y},\mathcal{H})$ satisfying
$$
\langle L_{\alpha}(f),\xi\rangle_{\mathcal{Y}}
=\langle f,L_{\alpha}^{*}\xi\rangle_\mathcal{H},\ \mbox{for all} \ f\in \mathcal{H}\ \mbox{and}\  \xi\in\mathcal{Y}.
$$
We define the operator $K$ from $\Lambda$ to $\mathcal{B}(\mathcal{Y},\mathcal{H})$ by
$K(\alpha):=L_{\alpha}^{*}$, $\alpha\in\Lambda$. Then, clearly, $K$ satisfies statement (1).

To prove statement (2), we suppose that $f\in\mathcal{H}$ satisfies $\langle f,K(\alpha)\xi\rangle_{\mathcal{H}}=0$,  for all $\alpha\in \Lambda, \xi\in\mathcal{Y}$. This together with (\ref{reproducing_property})
ensures that $L_{\alpha}(f)=0$ for all
$\alpha\in \Lambda$. By the hypothesis of this theorem, $\mathcal{H}$ is an ORKHS with respect to $\mathcal{L}$. It follows that the norm of $\mathcal{H}$ is compatible with  $\mathcal{L}$. Thus, we have that $f=0$. This implies that $\overline{\mathcal{S}_{K}}=\mathcal{H}.$

We next show statement (3). For each $n\in\mathbb{N}$ and for $\{\alpha_j:j\in\mathbb{N}_n\}\subseteq\Lambda$ and $\{\xi_j:j\in\mathbb{N}_n\}\subseteq\mathcal{Y}$, we let $f_n:=\sum_{j\in\mathbb{N}_n} K(\alpha_j)\xi_j$. We then observe from equation (\ref{reproducing_property}) that
$$
\sum_{j\in\mathbb{N}_n}\sum_{k\in\mathbb{N}_n}
\langle L_{\alpha_k}(K(\alpha_j)\xi_j),
\xi_k\rangle_{\mathcal{Y}}
=\sum_{j\in\mathbb{N}_n}\sum_{k\in\mathbb{N}_n}
\langle K(\alpha_j)\xi_j, K(\alpha_k)\xi_k
\rangle_{\mathcal{H}}
=\left<f_n,f_n\right>_{\mathcal{H}}\geq0,
$$
proving the desired result.
\end{proof}

It is known from equation (\ref{reproducing_property}) that one can use the operator $K$ to reproduce the functional values $\langle L_{\alpha}(f),
\xi\rangle_{\mathcal{Y}},\ \alpha\in\Lambda,
\xi\in\mathcal{Y}$. Hence, we shall call equation (\ref{reproducing_property}) the reproducing property and $K$ the operator reproducing kernel for ORKHS $\mathcal{H}$. We remark on the term ``operator reproducing kernel". The reproducing kernel $\mathcal{K}$ for a vector-valued RKHS, consisting of functions from $\mathcal{X}$ to a Hilbert space $\mathcal{Y}$, is called in the literature an operator-valued reproducing kernel. It is because for each $(x,y)\in\mathcal{X}\times \mathcal{X}$, $\mathcal{K}(x,y)$ is an operator from $\mathcal{Y}$ to itself. However, the term ``operator" in the definition of an operator reproducing kernel for an ORKHS $\mathcal{H}$ is used to demonstrate the family of operators, with respect to which $\mathcal{H}$ is an ORKHS.

As a special case, the conclusion in Theorem \ref{property} in the scalar case when $\mathcal{Y}=\mathbb{C}$ can be stated as the following corollary.

\begin{cor}\label{FRKHS_property}
Suppose that $\mathcal{H}$ is a Hilbert space and $\mathcal{F}:=\{L_{\alpha}: \alpha\in \Lambda\}$ is a family of linear functionals on $\mathcal{H}$. If $\mathcal{H}$ is an FRKHS with respect to $\mathcal{F}$, then there exists a unique operator $K:\Lambda\rightarrow \mathcal{H}$ such that the linear span  $\mathcal{S}_{K}:=\span\{K(\alpha):\alpha\in \Lambda\}$
is dense in $\mathcal{H}$ and for each $\alpha\in \Lambda$, there holds the reproducing property
\begin{equation*}\label{FRKHS_reproducing_property}
L_{\alpha}(f)=\langle f,K(\alpha)\rangle_\mathcal{H},
\ \ \mbox{for all}\ \ f\in \mathcal{H}.
\end{equation*}
Moreover, for each $n\in\mathbb{N}$ and each finite set $\Lambda_n:=\{\alpha_j:j\in \mathbb{N}_n\}
\subseteq\Lambda$ the matrix
${\bf F}_{\Lambda_n}:=[L_{\alpha_k}(K(\alpha_j)):
j,k\in\mathbb{N}_n]$ is hermitian and positive semi-definite.
\end{cor}

We present the operator reproducing kernels for the specific ORKHSs discussed earlier. We first consider the ORKHS identical to the vector-valued RKHS. It is known \cite{MP05} that a vector-valued RKHS $\mathcal{H}$ consisting of functions from $\mathcal{X}$ to $\mathcal{Y}$ has a reproducing kernel $\mathcal{K}$ from $\mathcal{X}\times \mathcal{X}$ to $\mathcal{B}(\mathcal{Y},\mathcal{Y})$ such that
$$
\langle f(x),\xi\rangle_{\mathcal{Y}}
=\langle f,\mathcal{K}(x,\cdot)\xi\rangle_{\mathcal{H}},
\ f\in\mathcal{H},\ x\in \mathcal{X},\ \xi\in\mathcal{Y}.
$$
Then the operator reproducing kernel $K$ may be represented via $\mathcal{K}$ by
$$
K(x)\xi:=\mathcal{K}(x,\cdot)\xi,\ x\in \mathcal{X},
\ \xi\in\mathcal{Y}.
$$
It follows from
$$
\|K(x)\xi\|_{\mathcal{H}}
=\langle \mathcal{K}(x,\cdot)\xi,\mathcal{K}(x,\cdot)
\xi\rangle_{\mathcal{H}}^{1/2}
=\langle \mathcal{K}(x,x)\xi,\xi
\rangle_{\mathcal{Y}}^{1/2}
\leq\|\mathcal{K}(x,x)\|^{1/2}_{\mathcal{B}
(\mathcal{Y},\mathcal{Y})}\|\xi\|_{\mathcal{Y}}
$$
that $K(x)\in \mathcal{B}(\mathcal{Y},\mathcal{H})$. Moreover, when $\mathcal{Y}=\mathbb{C}$, the functional reproducing kernel can be represented via the classical reproducing kernel $\mathcal{K}:\mathcal{X}
\times \mathcal{X}\rightarrow \mathbb{C}$ by $K(x):=\mathcal{K}(x,\cdot),\ x\in \mathcal{X}$. In the remaining part of this paper, we shall use the conventional notation of the reproducing kernels for the classical scalar-valued or vector-valued RKHSs.

The kernel for the ORKHS $L^2(\mathcal{X},\mathbb{C}^{n})$ with respect to the family $\mathcal{L}$ of the integral operators can be represented by the functions
$u_{\alpha},\ \alpha\in\Lambda,$ appearing in (\ref{example_integration_functional}). To see this, we let $\mathcal{L}$ be the family of the integral operators defined as in (\ref{example_integration_functional})
and recall that if for each $j\in\mathbb{N}_{n}$ the linear span of $u_{\alpha,j}, \alpha\in\Lambda$, is dense in $L^2(\mathcal{X})$, then  $L^2(\mathcal{X},\mathbb{C}^{n})$ is an ORKHS with respect to $\mathcal{L}$. Observing from
$$
\langle L_{\alpha}(f),\xi\rangle_{\mathbb{C}^{n}}
=\int_{\mathcal{X}}\langle f(x),u_{\alpha}(x)
\circ\xi\rangle_{\mathbb{C}^{n}}dx
=\langle f,u_{\alpha}\circ\xi
\rangle_{L^2(\mathcal{X},\mathbb{C}^{n})},
$$
we get that the operator reproducing kernel for the ORKHS $L^2(\mathcal{X},\mathbb{C}^{n})$ has the form
$$
K(\alpha)\xi:=u_{\alpha}\circ\xi,
\ \alpha\in\Lambda,\ \xi\in\mathbb{C}^{n}.
$$
Particularly, the functional reproducing kernel for $L^2([0,2\pi])$ with respect to the family
$\mathcal{L}$ of the linear functionals defined by (\ref{integration_functional_Fourier}) is determined by
$$
K(j):=\frac{1}{\sqrt{2\pi}}e^{ij(\cdot)},
\ j\in\mathbb{Z}.
$$
Similarly, the functional reproducing kernel for $L^2(\mathbb{R}^d)$ with respect to the family
$\mathcal{L}$ of the linear functionals defined by (\ref{integration_functional_wavelet}) coincides with the wavelet basis functions. That is,
$$
K(j,k):=2^{\frac{jd}{2}}\psi(2^j(\cdot)-k),
\ j\in\mathbb{Z},k\in\mathbb{Z}^d.
$$

Another example of ORKHSs with respect to the family of the integral operators is the Paley-Wiener space of functions from $\mathbb{R}^d$ to $\mathbb{C}^{n}$. Recall that $\mathcal{B}_{\Delta}(\mathbb{R}^d,\mathbb{C}^{n})$
is a vector-valued RKHS with the translation-invariant reproducing kernel
\begin{equation}\label{sinc}
\mathcal{K}(x,y):={\rm diag}(\widetilde{\mathcal{K}}(x,y):
l\in\mathbb{N}_{n}),
\end{equation}
where
$$
\widetilde{\mathcal{K}}(x,y):=\prod_{j\in\mathbb{N}_d}
\frac{\sin \delta_j(x_j-y_j)}{\pi(x_j-y_j)},\ x,y\in \mathbb{R}^d.
$$
Let $\mathcal{L}$ be the family of the integral operators defined as in (\ref{example_integration_functional1}). We also verified earlier in this section that $\mathcal{B}_{\Delta}(\mathbb{R}^d,\mathbb{C}^{n})$ is an ORKHS with respect to such a set $\mathcal{L}$. Let $\mathscr{P}_{\mathcal{B}_{\Delta}}$ denote the orthogonal projection from $L^2(\mathbb{R}^d,\mathbb{C}^{n})$ onto $\mathcal{B}_{\Delta}(\mathbb{R}^d,\mathbb{C}^{n})$. Then, there holds for each $x\in \mathbb{R}^d$ and each $f\in\mathcal{B}_{\Delta}(\mathbb{R}^d,\mathbb{C}^{n})$
\begin{equation*}
\langle L_x(f),\xi\rangle_{\mathbb{C}^n}=\langle f,u(\cdot-x)\xi\rangle_{L^2(\mathbb{R}^d,\mathbb{C}^{n})}
=\langle f,\mathscr{P}_{\mathcal{B}_{\Delta}}
(u(\cdot-x)\xi)\rangle_{L^2(\mathbb{R}^d,\mathbb{C}^{n})}.
\end{equation*}
This leads to $K(x)\xi=\mathscr{P}_{\mathcal{B}_{\Delta}}
(u(\cdot-x)\xi)$, $x\in\mathbb{R}^d,\xi\in\mathbb{C}^{n}$.

Our last example concerns the finite-dimensional ORKHS. If there exist $n$ linearly independent elements in the family $\mathcal{L}$ of linear operators from $\mathcal{H}$ to a Hilbert space $\mathcal{Y}$, then the finite-dimensional Hilbert space $\mathcal{H}$ was proved to be an ORKHS with respect to such a family $\mathcal{L}$. To present the operator reproducing kernel, we suppose that $\{\phi_j: j\in\mathbb{N}_n\}$ is a linearly independent sequence in $\mathcal{H}$. For each $\alpha\in\Lambda$ and each $\xi\in\mathcal{Y}$, we can represent $K(\alpha)\xi\in\mathcal{H}$ as
$$
K(\alpha)\xi=\sum_{j\in\mathbb{N}_n}c_j\phi_j,\
\ \mbox{for some sequence}\ \  \{c_j:j\in\mathbb{N}_n\}\subseteq\mathbb{C}.
$$
It follows from the reproducing property (\ref{reproducing_property}) that the sequence $\{c_j:j\in\mathbb{N}_n\}$ satisfies the following system
$$
\langle \xi, L_{\alpha}(\phi_k)\rangle_{\mathcal{Y}}
=\sum_{j\in\mathbb{N}_n}c_j
\langle\phi_j,\phi_k\rangle_\mathcal{H},\ k\in\mathbb{N}_n.
$$
We denote by $\mathbf{A}$ the coefficients matrix, that is, $\mathbf{A}:=[\langle\phi_k,\phi_j\rangle_\mathcal{H}:
j,k\in\mathbb{N}_n]$. Due to the linear independence of $\phi_j,\ j\in\mathbb{N}_n,$ matrix $\mathbf{A}$ is positive definite. Denoting by $\mathbf{B}:=[b_{j,k}:
j,k\in\mathbb{N}_n]$ the inverse matrix of $\mathbf{A}$, we obtain the operator reproducing kernel as follows:
$$
K(\alpha)\xi=\sum_{j\in\mathbb{N}_n}
\sum_{k\in\mathbb{N}_n}b_{j,k}\langle \xi, L_{\alpha}(\phi_k)\rangle_{\mathcal{Y}}\phi_j,
\ \alpha\in\Lambda, \xi\in\mathcal{Y}.
$$

Theorem \ref{property} indicates that an ORKHS has a unique operator reproducing kernel. We now turn to the reverse problem. That is, associated with a given operator reproducing kernel, there exists a unique ORKHS. To show this, we introduce an alternative definition for the operator reproducing kernel according to Property (3) in Theorem \ref{property} and prove its equivalence to the original definition. For two vector space $V_1,V_2$, we denote by $\mathcal{L}(V_1,V_2)$ the set of linear operators from $V_1$ to $V_2$.

\begin{defn}\label{Alter-Defn}
Suppose that $\Lambda$ is a set, $\mathcal{Z}$ is a vector space and $\mathcal{Y}$ is a Hilbert space. Let $K$ be an operator from $\Lambda$ to $\mathcal{L}(\mathcal{Y},\mathcal{Z})$ and $\mathcal{L}:=\{L_{\alpha}:\alpha\in \Lambda\}$ be a family of linear operators from the linear span $\mathcal{S}_{K}:=\span\{K(\alpha)\xi:\alpha\in \Lambda,\xi\in\mathcal{Y}\}\subseteq\mathcal{Z}$ to $\mathcal{Y}$. We call $K$ an operator reproducing kernel with respect to $\mathcal{L}$ if for each
$n\in\mathbb{N}$ and each pair of finite sets $\{\alpha_j:j\in \mathbb{N}_n\}\subseteq\Lambda$, $\{\xi_j:j\in \mathbb{N}_n\}\subseteq\mathcal{Y}$ there holds
\begin{equation}\label{Operator-Kernel}
\sum_{j\in\mathbb{N}_n}\sum_{k\in\mathbb{N}_n}
\langle L_{\alpha_k}(K(\alpha_j)\xi_j),
\xi_k\rangle_{\mathcal{Y}}\geq0.
\end{equation}
\end{defn}

In the scalar case that $\mathcal{Y}=\mathbb{C}$, the operator reproducing kernel $K:\Lambda\rightarrow \mathcal{Z}$ becomes the functional reproducing kernel with respect to a family
$\mathcal{F}:=\{L_{\alpha}:\alpha\in \Lambda\}$
of linear functionals on
$\mathcal{S}_{K}:=\span\{K(\alpha):\alpha\in\Lambda\}$. Moreover, in this case, condition (\ref{Operator-Kernel}) reduces to that for each $n\in\mathbb{N}$ and each finite set $\Lambda_n:=\{\alpha_j:j\in \mathbb{N}_n\}
\subseteq\Lambda$, the matrix $\mathbf{F}_{\Lambda_n}:=
[L_{\alpha_k}(K(\alpha_j)):j,k\in\mathbb{N}_n]$ is hermitian and positive semi-definite.

The next theorem shows that corresponding to each operator reproducing kernel, there associates an ORKHS.

\begin{thm}
Suppose that $K$ is an operator from $\Lambda$ to $\mathcal{L}(\mathcal{Y},\mathcal{Z})$.
If $K$ is an operator reproducing kernel with respect to a family $\mathcal{L}:=\{L_{\alpha}:\ \alpha\in \Lambda\}$ of linear operators from the linear span $\mathcal{S}_{K}$ to $\mathcal{Y}$, then there exists a unique ORKHS $\mathcal{H}$ with respect to $\mathcal{L}$ such that $\overline{\mathcal{S}_{K}}=\mathcal{H}$ and for each $f\in\mathcal{H}$ and each
$\alpha\in \Lambda,\xi\in\mathcal{Y}$,
$\langle L_{\alpha}(f),\xi\rangle_{\mathcal{Y}}
=\langle f,K(\alpha)\xi\rangle_\mathcal{H}$.
\end{thm}
\begin{proof}
We specifically construct the ORKHS and verify that it has the desired properties. We first construct a vector space
based upon the given operator reproducing kernel. To this end, we define equivalent classes in the linear span $\mathcal{S}_{K}$. Specifically, two functions $f_1,f_2\in\mathcal{S}_{K}$ are said to be equivalent provided that $L_{\alpha}(f_1)=L_{\alpha}(f_2)$ holds for all $\alpha\in\Lambda$. Accordingly, they induce a partition of $\mathcal{S}_{K}$ by which $\mathcal{S}_{K}$ is partitioned into a collection of disjoint equivalent classes. For a function $f\in\mathcal{S}_{K}$, we denote by $\tilde{f}$ the class equivalent to $f$ and introduce the vector space $\mathcal{H}_0:=\{\tilde{f}:
f\in\mathcal{S}_{K}\}$.

We then equip the vector space $\mathcal{H}_0$ with an inner product so that it becomes an inner product space. For this purpose, we define a bilinear mapping $\langle \cdot,\cdot\rangle_{\mathcal{H}_0}:\mathcal{H}_0\times \mathcal{H}_0\rightarrow \mathbb{C}$ for $f:=\sum_{j\in\mathbb{N}_n}K(\alpha_j)\xi_j$ and $g:=\sum_{k\in\mathbb{N}_m}K(\beta_k)\eta_k$ by
\begin{equation}\label{inner-product}
\langle \tilde{f},\tilde{g}\rangle_{\mathcal{H}_0}
:=\sum_{j\in\mathbb{N}_n}\sum_{k\in\mathbb{N}_m}\langle
L_{\beta_k}(K(\alpha_j)\xi_j),\eta_k\rangle_{\mathcal{Y}}.
\end{equation}
From the representation
$$
\langle \tilde{f},\tilde{g}\rangle_{\mathcal{H}_0}
=\sum_{k\in\mathbb{N}_m}\left\langle L_{\beta_k}
\left(\sum_{j\in\mathbb{N}_n}K(\alpha_j)\xi_j\right),
\eta_k\right\rangle_{\mathcal{Y}}
=\sum_{k\in\mathbb{N}_m}\langle L_{\beta_k}(f),
\eta_k\rangle_{\mathcal{Y}},
$$
we observe that the bilinear mapping $\langle \cdot,\cdot\rangle_{\mathcal{H}_0}$ is independent of the choice of the equivalent class representation of elements in $\mathcal{H}_0$ for the first variable. Likewise, we may show that the bilinear mapping
$\langle \cdot,\cdot\rangle_{\mathcal{H}_0}$ is also independent of the choice of the equivalent class representation of elements in $\mathcal{H}_0$ for the second variable. Thus, we conclude that the bilinear mapping (\ref{inner-product}) is well-defined.

We next verify that the bilinear mapping
$\langle \cdot,\cdot\rangle_{\mathcal{H}_0}$ is indeed an inner product on $\mathcal{H}_0$. It follows that for any $f:=\sum_{j\in\mathbb{N}_n}K(\alpha_j)\xi_j$ and $g:=\sum_{k\in\mathbb{N}_m}K(\beta_k)\eta_k$, there holds
\begin{equation*}
\langle \tilde{f},\tilde{g}\rangle_{\mathcal{H}_0}
=\sum_{j\in\mathbb{N}_n}\sum_{k\in\mathbb{N}_m}
\langle L_{\beta_k}(K(\alpha_j)\xi_j),
\eta_k\rangle_{\mathcal{Y}}
=\overline{\sum_{j\in\mathbb{N}_n}\sum_{k\in\mathbb{N}_m}
\langle L_{\alpha_j}(K(\beta_k)\eta_k),\xi_j
\rangle_{\mathcal{Y}}}=\overline{\langle\tilde{g},
\tilde{f}\rangle_{\mathcal{H}_0}}.
\end{equation*}
Moreover, inequality (\ref{Operator-Kernel}) leads directly to
\begin{equation*}
\langle \tilde{f},\tilde{f}\rangle_{\mathcal{H}_0}
=\sum_{j\in\mathbb{N}_n}\sum_{k\in\mathbb{N}_m}\langle
L_{\alpha_k}(K(\alpha_j)\xi_j),\xi_k\rangle_{\mathcal{Y}}
\geq 0.
\end{equation*}
Hence, the bilinear mapping $\langle\cdot,\cdot\rangle_{\mathcal{H}_0}$ is a semi-inner product on $\mathcal{H}_0$. It remains to verify that if $\langle \tilde{f},\tilde{f}
\rangle_{\mathcal{H}_0}=0$ then $\tilde{f}=0$. Suppose that $f:=\sum_{j\in\mathbb{N}_n}K(\alpha_j)\xi_j$ satisfies $\langle\tilde{f},\tilde{f}
\rangle_{\mathcal{H}_0}=0$. By the Cauchy-Schwarz inequality we have for any $\alpha\in\Lambda,\xi\in\mathcal{Y}$ that
$$
|\langle \tilde{f},\widetilde{K(\alpha)\xi}
\rangle_{\mathcal{H}_0}|\leq
\langle \tilde{f},\tilde{f}\rangle_{\mathcal{H}_0}
\langle\widetilde{K(\alpha)\xi},
\widetilde{K(\alpha)\xi}\rangle_{\mathcal{H}_0}=0.
$$
Combining this with (\ref{inner-product}), we have for
any $\alpha\in\Lambda,\xi\in\mathcal{Y}$ that
$$
\langle L_{\alpha}(f),\xi\rangle_{\mathcal{Y}}
=\sum_{j\in\mathbb{N}_n}\langle L_{\alpha}(K(\alpha_j)\xi_j),\xi\rangle
=\langle \tilde{f},\widetilde{K(\alpha)\xi}
\rangle_{\mathcal{H}_0}=0,
$$
which implies that $L_{\alpha}(f)=0$ holds for all $\alpha\in\Lambda$. The definition of the equivalent classes in  $\mathcal{S}_{K}$ implies that $\tilde{f}=0$. Consequently, we have established that $\mathcal{H}_0$ is an inner product space endowed with the inner product $\langle \cdot,\cdot\rangle_{\mathcal{H}_0}$.

Let $\mathcal{H}$ be the completion of $\mathcal{H}_0$ upon the inner product
$\langle \cdot,\cdot\rangle_{\mathcal{H}_0}.$ We finally show that $\mathcal{H}$ is an ORKHS. For each $\tilde{f}\in\mathcal{H}_0$ and each $\alpha\in \Lambda$, we set $L_{\alpha}(\tilde{f}):=L_{\alpha}(f)$.
Then, for any $\alpha\in\Lambda, \xi\in\mathcal{Y}$ and any $\tilde{f}\in\mathcal{H}_0$,
$\langle L_{\alpha}(\tilde{f}),\xi\rangle_{\mathcal{Y}}
=\langle \tilde{f},\widetilde{K(\alpha)\xi}
\rangle_{\mathcal{H}_0}$. This implies that the linear functional $\langle L_{\alpha}(\cdot), \xi\rangle_{\mathcal{Y}}$ is continuous on $\mathcal{H}_0$. For any $F\in \mathcal{H}$, there exists a sequence $\{\tilde{f}_n:\ n\in\mathbb{N}\}\subset \mathcal{H}_0$ converging to $F$ as $n\rightarrow \infty$. By the continuity of the linear functional
$\langle L_{\alpha}(\cdot), \xi\rangle_{\mathcal{Y}}$, we obtain that $\langle L_{\alpha}(\tilde{f}_n), \xi\rangle_{\mathcal{Y}}, n\in\mathbb{N},$ is a Cauchy sequence. We then define $L_{\alpha}(F)$ as an element in $\mathcal{Y}$ such that $\langle L_{\alpha}(F), \xi\rangle_{\mathcal{Y}}=\lim_{n\rightarrow \infty}\langle L_{\alpha}(\tilde{f}_n),\xi\rangle_{\mathcal{Y}}$ for any $\xi\in\mathcal{Y}$. That is, the linear operators $L_{\alpha}$, $\alpha\in \Lambda,$ can be extended to the Hilbert space $\mathcal{H}$. Moreover, it follows that
$$
\langle L_{\alpha}(F),\xi\rangle_{\mathcal{Y}}
=\lim_{n\rightarrow \infty}\langle L_{\alpha}(\tilde{f}_n),\xi\rangle_{\mathcal{Y}}
=\lim_{n\rightarrow \infty}\langle \tilde{f}_n,\widetilde{K(\alpha)\xi}
\rangle_{\mathcal{H}_0}
=\langle F, \widetilde{K(\alpha)\xi}\rangle_{\mathcal{H}}.
$$
This together with the density $\overline{\mathcal{H}_0} =\mathcal{H}$ yields that if $L_{\alpha}(F)=0$ holds for all $\alpha\in \Lambda$ then $\|F\|_{\mathcal{H}}=0$.
This shows that the norm of $\mathcal{H}$ is compatible with $\mathcal{L}$. To show for each $\alpha\in\Lambda,$ $L_{\alpha}$ is continuous from $\mathcal{H}$ to $\mathcal{Y}$, we note that
$$
\|L_{\alpha}(F)\|_{\mathcal{Y}}
=\sup_{\|\xi\|_{\mathcal{Y}}\leq1}|
\langle L_{\alpha}(F),\xi\rangle_{\mathcal{Y}}|
\leq\|F\|_{\mathcal{H}}\sup_{\|\xi\|_{\mathcal{Y}}\leq1}
\|\widetilde{K(\alpha)\xi}\|_{\mathcal{H}}.
$$
For each $G\in\mathcal{H},$ there holds
$$
\sup_{\|\xi\|_{\mathcal{Y}}\leq1}|\langle G, \widetilde{K(\alpha)\xi}\rangle_{\mathcal{H}}|
=\sup_{\|\xi\|_{\mathcal{Y}}\leq1}|\langle L_{\alpha}(G),
\xi\rangle_{\mathcal{Y}}|=\|L_{\alpha}(G)\|_{\mathcal{Y}}.
$$
By the principle of uniform boundedness, we have that the subset $\{\widetilde{K(\alpha)\xi}:\xi\in\mathcal{Y},
\|\xi\|_{\mathcal{Y}}\leq1\}$ of $\mathcal{H}$ is bounded, which leads to the existence of a positive constant $c$ such that $\|L_{\alpha}(F)\|_{\mathcal{Y}}\leq c\|F\|_{\mathcal{H}}$. Consequently, we conclude that $\mathcal{H}$ is an ORKHS with respect to $\mathcal{L}$. We abuse the notation by denoting the equivalent class of $K(\alpha)\xi$ by $K(\alpha)\xi$.  Then we have  that $\overline{\mathcal{S}_{K}}=\mathcal{H}$
and for each $F\in\mathcal{H}$, each $\alpha\in \Lambda$
and each $\xi\in\mathcal{Y}$,
$\langle L_{\alpha}(F),\xi\rangle_{\mathcal{Y}}=\langle F,K(\alpha)\xi\rangle_\mathcal{H}$. In addition, by the density, we also have the uniqueness of $\mathcal{H}$.
\end{proof}

As a direct consequence, we identifies a functional reproducing kernel with an FRKHS in the following result.

\begin{cor}
Suppose that $K$ is an operator from $\Lambda$ to $\mathcal{Z}$. If $K$ is a functional reproducing kernel with respect to a family $\mathcal{F}:=\{L_{\alpha}:
\ \alpha\in \Lambda\}$ of linear functionals on $\mathcal{S}_{K}$, then there exists a unique FRKHS $\mathcal{H}$ with respect to $\mathcal{F}$ such that $\overline{\mathcal{S}_{K}}=\mathcal{H}$ and for each $f\in\mathcal{H}$ and each $\alpha\in \Lambda$, $L_{\alpha}(f)=\langle f,K(\alpha)\rangle_\mathcal{H}$.
\end{cor}

Vector-valued RKHSs are special cases of ORKHSs. We shall show that there exists an isometric isomorphism between an ORKHS and a vector-valued RKHS. Suppose that $\mathcal{H}$ is an ORKHS with respect to a family $\mathcal{L}:=\{L_\alpha: \alpha\in \Lambda\}$ of linear operators form $\mathcal{H}$ to $\mathcal{Y}$ and $K$ is the corresponding operator reproducing kernel. Associated with $K$, we introduce a function $\mathcal{K}:\Lambda\times \Lambda\rightarrow \mathcal{B}(\mathcal{Y},\mathcal{Y})$ by
\begin{equation}\label{kernel_representation}
\mathcal{K}(\alpha,\beta)\xi:=L_{\beta}(K(\alpha)\xi),\ \alpha,\beta\in \Lambda,\ \xi\in\mathcal{Y}.
\end{equation}
We next show that the function so defined is a  reproducing kernel on $\Lambda$. For this purpose, we review the notion of feature maps in the theory of RKHSs. It is known that every reproducing kernel has a feature map representation, which is the reason that the reproducing kernels can be used to measure the similarity of any two inputs in machine learning. Specifically, $\mathcal{K}:\mathcal{X}\times\mathcal{X}\rightarrow \mathcal{B}(\mathcal{Y},\mathcal{Y})$ is a reproducing kernel on an input set $\mathcal{X}$ if and only if there exist a Hilbert space $W$ and a mapping $\Phi:\mathcal{X}\rightarrow \mathcal{B}(\mathcal{Y},W)$ such that
$$
\mathcal{K}(x,y)=\Phi(y)^{*}\Phi(x),\  x,y\in \mathcal{X}.
$$
If $\overline{\span}\{\Phi(x)\xi:\ x\in \mathcal{X},\xi\in\mathcal{Y}\}=W$, the vector-valued RKHS of $\mathcal{K}$ can be determined by
\begin{eqnarray}\label{vector-feature_representation_HK}
\mathcal{H}:=\{\Phi(\cdot)^{*}w:\ w\in W\}
\end{eqnarray}
with the inner product
$
\big\langle\Phi(\cdot)^{*}w,\Phi(\cdot)^{*}v
\big\rangle_{\mathcal{H}}:=\langle w,v\rangle_W.
$

\begin{prop}\label{isometric-isomorphism}
Let $\mathcal{H}$ be an ORKHS with respect to a family $\mathcal{L}=\{L_{\alpha}:\ \alpha\in\Lambda\}$ of linear operators from $\mathcal{H}$ to $\mathcal{Y}$ and $K$ the operator reproducing kernel for $\mathcal{H}$. If $\mathcal{K}$ is the function on $\Lambda\times \Lambda$ defined as in \eqref{kernel_representation}, then $\mathcal{K}$ is a reproducing kernel on $\Lambda$ and the vector-valued RKHS of $\mathcal{K}$ is determined by
\begin{equation}\label{classicalRKHS1}
\mathcal{H}_{\mathcal{K}}:=\{L_{\alpha}(f):
\ \alpha\in\Lambda,\ f\in \mathcal{H}\}
\end{equation}
with the inner product defined for $\widetilde{f}(\alpha):=L_{\alpha}(f)$ and
$\widetilde{g}(\alpha):=L_{\alpha}(g),
\ \alpha\in \Lambda$, by
\begin{equation}\label{classical_inner_product1}
\langle\tilde{f},\tilde{g}\rangle_{\mathcal{H}_{
\mathcal{K}}}:=\langle f,g\rangle_{\mathcal{H}}.
\end{equation}
\end{prop}
\begin{proof}
We prove the desired result by defining the feature map representation of $\mathcal{K}$. We define $\Phi:\Lambda\rightarrow \mathcal{B}(\mathcal{Y},
\mathcal{H})$ for $\alpha\in\Lambda$ by $\Phi(\alpha):=K(\alpha)$. By the definition of $\mathcal{K}$, we have for each $\alpha,\beta\in\Lambda$ that
$$
\langle \mathcal{K}(\alpha,\beta)\xi,
\eta\rangle_{\mathcal{Y}}
=\langle L_{\beta}(K(\alpha)\xi),
\eta\rangle_{\mathcal{Y}},
\  \xi,\eta\in\mathcal{Y}.
$$
This together with the reproducing property (\ref{reproducing_property}) leads to
$$
\langle \mathcal{K}(\alpha,\beta)\xi,
\eta\rangle_{\mathcal{Y}}
=\langle K(\alpha)\xi,K(\beta)\eta\rangle_{\mathcal{H}},
\ \xi,\eta\in\mathcal{Y}.
$$
According to the definition of $\Phi$, we then get for any $\xi,\eta\in\mathcal{Y}$ that
$$
\langle \mathcal{K}(\alpha,\beta)\xi,
\eta\rangle_{\mathcal{Y}}
=\langle \Phi(\alpha)\xi,\Phi(\beta)\eta \rangle_{\mathcal{H}}
=\langle \Phi(\beta)^{*}\Phi(\alpha)\xi,
\eta\rangle_{\mathcal{Y}}.
$$
This leads to $\mathcal{K}(\alpha,\beta)
=\Phi(\beta)^{*}\Phi(\alpha).$ Hence, we have that $\mathcal{K}$ is a reproducing kernel on $\Lambda$. Moreover, $\Phi$ and $\mathcal{H}$ are, respectively,
the feature map and the feature space of $\mathcal{K}$, and there holds the density
$$
\overline{\span}\{\Phi(\alpha)\xi:\ \alpha\in \Lambda,\xi\in\mathcal{Y}\}=\mathcal{H}.
$$
By the feature map representation (\ref{vector-feature_representation_HK}) of the vector-valued RKHS, we have for each $\tilde{f}$ in the vector-valued RKHS of $\mathcal{K}$ that there exists $f\in\mathcal{H}$ such that
$$
\langle\tilde{f}(\alpha),\xi\rangle_{\mathcal{Y}}
=\langle\Phi(\alpha)^{*}f,\xi\rangle_{\mathcal{Y}}
=\langle f,K(\alpha)\xi\rangle_{\mathcal{H}},
\ \alpha\in \Lambda,
$$
and there holds $\langle\tilde{f},\tilde{g}\rangle
=\langle f,g\rangle_{\mathcal{H}},$ for $\tilde{g}(\alpha)=\Phi(\alpha)^{*}g, \alpha\in \Lambda$.
This together with the reproducing property (\ref{reproducing_property}) leads to the representation of $\tilde{f}$ as $\tilde{f}(\alpha)=L_{\alpha}(f),\ \alpha\in\Lambda$. This proves the desired result.
\end{proof}

Proposition \ref{isometric-isomorphism} reveals an isometric isomorphism between the ORKHS $\mathcal{H}$
and $\mathcal{H}_{\mathcal{K}}$, which is established in the next theorem.

\begin{thm}\label{isometric-isomorphism1}
Suppose that $\mathcal{H}$ is an ORKHS with respect to a family $\mathcal{L}=\{L_{\alpha}:\ \alpha\in\Lambda\}$ of linear operators from $\mathcal{H}$ to $\mathcal{Y}$ and $K$ is the operator reproducing kernel for $\mathcal{H}$. If $\mathcal{K}$ is the reproducing kernel on $\Lambda$ defined as in  \eqref{kernel_representation}, then there is an isometric isomorphism between $\mathcal{H}$ and the RKHS $\mathcal{H}_{\mathcal{K}}$ of $\mathcal{K}$.
\end{thm}
\begin{proof}
We introduce the operator $\mathcal{T}:\mathcal{H}
\rightarrow \mathcal{H}_{\mathcal{K}}$ for each $f\in \mathcal{H}$ by $(\mathcal{T}f)(\alpha):=L_{\alpha}(f),
\ \alpha\in \Lambda$. The representation (\ref{classicalRKHS1}) leads to that $\mathcal{T}$
is a surjection. Moreover, it follows from (\ref{classical_inner_product1}) that there holds $\|\mathcal{T}f\|_{\mathcal{H}_{\mathcal{K}}}
=\|f\|_{\mathcal{H}}$, which yields that $\mathcal{T}$
is an isometric isomorphism between $\mathcal{H}$ and $\mathcal{H}_{\mathcal{K}}$.
\end{proof}

When Theorem \ref{isometric-isomorphism1} is specialized to an FRKHS and a scalar-valued RKHS, it gives an isometric isomorphism between an FRKHS and a scalar-valued RKHS. We state this result in the next corollary.

\begin{cor}\label{FRKHS_isometric-isomorphism}
Suppose that $\mathcal{H}$ is an FRKHS with respect to a family $\mathcal{F}=\{L_{\alpha}:\ \alpha\in \Lambda\}$ of linear functionals and $K$ is the functional reproducing kernel for $\mathcal{H}$. If $\mathcal{K}:\Lambda\times \Lambda\rightarrow\mathbb{C}$ is defined by $\mathcal{K}(\alpha,\beta):=L_{\beta}(K(\alpha))$, $\alpha,\beta\in \Lambda$, then $\mathcal{K}$ is a reproducing kernel on $\Lambda$ and the RKHS of $\mathcal{K}$ is determined by
\begin{equation*}\label{classicalRKHS}
\mathcal{H}_{\mathcal{K}}:=\{L_{\alpha}(f):
\ \alpha\in\Lambda,\ f\in \mathcal{H}\}
\end{equation*}
with the inner product defined for $\widetilde{f}(\alpha):=L_{\alpha}(f)$ and
$\widetilde{g}(\alpha):=L_{\alpha}(g),
\ \alpha\in \Lambda$, by
\begin{equation*}\label{classical_inner_product}
\langle\tilde{f},\tilde{g}\rangle_{\mathcal{H}_{
\mathcal{K}}}:=\langle f,g\rangle_{\mathcal{H}}.
\end{equation*}
Moreover, there is an isometric isomorphism between $\mathcal{H}$ and $\mathcal{H}_{\mathcal{K}}$.
\end{cor}

Even though an ORKHS is isometrically isomorphic to a usual vector-valued RKHS, one cannot reduce trivially the study of ORKHSs to that of vector-valued RKHSs. Through the  isometric isomorphism procedure, the new input space  $\Lambda$ may not have any useful structure. Moreover, we observe from the above discussion that the functions in the resulting RKHS and the corresponding reproducing kernel are obtained by taking the linear operators on the original ones. The use of general operators such as the integral operators will bring difficulties to analyzing the new RKHS and its reproducing kernel.

To close this section, we present the characterization of an operator reproducing kernel in terms of its feature map and feature space.
\begin{thm}\label{feature_representation_KL}
An operator $K:\Lambda\to\mathcal{L}(\mathcal{Y},
\mathcal{Z})$ is an operator reproducing kernel with respect to a family $\mathcal{L}=\{L_{\alpha}:
\ \alpha\in\Lambda\}$ of linear operators from $\mathcal{S}_{K}\subseteq\mathcal{Z}$ to $\mathcal{Y}$ if and only if there exists a Hilbert space $W$ and a mapping $\Phi:\Lambda\rightarrow \mathcal{B}(\mathcal{Y},W)$ such that
\begin{equation}\label{feature_representation_KL1}
L_{\beta}(K(\alpha)\xi)=\Phi(\beta)^{*}\Phi(\alpha)\xi,
\ \alpha,\beta\in \Lambda,\xi\in\mathcal{Y}.
\end{equation}
\end{thm}

\begin{proof}
If $K$ is an operator reproducing kernel, we obtain (\ref{feature_representation_KL1}) by choosing
$W$ as the ORKHS $\mathcal{H}$ of $K$ and $\Phi$ as the operator $K$.

Conversely, we assume that (\ref{feature_representation_KL1}) holds for any $\alpha,\beta\in \Lambda$ and any $\xi\in\mathcal{Y}$. Hence, for all $n\in\mathbb{N}$ and $\alpha_j\in \Lambda,\xi_j\in\mathcal{Y},\  j\in \mathbb{N}_n$,
we get that
\begin{equation*}
\sum_{j\in\mathbb{N}_n}\sum_{k\in\mathbb{N}_n}
\langle L_{\alpha_k}(K(\alpha_j)\xi_j),
\xi_k\rangle_{\mathcal{Y}}
=\left\langle \sum_{j\in\mathbb{N}_n}\Phi(\alpha_j)\xi_j,
\sum_{k\in\mathbb{N}_n}\Phi(\alpha_k)\xi_k
\right\rangle_W\geq 0.
\end{equation*}
This together with Definition \ref{Alter-Defn} yields
that $K$ is an operator reproducing kernel with respect
to $\mathcal{L}$.
\end{proof}

As in the theory of RKHSs, a Hilbert space $W$ and a map  $\Phi:\Lambda\rightarrow \mathcal{B}(\mathcal{Y},W)$ satisfying (\ref{feature_representation_KL1}) are called, respectively, the feature space and the feature map of the operator reproducing kernel $K$.

In the scalar case, the feature map representation of a reproducing kernel $\mathcal{K}$ on $\mathcal{X}$ has the form
$$
\mathcal{K}(x,y)=\langle \Phi(x),\Phi(y)\rangle_W,
\ \ x,y\in \mathcal{X},
$$
with $W$ being a Hilbert space and $\Phi$ a mapping from $\mathcal{X}$ to $W$. If there holds $\overline{\span}\{\Phi(x):\ x\in \mathcal{X}\}=W$,
the RKHS of $\mathcal{K}$ can be determined by
\begin{eqnarray*}
\mathcal{H}_{\mathcal{K}}
:=\{\langle w,\Phi(\cdot)\rangle_W:\ w\in W\}
\end{eqnarray*}
with the inner product
$$
\big\langle\langle w,\Phi(\cdot)\rangle_W,
\langle v,\Phi(\cdot)\rangle_W
\big\rangle_{\mathcal{H}_{\mathcal{K}}}
:=\langle w,v\rangle_W.
$$
Due to above feature map representation and Corollary \ref{FRKHS_isometric-isomorphism}, we have the following special result regarding the feature map representation of a functional reproducing kernel.

\begin{thm}\label{FRKHS_feature_representation_KL}
The operator $K:\Lambda\to\mathcal{Z}$ is a functional reproducing kernel with respect to a family $\mathcal{F}=\{L_{\alpha}:\ \alpha\in\Lambda\}$ of linear functionals on $\mathcal{S}_{K}\subseteq \mathcal{Z}$ if and only if there exists a Hilbert space $W$ and a mapping $\Phi:\Lambda\rightarrow W$ such that
\begin{equation}\label{FRKHS_feature_representation_KL1}
L_{\beta}(K(\alpha))=\langle \Phi(\alpha),
\Phi(\beta)\rangle_W,\ \alpha,\beta\in \Lambda.
\end{equation}
\end{thm}

\section{Perfect ORKHSs}

In this section, we study a special interesting class of ORKHSs (of functions from $\mathcal{X}$ to a Hilbert space $\mathcal{Y}$) with respect to two different families of linear operators, one of which is the family of the standard point-evaluation operators. We shall characterize the operator reproducing kernels of this type and study their universality.

For two vector spaces $V_1$ and $V_2$ we say two families of linear operators from $V_1$ to $V_2$ are not equivalent if each linear operator in one family can not be represented by any finite linear combination of linear operators in the other family. Throughout this paper we denote by $\mathcal{L}_0$ the family of point-evaluation operators.

\begin{defn}\label{perfect}
Let a Hilbert space $\mathcal{H}$ of functions from $\mathcal{X}$ to a Hilbert space $\mathcal{Y}$ be an
ORKHS with respect to a family $\mathcal{L}:=\{L_{\alpha}:\alpha\in\Lambda\}$ of linear operators from $\mathcal{H}$ to $\mathcal{Y}$, which is not equivalent to $\mathcal{L}_0$. We call $\mathcal{H}$ a {\it perfect} ORKHS with respect to $\mathcal{L}$ if it is also an ORKHS with respect to $\mathcal{L}_0$.
\end{defn}

In the scalar case when $\mathcal{Y}=\mathbb{C}$, if $\mathcal{H}$ is an FRKHS with respect to not only the family of point-evaluation functionals but also a family of linear functionals, not equivalent to the family of point-evaluation functionals, we shall call it a perfect FRKHS.

It follows from Theorem \ref{property} that a perfect ORKHS admits two operator reproducing kernels, which reproduce two different families of operators, one being $\mathcal{L}_0$. The following proposition gives the relation between the two operator reproducing kernels.

\begin{prop}
If $\mathcal{H}$ is a perfect ORKHS with respect to a family $\mathcal{L}:=\{L_{\alpha}:\alpha\in\Lambda\}$ of linear operators from $\mathcal{H}$ to $\mathcal{Y}$ and if $K$ and $\mathcal{K}$ are the corresponding operator reproducing kernels with respect to $\mathcal{L}$ and $\mathcal{L}_0$, respectively, then
\begin{equation}\label{relation-kernels}
\langle(K(\alpha)\xi)(x),\eta\rangle_{\mathcal{Y}}
=\langle \xi, L_{\alpha}(\mathcal{K}(x,\cdot)\eta)
\rangle_{\mathcal{Y}},\ x\in \mathcal{X},
\ \alpha\in\Lambda,\ \xi,\eta\in\mathcal{Y}.
\end{equation}
\end{prop}
\begin{proof}
It follows from the reproducing properties of the two operator reproducing kernels that for each $f\in\mathcal{H}$ there hold
$$
\langle L_{\alpha}(f),\xi\rangle_{\mathcal{Y}}
=\langle f,K(\alpha)\xi\rangle_{\mathcal{H}},
\ \ \alpha\in\Lambda,\ \xi\in\mathcal{Y}
$$
and
$$
\langle f(x),\eta\rangle_{\mathcal{Y}}
=\langle f,\mathcal{K}(x,\cdot)\eta\rangle_{\mathcal{H}},
\ \ x\in \mathcal{X},\ \eta\in\mathcal{Y}.
$$
In particular, in the first equation we choose
$$
f:=\mathcal{K}(x,\cdot)\eta,\ \ \mbox{for}\ \ x\in \mathcal{X},\ \eta\in\mathcal{Y}
$$
and we obtain that
$$
\langle L_{\alpha}(\mathcal{K}(x,\cdot)\eta),
\xi\rangle_{\mathcal{Y}}
=\langle \mathcal{K}(x,\cdot)\eta,
K(\alpha)\xi\rangle_{\mathcal{H}},
\ \ x\in \mathcal{X},\ \alpha\in\Lambda,
\ \xi,\eta\in\mathcal{Y}.
$$
Likewise, in the second equation we choose
$$
f:=K(\alpha)\xi,\ \ \mbox{for}\ \ \alpha\in\Lambda,\ \xi\in\mathcal{Y}
$$
and we obtain that
$$
\langle (K(\alpha)\xi)(x),\eta\rangle_{\mathcal{Y}}
=\langle K(\alpha)\xi,\mathcal{K}(x,\cdot)
\eta\rangle_{\mathcal{H}},\ \ x\in \mathcal{X},
\ \alpha\in\Lambda,\ \xi,\eta\in\mathcal{Y}.
$$
The desired result of this proposition follows directly from the above two equations.
\end{proof}

In the special case when $\mathcal{H}$ is a perfect FRKHS with respect to the family $\mathcal{F}$ of linear functionals, equation (\ref{relation-kernels}) reduces to
\begin{equation*}\label{relation-kernels1}
K(\alpha)(x)=\overline{L_{\alpha}(\mathcal{K}(x,\cdot))},
\ x\in \mathcal{X},\ \alpha\in\Lambda,
\end{equation*}
where $K$ is the functional reproducing kernel with respect to $\mathcal{F}$ and $\mathcal{K}$ is the classical reproducing kernel on $\mathcal{X}$.

We next represent the operator reproducing kernel $K$ for a perfect ORKHS $\mathcal{H}$ by the feature map. For this purpose, we let $W$ and $\Phi:\mathcal{X}\rightarrow \mathcal{B}(\mathcal{Y},W)$ be the feature space and the feature map of $\mathcal{K}$, respectively. We assume that there holds the density condition
\begin{equation}\label{density_phi}
\overline{\span}\{\Phi(x)\xi:\ x\in \mathcal{X},
\xi\in\mathcal{Y}\}=W.
\end{equation}
Associated with the features $\Phi$ and $W$ of $\mathcal{K}$, we introduce a map $\Psi$ from $\Lambda$
to $\mathcal{B}(\mathcal{Y},W)$ as follows. Noting that for each $\alpha\in\Lambda$, the linear operator $\widetilde{L}_{\alpha}$ defined on $W$ for each
$u\in W$ by
$$
\widetilde{L}_{\alpha}(u):=L_{\alpha}(\Phi(\cdot)^{*}u)
$$
is continuous since there holds for each $u\in W$,
$$
\|\widetilde{L}_{\alpha}(u)\|_{\mathcal{Y}}
\leq \|L_{\alpha}\|_{\mathcal{B}(\mathcal{H},\mathcal{Y})}
\|\Phi(\cdot)^{*}u\|_{\mathcal{H}}
=\|L_{\alpha}\|_{\mathcal{B}
(\mathcal{H},\mathcal{Y})}\|u\|_W.
$$
This allows us to define $\Psi$ by
\begin{equation}\label{Psi}
\Psi(\alpha):=\widetilde{L}^{*}_{\alpha},
\ \ \alpha\in\Lambda.
\end{equation}
In terms of the maps $\Phi$ and $\Psi$, we can represent $K$ as in the following proposition.

\begin{prop}\label{feature_representation_perfect}
Suppose that $\mathcal{H}$ is a perfect ORKHS with respect to a family $\mathcal{L}:=\{L_{\alpha}:\alpha\in\Lambda\}$ of linear operators from $\mathcal{H}$ to $\mathcal{Y}$. Let $K$ and $\mathcal{K}$ denote the operator reproducing kernels with respect to $\mathcal{L}$ and $\mathcal{L}_0$, respectively. If $W$ is the feature space of $\mathcal{K}$, $\Phi:\mathcal{X}\rightarrow \mathcal{B}(\mathcal{Y},W)$ is its feature map with the density condition \eqref{density_phi} and $\Psi:\Lambda\rightarrow\mathcal{B}(\mathcal{Y},W)$
is defined by \eqref{Psi}, then
$$
K(\alpha)=\Phi(\cdot)^{*}\Psi(\alpha),
\ \ \alpha\in\Lambda.
$$
Moreover, $W$ and $\Psi$ are, respectively, the feature space and the feature map of $K$, satisfying the density condition
\begin{equation}\label{density_psi}
\overline{\span}\{\Psi(\alpha)\xi:\ \alpha\in \Lambda,\xi\in\mathcal{Y}\}=W.
\end{equation}
\end{prop}
\begin{proof}
It follows from the feature map representation $\mathcal{K}(x,y)=\Phi(y)^{*}\Phi(x),\ x,y\in\mathcal{X},$ that
$$
L_{\alpha}(\mathcal{K}(x,\cdot)\eta)
=L_{\alpha}(\Phi(\cdot)^{*}\Phi(x)\eta)
=\widetilde{L}_{\alpha}(\Phi(x)\eta),
\ x\in \mathcal{X},\ \alpha\in\Lambda,
\ \eta\in\mathcal{Y}.
$$
Substituting the above equation into (\ref{relation-kernels}), we obtain that
$$
\langle(K(\alpha)\xi)(x), \eta\rangle_{\mathcal{Y}}
=\langle \xi, \widetilde{L}_{\alpha}
(\Phi(x)\eta)\rangle_{\mathcal{Y}}.
$$
Then by the definition of $\Psi$ we get that
$$
\langle(K(\alpha)\xi)(x),\eta\rangle_{\mathcal{Y}}
=\langle \Psi(\alpha)\xi,\Phi(x)\eta\rangle_W
=\langle \Phi(x)^{*}\Psi(\alpha)
\xi,\eta\rangle_{\mathcal{Y}},
$$
which implies $K(\alpha)=\Phi(\cdot)^{*}\Psi(\alpha), \alpha\in\Lambda.$

Furthermore, for any $\alpha,\beta\in\Lambda$, and any $\xi,\eta\in\mathcal{Y}$, we have that
$$
\langle L_{\beta}(K(\alpha)\xi),\eta\rangle_{\mathcal{Y}}
=\langle\Psi(\alpha)\xi,\Psi(\beta)\eta\rangle_W
=\langle\Psi(\beta)^{*}\Psi(\alpha)\xi,
\eta\rangle_{\mathcal{Y}},
$$
which yields $L_{\beta}(K(\alpha)\xi)
=\Psi(\beta)^{*}\Psi(\alpha)\xi$. We conclude from this that $W$ and $\Psi$ are, respectively, the feature space and the feature map of $K$.

To verify the density condition (\ref{density_psi}),
it suffices to show that if $u\in W$ satisfies
$\langle u,\Psi(\alpha)\xi\rangle_W=0$ for all
$\alpha\in \Lambda$ and all $\xi\in\mathcal{Y}$, then $u=0$. Since there holds for all $\alpha\in \Lambda$
and all $\xi\in\mathcal{Y}$,
$$
\langle\widetilde{L}_{\alpha}(u),\xi\rangle_{\mathcal{Y}}
=\langle u,\Psi(\alpha)\xi\rangle_W,
$$
we have that $\widetilde{L}_{\alpha}(u)=0$
for all $\alpha\in\Lambda$. By the definition $\widetilde{L}_{\alpha}$, we get that $L_{\alpha}(\Phi(\cdot)^{*}u)=0$ for all $\alpha\in\Lambda$. Since the norm of $\mathcal{H}$
is compatible with $\mathcal{L}$, we observe that $\|\Phi(\cdot)^{*}u\|_{\mathcal{H}}=0$, which is equivalent to $\|u\|_W=0$.
\end{proof}

Motivated by the representation of an operator
reproducing kernel for a perfect ORKHS in Proposition \ref{feature_representation_perfect}, we now turn to characterizing the operator reproducing kernels of this type, which is useful in constructing operator reproducing kernels.

\begin{thm}\label{L_RKHS_Thm}
An operator $K:\Lambda\rightarrow\mathcal{B}
(\mathcal{Y},\mathcal{H})$ is an operator reproducing kernel for a perfect ORKHS $\mathcal{H}$ with respect to a family $\mathcal{L}=\{L_{\alpha}:\alpha\in\Lambda\}$ of linear operators from $\mathcal{H}$ to $\mathcal{Y}$ if and only if there exist a Hilbert space $W$ and two mappings $\Phi:\mathcal{X}\rightarrow \mathcal{B}(\mathcal{Y},W)$ with the density condition \eqref{density_phi}
and $\Psi:\Lambda\rightarrow \mathcal{B}(\mathcal{Y},W)$ with the density condition \eqref{density_psi} such that
\begin{equation}\label{kernel-perfectRKHS}
K(\alpha)=\Phi(\cdot)^{*}\Psi(\alpha),
\ \ \alpha\in\Lambda.
\end{equation}
Moreover, the perfect ORKHS of $K$ has the form $\mathcal{H}:=\{\Phi(\cdot)^{*}u:u\in W\}$ with
the inner product
$$
\langle \Phi(\cdot)^{*}u,\Phi(\cdot)^{*}v
\rangle_{\mathcal{H}}:=\langle u,v\rangle_W.
$$
\end{thm}
\begin{proof}
Suppose that $K:\Lambda\rightarrow\mathcal{B}(\mathcal{Y},\mathcal{H})$ is an operator reproducing kernel for a perfect ORKHS $\mathcal{H}$ with respect to a family  $\mathcal{L}=\{L_{\alpha}:\alpha\in\Lambda\}$ of linear operators from $\mathcal{H}$ to $\mathcal{Y}$. According to the definition of the perfect ORKHS, $\mathcal{H}$ is a reproducing kernel Hilbert space with its reproducing kernel $\mathcal{K}$. By theory of the classical reproducing kernel Hilbert space, we can choose a Hilbert space $W$ as the feature space of the reproducing kernel $\mathcal{K}$ and a mapping $\Phi:\mathcal{X}\rightarrow \mathcal{B}(\mathcal{Y},W)$ which satisfying the density condition \eqref{density_phi} as its feature map. We then define the map $\Psi$ according to (\ref{Psi}). Hence, by Proposition \ref{feature_representation_perfect}, we get the representation of $K$ as in  (\ref{kernel-perfectRKHS}).

Conversely, we suppose that $W$ is a Hilbert space, $\Phi:\mathcal{X}\rightarrow \mathcal{B}(\mathcal{Y},W)$ satisfies the density condition \eqref{density_phi} and $\Psi:\Lambda\rightarrow \mathcal{B}(\mathcal{Y},W)$ satisfies the density condition \eqref{density_psi}, and $K$ has the form \eqref{kernel-perfectRKHS}.
Associated with $W$ and $\Phi$ we introduce a space $\mathcal{H}:=\{\Phi(\cdot)^{*}u:u\in W\}$ of functions from $\mathcal{X}$ to $\mathcal{Y}$ and define a
bilinear mapping on $\mathcal{H}$ by
$$
\langle \Phi(\cdot)^{*}u,\Phi(\cdot)^{*}
v\rangle_{\mathcal{H}}:=\langle u,v\rangle_W.
$$
By the theory of classical RKHSs, we conclude that $\mathcal{H}$ is an RKHS. We shall show that $\mathcal{H}$ is a perfect ORKHS with respect to some family of linear operators and the operator $K$ is exactly the operator reproducing kernel for $\mathcal{H}$.

We first introduce a family of linear operators
from  $\mathcal{H}$ to $\mathcal{Y}$. For each $\alpha\in\Lambda$ we define operator $L_{\alpha}$
on $\mathcal{H}$ by
$$
L_{\alpha}(\Phi(\cdot)^{*}u):=\Psi(\alpha)^{*}u,\ u\in W.
$$
It is clear that $L_{\alpha}$ is linear. It follows that
$$
\|L_{\alpha}(\Phi(\cdot)^{*}u)\|_{\mathcal{Y}}
\leq\|\Psi(\alpha)\|_{\mathcal{B}(\mathcal{Y},W)}\|u\|_W
=\|\Psi(\alpha)\|_{\mathcal{B}(\mathcal{Y},W)}
\|\Phi(\cdot)^{*}u\|_{\mathcal{H}},
$$
which ensures that $L_{\alpha}$ is continuous
on $\mathcal{H}$. Set $\mathcal{L}:=\{L_{\alpha}:\alpha\in\Lambda\}$.
Then $\mathcal{L}$ is a family of continuous linear operators. We next verify that the norm of $\mathcal{H}$ is compatible with $\mathcal{L}$. Suppose that $f:=\Phi(\cdot)^{*}u\in\mathcal{H}$ satisfies $L_{\alpha}(f)=0$ for all $\alpha\in\Lambda$.
By the definition of operators $L_{\alpha}$,
we have for all $\alpha\in\Lambda$ and all $\xi\in\mathcal{Y}$ that
$$
\langle u, \Psi(\alpha)\xi\rangle_W
=\langle \Psi(\alpha)^{*}u, \xi\rangle_{\mathcal{Y}}
=\langle L_{\alpha}(\Phi(\cdot)^{*}u),
\xi\rangle_{\mathcal{Y}}=0.
$$
This together with the density condition
(\ref{density_psi}) leads to $\|u\|_{W}=0$
and then $\|f\|_{\mathcal{H}}=0$. That is, the norm of $\mathcal{H}$ is compatible with $\mathcal{L}$.
We then conclude that $\mathcal{H}$ is a perfect ORKHS with respect to $\mathcal{L}$.

It remains to prove that the operator $K$ with the form (\ref{kernel-perfectRKHS}) is the operator reproducing kernel for $\mathcal{H}$. It follows from
$$
\|K(\alpha)\xi\|_{\mathcal{H}}
=\|\Psi(\alpha)\xi\|_W\leq\|\Psi(\alpha)
\|_{\mathcal{B}(\mathcal{Y},W)}\|\xi\|_{\mathcal{Y}},
\ \xi\in\mathcal{Y},
$$
that $K(\alpha)\in\mathcal{B}(\mathcal{Y},\mathcal{H})$. For each $\alpha\in\Lambda,\xi\in\mathcal{Y}$ and each $f:=\Phi(\cdot)^{*}u\in\mathcal{H}$, there holds
\begin{equation*}
\langle f,K(\alpha)\xi\rangle_{\mathcal{H}}
=\langle\Phi(\cdot)^{*}u,\Phi(\cdot)^{*}
\Psi(\alpha)\xi\rangle_{\mathcal{H}}
=\langle u,\Psi(\alpha)\xi\rangle_W
=\langle L_{\alpha}(f),\xi\rangle_{\mathcal{Y}}.
\end{equation*}
This ensures that $K$ is the operator reproducing kernel for $\mathcal{H}$.
\end{proof}

We remark on the representation (\ref{kernel-perfectRKHS}) of an operator reproducing kernel for a perfect ORKHS. For a general operator reproducing kernel, the feature map can only give the representation (\ref{feature_representation_KL1}) for the values of the operators $L_{\beta}$, $\beta\in \Lambda$, at the kernel sections $K(\alpha)\xi, \alpha\in\Lambda, \xi\in\mathcal{Y}$. Due to the special structure of the perfect ORKHS, we can give a more precise representation (\ref{kernel-perfectRKHS}) for its operator reproducing kernel.

We next present a specific example to illustrate the representation (\ref{kernel-perfectRKHS}) of the
operator reproducing kernel for a perfect ORKHS.
To this end, we recall the RKHS $\mathcal{B}_{\Delta}
(\mathbb{R}^d,\mathbb{C}^{n})$. Proposition
\ref{Paley} shows that the space $\mathcal{B}_{\Delta}
(\mathbb{R}^d,\mathbb{C}^{n})$ is also an ORKHS with respect to the family $\mathcal{L}$ of the integral operators defined as in (\ref{example_integration_functional1}). Hence, it is a perfect ORKHS. It was pointed out in section 3 that the operator reproducing kernel for $\mathcal{B}_{\Delta}
(\mathbb{R}^d,\mathbb{C}^{n})$ has the form
\begin{equation}\label{example-feature-representation}
K(x)\xi=\mathscr{P}_{\mathcal{B}_{\Delta}}(u(\cdot-x)\xi),
\ \ x\in \mathbb{R}^d, \ \ \xi\in\mathbb{C}^{n}.
\end{equation}
To give the feature representation (\ref{kernel-perfectRKHS}) of $K$, we first describe
the choice of the space $W$ as the feature space of
the kernel $\mathcal{K}$ defined as in (\ref{sinc})
and the mapping $\Phi$ as its feature map. Let $W:=L^2(I_{\Delta},\mathbb{C}^n)$. We define $\Phi:\mathbb{R}^d\to\mathcal{B}(\mathbb{C}^n,W)$
for each $x\in\mathbb{R}^d$ by
$$
\Phi(x):=\frac{1}{(2\pi)^{d/2}}
{\rm diag}(e^{i(x,\cdot)}:j\in\mathbb{N}_{n}).
$$
Then the adjoint of $\Phi(x)$ can be represented
for each $w:=[w_j:j\in\mathbb{N}_n]\in W$ as
$$
\Phi(x)^{*}w=\frac{1}{(2\pi)^{d/2}}
[\langle w_j,e^{i(x,\cdot)}\rangle_{L^2(I_{\Delta})}:
j\in\mathbb{N}_{n}].
$$
Associated with $W$ and $\Phi$, the kernel $\mathcal{K}$ having the form (\ref{sinc}) may be rewritten as
$$
\mathcal{K}(x,y)=\Phi(y)^{*}\Phi(x),
\ \ x,y\in\mathbb{R}^d.
$$
This shows that $W$ and $\Phi$ are, respectively, the feature space and the feature map of $\mathcal{K}$.
We next present the other map $\Psi$ needed for the construction of $K$ in terms of the function $u$
appearing in (\ref{example_integration_functional1}). Specifically, for each $x\in\mathbb{R}^d$, we set
$$
\Psi(x):=(2\pi)^{d/2}{\rm diag}
(\check{u}e^{i(x,\cdot)}:j\in\mathbb{N}_{n}).
$$
By equation (\ref{example-feature-representation}),
we have for each $x,y\in\mathbb{R}^d$ and each $\xi\in\mathbb{C}^n$ that
\begin{eqnarray*}
(K(x)\xi)(y)=(\hat{u}e^{-i(x,\cdot)}\xi
\chi_{I_{\Delta}})^{\vee}(y)
=\langle \check{u}e^{i(x,\cdot)},e^{i(y,\cdot)}
\rangle_{L^2(I_{\Delta})}\xi.
\end{eqnarray*}
This together with the definition of $\Phi(y)^{*}$ and $\Psi(x)$ leads to $(K(x)\xi)(y)=\Phi(y)^{*}\Psi(x)\xi.$
That is, $K(x)=\Phi(\cdot)^{*}\Psi(x)$.

The following corollary gives a characterization of a functional reproducing kernel for a perfect FRKHS.

\begin{cor}
An operator $K:\Lambda\to\mathcal{H}$ is a functional reproducing kernel for a perfect FRKHS $\mathcal{H}$
with respect to a family  $\mathcal{F}=\{L_{\alpha}:\alpha\in\Lambda\}$ of linear functionals on $\mathcal{H}$ if and only if there
exist a Hilbert space $W$ and two mappings $\Phi:\mathcal{X}\rightarrow W$ with $\overline{\span}\{\Phi(x):x\in \mathcal{X}\}=W$ and $\Psi:\Lambda\rightarrow W$ with $\overline{\span}\{\Psi(\alpha):\alpha\in \Lambda\}=W$ such that
\begin{equation}\label{FRKHS_kernel-perfectRKHS}
K(\alpha)=\langle\Psi(\alpha),\Phi(\cdot)\rangle_W,
\ \ \alpha\in\Lambda.
\end{equation}
Moreover, the perfect FRKHS of $K$ has the form $\mathcal{H}:=\{\langle u,\Phi(\cdot)\rangle_W:u\in W\}$ with the inner product
$$
\langle \langle u,\Phi(\cdot)\rangle_W,\langle v,\Phi(\cdot)\rangle_W\rangle_{\mathcal{H}}
:=\langle u,v\rangle_W.
$$
\end{cor}

Associated with an operator reproducing kernel in the form (\ref{kernel-perfectRKHS}) for a perfect ORKHS, there are two reproducing kernels
$$
\mathcal{K}(x,y)=\Phi(y)^{*}\Phi(x),\ x,y\in\mathcal{X}
\ \ \mbox{and}\ \ \widetilde{\mathcal{K}}(\alpha,\beta)
=\Psi(\beta)^{*}\Psi(\alpha),\ \alpha,\beta\in\Lambda.
$$
As a direct consequence of Theorems \ref{isometric-isomorphism1} and \ref{L_RKHS_Thm}, there exists an isometric isomorphism between the RKHSs of $\mathcal{K}$ and $\widetilde{\mathcal{K}}$. We present this result below.

\begin{thm}\label{K_L_K}
Suppose that $W$ is a Hilbert space and two mappings $\Phi:\mathcal{X}\rightarrow \mathcal{B}(\mathcal{Y},W)$
and $\Psi:\Lambda\rightarrow \mathcal{B}(\mathcal{Y},W)$ satisfy the density conditions \eqref{density_phi} and \eqref{density_psi}, respectively.
If the mapping $K$ has the form \eqref{kernel-perfectRKHS} and two mappings $\mathcal{K}$ and $\widetilde{\mathcal{K}}$ are defined as above, then $K$ and $\mathcal{K}$ are the operator reproducing kernel and the reproducing kernel, respectively, for a perfect ORKHS $\mathcal{H}$ and $\widetilde{\mathcal{K}}$ is a reproducing kernel for an RKHS $\mathcal{\widetilde{H}}$. Moreover, there is an isometric isomorphism between $\mathcal{H}$ and $\mathcal{\widetilde{H}}$.
\end{thm}
\begin{proof}
It is known from Theorem \ref{L_RKHS_Thm} that $K$
is an operator reproducing kernel for a perfect ORKHS $\mathcal{H}$ with respect $\mathcal{L}=\{L_{\alpha}:\ \alpha\in\Lambda\}$
defined by
$$
L_{\alpha}(\Phi(\cdot)^{*}u):=\Psi(\alpha)^{*}u,\ u\in W.
$$
Moreover, $\mathcal{K}$ is the corresponding reproducing kernel for $\mathcal{H}$ with respect to the family of the  point-evaluation operators. The feature map representation of $\widetilde{\mathcal{K}}$ leads directly to that  $\widetilde{\mathcal{K}}$ is a reproducing kernel on $\Lambda$ for the RKHS  $\mathcal{\widetilde{H}}:=\{\Psi(\cdot)^{*}u:u\in W\}$.
By definitions of $K, \widetilde{\mathcal{K}}$ and $L_{\beta}$, we get for all $\alpha,\beta\in\Lambda$ and all $\xi\in\mathcal{Y}$  that
$$
\widetilde{\mathcal{K}}(\alpha,\beta)\xi
=\Psi(\beta)^{*}\Psi(\alpha)\xi
=L_{\beta}(\Phi(\cdot)^{*}\Psi(\alpha)\xi)
=L_{\beta}(K(\alpha)\xi).
$$
We then conclude from Theorem \ref{isometric-isomorphism1} that there is an isometric isomorphism between $\mathcal{H}$ and $\mathcal{\widetilde{H}}$.
\end{proof}

To close this section, we consider universality of an operator reproducing kernel for a perfect ORKHS.
Universality of classical reproducing kernels has attracted much attention in theory of machine learning \cite{CMPY,MXZ,PMRRV,St}. This important property ensures that a continuous target function can be uniformly approximated on a compact subset of the input space by the linear span of kernel sections. It is crucial for the consistence of the kernel-based learning algorithms, which generally use the representer theorem to learn a target function in the linear span of kernel sections. For learning a function from its operator values, we also have similar representer theorems, which will be presented in section 7. Below, we discuss universality of an operator reproducing kernel for a perfect ORKHS.

Let Hilbert space $\mathcal{H}$ of $\mathcal{Y}$-valued functions on a metric space $\mathcal{X}$ is a perfect ORKHS with respect to a family $\mathcal{L}$ of linear operators from $\mathcal{H}$ to $\mathcal{Y}$. Suppose that the operator reproducing kernel $K$ for $\mathcal{H}$ satisfies that for each $\alpha\in\Lambda$ and each $\xi\in\mathcal{Y}$, $K(\alpha)\xi$ is continuous on $\mathcal{X}$. We denote by $C(\mathcal{Z},\mathcal{Y})$ the Banach space of continuous $\mathcal{Y}$-valued functions on compact subset $\mathcal{Z}$ of $\mathcal{X}$ with the norm defined by $\|f\|_{C(\mathcal{Z},\mathcal{Y})}
:=\sup_{x\in\mathcal{Z}}\|f(x)\|_{\mathcal{Y}}$.
We say that $K$ is universal if for each compact
subset $\mathcal{Z}$ of $\mathcal{X}$,
\begin{equation*}\label{universal}
\overline{\span}\{K(\alpha)\xi:
\alpha\in\Lambda,\ \xi\in\mathcal{Y}\}
=C(\mathcal{Z},\mathcal{Y}),
\end{equation*}
where for each $\alpha\in\Lambda$ and each $\xi\in\mathcal{Y}$, $K(\alpha)\xi$ is taken as a $\mathcal{Y}$-valued function on $\mathcal{Z}$ and the closure is taken in the maximum norm. To characterize universality of the operator reproducing kernel having
the form (\ref{kernel-perfectRKHS}), we first recall a convergence property of functions in an RKHS. That is,
if a sequence $f_n$ converges to $f$ in an RKHS, then $f_n$, as a sequence of functions on $\mathcal{X}$, converges pointwisely to $f$. As a direct consequence,
we have the following result concerning the density in $\mathcal{H}$.
\begin{lemma}\label{density-H-C}
Let $\mathcal{H}$ be an RKHS of continuous $\mathcal{Y}$-valued functions on a metric space $\mathcal{X}$ and $\mathcal{M}$ a subset of $\mathcal{H}$. If $\mathcal{M}$ is dense in $\mathcal{H}$, then for any compact subset $\mathcal{Z}$ of $\mathcal{X}$ there holds $\overline{\mathcal{M}}=\overline{\mathcal{H}}$, where the closure is taken in the maximum norm of $C(\mathcal{Z},\mathcal{Y})$.
\end{lemma}
\begin{proof}
It suffices to prove that for any compact subset $\mathcal{Z}$ there holds $\mathcal{H}\subseteq \overline{\mathcal{M}},$ where the closure is taken
in the maximum norm of $C(\mathcal{Z},\mathcal{Y})$.
Let $f$ be any element in $\mathcal{H}$. By the density
of $\mathcal{M}$ in Hilbert space $\mathcal{H}$, there exists a sequence $f_n$ converging to $f$. Due to the convergence property in RKHSs, we have for any compact subset $\mathcal{Z}$ of $\mathcal{X}$ that $f_n$
converges to $f$ in $C(\mathcal{Z},\mathcal{Y})$. According to the arbitrary of $f$ in $\mathcal{H}$,
we get $\mathcal{H}\subseteq \overline{\mathcal{M}}.$
\end{proof}
By the above lemma, we give a characterization of universality of the operator reproducing kernel with
the form (\ref{kernel-perfectRKHS}) as follows.
\begin{thm}\label{universal-characterize}
Let $\mathcal{X}$ be a metric space, $\Lambda$ a set and $W,\mathcal{Y}$ both Hilbert spaces. Suppose that $\Phi:\mathcal{X}\rightarrow \mathcal{B}(\mathcal{Y},W)$ is continuous and satisfies the density condition \eqref{density_phi} and $\Psi:\Lambda\rightarrow \mathcal{B}(\mathcal{Y},W)$ satisfies the density condition \eqref{density_psi}. Then the operator reproducing kernel $K$ defined by \eqref{kernel-perfectRKHS} is universal if and
only if there holds for any compact subset
$\mathcal{Z}$ of $\mathcal{X}$,
\begin{equation}\label{universal1}
\overline{\span}\{\Phi(\cdot)^{*}u:
u\in W\}=C(\mathcal{Z},\mathcal{Y}).
\end{equation}
\end{thm}
\begin{proof}
It follows from Theorem \ref{L_RKHS_Thm} that the perfect ORKHS of the operator reproducing kernel $K$ defined by   \eqref{kernel-perfectRKHS} has the form $\mathcal{H}:=\{\Phi(\cdot)^{*}u:u\in W\}$ with the inner product $\langle \Phi(\cdot)^{*}u,\Phi(\cdot)^{*}v
\rangle_{\mathcal{H}}:=\langle u,v\rangle_W.$ Since $\Phi$ is continuous, the functions in $\mathcal{H}$ are all continuous. Let $\mathcal{M}$ denote the linear span $\span\{\Phi(\cdot)^{*}\Psi(\alpha)\xi:\alpha\in \Lambda,
\xi\in\mathcal{Y}\}$. By the density condition \eqref{density_psi}, we have that $\mathcal{M}$
is dense in $\mathcal{H}$. Then according to Lemma
\ref{density-H-C}, we get for any compact subset $\mathcal{Z}$ of $\mathcal{X}$ that
\begin{equation*}
\overline{\span}\{\Phi(\cdot)^{*}\Psi(\alpha)\xi:
\xi\in\mathcal{Y}\}=\overline{\span}\{\Phi(\cdot)^{*}u: u\in W\},
\end{equation*}
where the closure is taken in the maximum norm of $C(\mathcal{Z},\mathcal{Y})$. The above equation together with \eqref{kernel-perfectRKHS} leads to
\begin{equation*}
\overline{\span}\{K(\alpha)\xi: \xi\in\mathcal{Y}\}
=\overline{\span}\{\Phi(\cdot)^{*}u: u\in W\},
\end{equation*}
which completes the proof.
\end{proof}

Besides the operator reproducing kernel $K$, the
perfect ORKHS also enjoys a classical reproducing
kernel on $\mathcal{X}$ with the form $\mathcal{K}(x,y)=\Phi(y)^{*}\Phi(x)$, $x,y\in \mathcal{X}$. The characterization of universality
of classical reproducing kernels in \cite{CMPY,MXZ}
shows that $\mathcal{K}$ is universal if and only if (\ref{universal1}) holds for any compact subset $\mathcal{Z}$ of $\mathcal{X}$. Hence, Theorem  \ref{universal-characterize} tells us that for a
perfect ORKHS the universality of its two kernels
is identical.

\section{Perfect ORKHSs with respect to Integral Operators}

Local average values of a function are often used as observed information in data analysis. For this reason,
we investigate in this section the perfect ORKHS with respect to a family of integral operators.

We first introduce the integral operators to be considered in this section. Let $(\mathcal{X},\mu)$ be a measure space and $\mathcal{Y}$ a Hilbert space. Suppose that $\Lambda$ is an index set and $\{\lambda_{\alpha}:
\ \alpha\in \Lambda\}$ is a family of mappings from $\mathcal{X}$ to $\mathcal{B}(\mathcal{Y},\mathcal{Y})$. Associated with these mappings, we define the family $\mathcal{L}:=\{L_{\alpha}:\alpha\in\Lambda\}$ of
integral operators for a function
$f$ from $\mathcal{X}$ to $\mathcal{Y}$ by
\begin{equation}\label{integral-operator}
L_{\alpha}(f):=\int_{\mathcal{X}}\lambda_{\alpha}(x)f(x)
d\mu(x),\ \alpha\in\Lambda.
\end{equation}

We now describe a construction of perfect ORKHSs with respect to $\mathcal{L}$ defined by \eqref{integral-operator} following Theorem \ref{L_RKHS_Thm}. To this end, we suppose that $W$ is a Hilbert space and a mapping $\Phi:\mathcal{X}\rightarrow \mathcal{B}(\mathcal{Y},W)$ satisfies the density condition \eqref{density_phi}. The perfect ORKHSs will
be constructed through a subspace of $W$. For each $\alpha\in\Lambda$ and each $\xi\in\mathcal{Y}$, we introduce a function from $\mathcal{X}$ to $W$ by letting
\begin{equation}\label{psi}
\psi_{\alpha,\xi}(x):=\Phi(x)\lambda_{\alpha}(x)^{*}\xi,
\ x\in\mathcal{X}.
\end{equation}
If the functions $\psi_{\alpha,\xi}$, $\alpha\in\Lambda$, $\xi\in\mathcal{Y}$, are all Bochner integrable, then we define a closed subspace of $W$ by
\begin{equation}\label{W}
\widetilde{W}:=\overline{\span}\left\{\int_{\mathcal{X}}
\psi_{\alpha,\xi}(x)d\mu(x):
\ \alpha\in\Lambda,\xi\in\mathcal{Y}\right\}.
\end{equation}
We let $\mathcal{P}_{\widetilde{W}}$ denote the orthogonal projection from $W$ to $\widetilde{W}$. We then define the mapping $\widetilde{\Phi}$ by
$$
\widetilde{\Phi}(x):=\mathcal{P}_{\widetilde{W}}\Phi(x),
\ \ x\in \mathcal{X}.
$$
The next lemma establishes the density condition involving the mapping $\widetilde{\Phi}:\mathcal{X}\rightarrow \mathcal{B}(\mathcal{Y},\widetilde{W})$ in the Hilbert space $\widetilde{W}$.

\begin{lemma}\label{Projection}
Let $W$ be a Hilbert space and $\Phi:\mathcal{X}\rightarrow\mathcal{B}(\mathcal{Y},W)$
a mapping satisfying the density condition \eqref{density_phi}. If $\psi_{\alpha,\xi}$, $\alpha\in\Lambda$, $\xi\in\mathcal{Y}$, defined
by \eqref{psi}, are all Bochner integrable and $\widetilde{W}$ and $\widetilde{\Phi}$ are defined
as above, then there holds the density condition
\begin{equation}\label{density_phi1}
\overline{\span}\{\widetilde{\Phi}(x)\xi:
x\in \mathcal{X},\xi\in\mathcal{Y}\}=\widetilde{W}.
\end{equation}
\end{lemma}
\begin{proof}
By the definition of $\widetilde{\Phi}$, we have for each $u\in \widetilde{W}$ that
$$
\langle\widetilde{\Phi}(x)\xi,u\rangle_{\widetilde{W}}
=\langle\Phi(x)\xi,u\rangle_{W},
\ x\in \mathcal{X},\xi\in\mathcal{Y}.
$$
This together with the density condition \eqref{density_phi} of $\Phi$ leads to (\ref{density_phi1}).
\end{proof}

We describe in the next theorem the construction of
the perfect ORKHS with respect to $\mathcal{L}$.

\begin{thm}\label{integral-perfectORKHS}
Let $W$ be a Hilbert space and $\Phi:\mathcal{X}\rightarrow\mathcal{B}(\mathcal{Y},W)$
a mapping satisfying the density condition \eqref{density_phi}. If a collection of mappings $\lambda_{\alpha}:\mathcal{X}\to
\mathcal{B}(\mathcal{Y},\mathcal{Y})$,
 $\alpha\in \Lambda$ satisfies
\begin{equation}\label{condition}
\int_{\mathcal{X}}\|\Phi(x)\|_{\mathcal{B}(\mathcal{Y},W)}
\|\lambda_{\alpha}(x)\|_{\mathcal{B}(\mathcal{Y},\mathcal{Y})}
d\mu(x)<\infty,\ \  \mbox{for all}\ \ \alpha\in\Lambda,
\end{equation}
then the space
\begin{equation}\label{ORKHS-integral-operator}
\widetilde{\mathcal{H}}:=\left\{\Phi(\cdot)^{*}u:
u\in\widetilde{W}\right\}
\end{equation}
with the inner product
\begin{equation*}\label{inner-product-integral-operator}
\langle\Phi(\cdot)^{*}u, \Phi(\cdot)^{*}v
\rangle_{\widetilde{\mathcal{H}}}
:=\langle u,v\rangle_{\widetilde{W}},
\end{equation*}
is a perfect ORKHS with respect to
$\mathcal{L}:=\{L_{\alpha}:\alpha\in\Lambda\}$
defined as in \eqref{integral-operator} and its
operator reproducing kernel $K$ has the form
\begin{equation}\label{L_reproducing_kernel_integral}
K(\alpha)\xi=\Phi(\cdot)^{*}\int_{\mathcal{X}}
\psi_{\alpha,\xi}(x)d\mu(x),\ \alpha\in\Lambda,
\ \xi\in \mathcal{Y}.
\end{equation}
\end{thm}
\begin{proof}
We prove this theorem by employing Theorem \ref{L_RKHS_Thm}. Specifically, we verify the hypothesis of Theorem \ref{L_RKHS_Thm}. For this purpose, we introduce a mapping $\Psi:\Lambda\rightarrow \mathcal{B}(\mathcal{Y},\widetilde{W})$ which satisfies the density condition \eqref{density_psi} with $W$ being replaced by $\widetilde{W}$. It follows for each $\alpha\in\Lambda$ and each $\xi\in\mathcal{Y}$ that
\begin{equation}\label{Bochner-integral1}
\int_{\mathcal{X}}\|\psi_{\alpha,\xi}\|_{W}d\mu(x)
\leq\|\xi\|_{\mathcal{Y}}\int_{\mathcal{X}}\|\Phi(x)
\|_{\mathcal{B}(\mathcal{Y},W)}
\|\lambda_{\alpha}(x)\|_{\mathcal{B}(\mathcal{Y},
\mathcal{Y})}d\mu(x).
\end{equation}
This together with (\ref{condition}) leads to  $\psi_{\alpha,\xi}$ is Bochner integrable. This allows
us to define the operator $\Psi:\Lambda\rightarrow
\mathcal{L}(\mathcal{Y},\widetilde{W})$ by
$$
\Psi(\alpha)\xi:=\int_{\mathcal{X}}\psi_{\alpha,\xi}(x)
d\mu(x),\ \ \alpha\in\Lambda,\ \ \xi\in\mathcal{Y}.
$$
By the definition of $\psi_{\alpha,\xi}$, the linearity
of $\Psi(\alpha)$ is clear. Combining equation (\ref{Bochner-integral1}) with the fact
$$
\|\Psi(\alpha)\xi\|_{\widetilde{W}}\leq\int_{\mathcal{X}}
\|\psi_{\alpha,\xi}\|_{W}d\mu(x),
$$
we conclude that $\Psi(\alpha)\in\mathcal{B}(\mathcal{Y},\widetilde{W})$.
Observing from the definition of $\Psi$, we obtain easily its density condition. By Lemma \ref{Projection}, the hypotheses of Theorem \ref{L_RKHS_Thm} are satisfied.
Consequently, by identifying for each $\alpha\in\Lambda$,
$$
L_{\alpha}(\widetilde{\Phi}(\cdot)^{*}u)
=\Psi(\alpha)^{*}u,\ u\in \widetilde{W},
$$
we conclude that the space
$\widetilde{\mathcal{H}}:=\left\{\widetilde{\Phi}
(\cdot)^{*}u:u\in\widetilde{W}\right\}$
with the inner product
$\langle\widetilde{\Phi}(\cdot)^{*}u, \widetilde{\Phi}(\cdot)^{*}v\rangle_{\widetilde{\mathcal{H}}}
:=\langle u,v\rangle_{\widetilde{W}}$
is a perfect ORKHS with respect to $\mathcal{L}:=\{L_{\alpha}:\alpha\in\Lambda\}$ and that the corresponding operator reproducing kernel
has the form $K(\alpha)=\widetilde{\Phi}(\cdot)^{*}
\Psi(\alpha), \alpha\in\Lambda$.

We next show that the perfect ORKHS $\widetilde{\mathcal{H}}$ has the form \eqref{ORKHS-integral-operator} and the operator reproducing kernel $K$ has the form (\ref{L_reproducing_kernel_integral}). Noting that
$$
\widetilde{\Phi}(x)^{*}
=\Phi(x)^{*}\mathcal{P}_{\widetilde{W}}^{*}
=\Phi(x)^{*}, \ \ \mbox{for all} \ \ x\in \mathcal{X},
$$
we can represent $\widetilde{\mathcal{H}}$ as in (\ref{ORKHS-integral-operator}) and $K$ as $K(\alpha)=\Phi(\cdot)^{*}\Psi(\alpha),\alpha\in\Lambda$. This together with the definition of $\Psi$ leads to the form (\ref{L_reproducing_kernel_integral}).
It suffices to rewrite $L_{\alpha},\alpha\in\Lambda,$
in the form (\ref{integral-operator}). It follows for
each $\alpha\in\Lambda$ and each
$f=\Phi(\cdot)^{*}u\in \widetilde{\mathcal{H}}$ that
$$
\langle L_{\alpha}(f),\xi\rangle_{\mathcal{Y}}
=\langle\Psi(\alpha)^{*}u,\xi\rangle_{\mathcal{Y}}
=\langle u,\Psi(\alpha)\xi\rangle_{W},
\ \xi\in\mathcal{Y}.
$$
Hence, by the definition of $\Psi$ and $\psi_{\alpha,\xi}$ we have that
\begin{eqnarray*}
\langle L_{\alpha}(f),\xi\rangle_{\mathcal{Y}}
=\int_{\mathcal{X}}\langle u,
\Phi(x)\lambda_{\alpha}(x)^{*}\xi\rangle_{W}d\mu(x)
=\left\langle\int_{\mathcal{X}}\lambda_{\alpha}(x)f(x)d\mu(x),
\xi\right\rangle_{\mathcal{Y}},\ \xi\in\mathcal{Y},
\end{eqnarray*}
which yields that $L_{\alpha}$ has the form (\ref{integral-operator}) and thus, completes the proof.
\end{proof}

Associated with the Hilbert space $W$ and the mapping $\Phi:\mathcal{X}\rightarrow \mathcal{B}(\mathcal{Y},W)$ in Theorem \ref{integral-perfectORKHS}, the space $\mathcal{H}:=\left\{\Phi(\cdot)^{*}u: u\in W\right\}$,
with the inner product $\langle\Phi(\cdot)^{*}u, \Phi(\cdot)^{*}v\rangle_{\mathcal{H}}
:=\langle u,v\rangle_{W}$, is an RKHS on $\mathcal{X}$.
The perfect ORKHS $\widetilde{\mathcal{H}} $ with the
form (\ref{ORKHS-integral-operator}) is a closed
subspace of $\mathcal{H}$. Observing from the representation of $\widetilde{\mathcal{H}} $,
we obtain the following result.
\begin{cor}\label{density-W}
Let $W$ be a Hilbert space and $\Phi:\mathcal{X}\rightarrow \mathcal{B}(\mathcal{Y},W)$  a mapping satisfying the density condition \eqref{density_phi}. If a collection of mappings $\lambda_{\alpha}:\mathcal{X}\to
\mathcal{B}(\mathcal{Y},\mathcal{Y})$,
$\alpha\in \Lambda,$ satisfies \eqref{condition} and $\widetilde{W}$ is defined by \eqref{W}, then the RKHS
$\mathcal{H}:=\left\{\Phi(\cdot)^{*}u:u\in W\right\}$
with the inner product $\langle\Phi(\cdot)^{*}u,
\Phi(\cdot)^{*}v\rangle_{\mathcal{H}}
:=\langle u,v\rangle_{W}$, is a perfect ORKHS with
respect to $\mathcal{L}:=\{L_{\alpha}:\alpha\in\Lambda\}$ defined as in \eqref{integral-operator} if and only if there holds $\widetilde{W}=W.$
\end{cor}
\begin{proof}
It follows from Theorem \ref{integral-perfectORKHS} that the closed subspace $\widetilde{\mathcal{H}}$, defined
by (\ref{ORKHS-integral-operator}), of $\mathcal{H}$ is
a perfect ORKHS with respect to $\mathcal{L}$. If $\widetilde{W}=W,$ then we have $\mathcal{H}=\widetilde{\mathcal{H}}$, which shows that $\mathcal{H}$ is just the perfect ORKHS. Conversely, we suppose that $\mathcal{H}$ is a perfect ORKHS with
respect to $\mathcal{L}$. For each $u\in W$, we set
\begin{equation*}\label{u-psi}
h_{\alpha,\xi}(u):=\langle u,\int_{\mathcal{X}}
\psi_{\alpha,\xi}(x)d\mu(x)\rangle_{W},
\ \alpha\in\Lambda,\xi\in\mathcal{Y}.
\end{equation*}
It suffices to verify that if $u\in W$ satisfies
$h_{\alpha,\xi}(u)=0$ for any $\alpha\in\Lambda,
\xi\in\mathcal{Y}$, then $u=0$. By the definition
of $\psi_{\alpha,\xi}$, we have for each $\alpha\in\Lambda,\xi\in\mathcal{Y}$ that
$$
h_{\alpha,\xi}(u)=\int_{\mathcal{X}}\langle u,\Phi(x)\lambda_{\alpha}(x)^{*}\xi\rangle_{W}d\mu(x)
=\langle\int_{\mathcal{X}}\lambda_{\alpha}(x)
\Phi(x)^{*}ud\mu(x),\xi\rangle_{\mathcal{Y}}.
$$
This together with the definition of $L_{\alpha},\alpha\in\Lambda,$ leads to
$h_{\alpha,\xi}(u)=\langle L_{\alpha}(\Phi(x)^{*}u)
,\xi\rangle_{\mathcal{Y}}.$ Since the norm of $\mathcal{H}$ is compatible with $\mathcal{L}$,
we have that $h_{\alpha,\xi}(u)=0$ holds for any $\alpha\in\Lambda,\xi\in\mathcal{Y}$ if and only if $\|\Phi(x)^{*}u\|_{\mathcal{H}}=0$, which is
equivalent to $u=0$.
\end{proof}

We now turn to presenting two specific examples of perfect
ORKHSs with respect to the family of integral operators
defined by (\ref{integral-operator}). These perfect ORKHSs
come from widely used RKHSs.

The first perfect ORKHS is constructed based on the
translation invariant RKHSs. Let $\mathcal{Y}$ be a
Hilbert space. A reproducing kernel $\mathcal{K}:
\mathbb{R}^d\times\mathbb{R}^d\to
\mathcal{B}(\mathcal{Y},\mathcal{Y})$
is said to be translation invariant if for all
$a\in\mathbb{R}^d$,  $\mathcal{K}(x-a,y-a)=\mathcal{K}(x,y)$,
for $x,y\in\mathbb{R}^d$. In the scalar-valued case,
a well-known characterization of translation invariant
kernels was established in \cite{B59}, which states that
every continuous translation invariant kernel on
$\mathbb{R}^d$ must be the Fourier transform of a finite
nonnegative Borel measure on $\mathbb{R}^d$, and vice versa. The characterization was generalized to the operator-valued case in \cite{F}. We consider a matrix-valued translation invariant reproducing kernel, for which the positive semi-definite matrix-valued Borel measure of bounded variation on $\mathbb{R}^d$ has the Radon-Nikodym property with respect to the Lebesgue measure. Specifically, we denote by $\mathbf{M}_n$ the space of $n\times n$ matrices endowed with the spectral norm and $\mathbf{M}_n^{+}$ its subspace of positive semi-definite matrices. Let $\varphi$ be a Bochner integrable function from $\mathbb{R}^d$ to
$\mathbf{M}_n^{+}$. The translation invariant kernel
$\mathcal{K}:\mathbb{R}^d\times\mathbb{R}^d\to
\mathbf{M}_n$ to be considered has the special form
\begin{equation}\label{translation_invariant_rk_phi}
\mathcal{K}(x,y)=\int_{\mathbb{R}^d}e^{i(x-y,t)}\varphi(t)dt,
\ \ x,y\in\mathbb{R}^d.
\end{equation}
We shall make use of Theorem \ref{integral-perfectORKHS} to construct a perfect ORKHS, which is a closed subspace of the translation invariant RKHS of $\mathcal{K}$. To this end, we first present the feature space and the feature map of $\mathcal{K}$. Associated with the $\mathbf{M}_n^{+}$-valued Bochner integrable function $\varphi$ on $\mathbb{R}^d$, we denote by $L^2_{\varphi}(\mathbb{R}^d,\mathbb{C}^n)$
the space of $\mathbb{C}^n$-valued Lebesgue measurable
functions $f$ on $\mathbb{R}^d$ satisfying $\varphi^{1/2}f$ belongs to $L^2(\mathbb{R}^d,\mathbb{C}^n)$. The space
$L^2_{\varphi}(\mathbb{R}^d,\mathbb{C}^n)$ is a
Hilbert space endowed with the inner product
$$
\langle f,g\rangle_{L^2_{\varphi}(\mathbb{R}^d,\mathbb{C}^n)}
:=\int_{\mathbb{R}^d}\langle\varphi(t)f(t),
g(t)\rangle_{\mathbb{C}^n}dt,
\ \ f,g\in L^2_{\varphi}(\mathbb{R}^d,\mathbb{C}^n).
$$
We choose $W:=L^2_{\varphi}(\mathbb{R}^d,\mathbb{C}^n)$
and $\Phi(x):=e^{i(x,\cdot)},\ x\in\mathbb{R}^d$.
We need to verify that $\Phi(x), x\in\mathbb{R}^d,$
belong to $\mathcal{B}(\mathbb{C}^n,W)$ and satisfy
the density condition \eqref{density_phi} with
$\mathcal{Y}=\mathbb{C}^n$. It follows for each
$x\in\mathbb{R}^d$ and each $\xi\in\mathbb{C}^n$ that
$$
\|\Phi(x)\xi\|^2_{W}
=\int_{\mathbb{R}^d}\langle\varphi(t)e^{i(x,t)}\xi,
e^{i(x,t)}\xi\rangle_{\mathbb{C}^n}dt
=\int_{\mathbb{R}^d}\langle\varphi(t)\xi,
\xi\rangle_{\mathbb{C}^n}dt
\leq\|\xi\|^2_{\mathbb{C}^n}\int_{\mathbb{R}^d}
\|\varphi(t)\|_{\mathbf{M}_n}dt.
$$
This together with the Bochner integrability of $\varphi$
yields that $\Phi(x)\xi\in W$ and $\Phi(x)$ is continuous
from $\mathbb{C}^n$ to $W$ with the norm satisfying
\begin{equation}\label{phinorm}
\|\Phi(x)\|_{\mathcal{B}(\mathbb{C}^n,W)}
\leq\|\varphi\|_{L^1(\mathbb{R}^d,\mathbf{M}_n)}^{1/2}.
\end{equation}
Moreover, it is clear that there holds the density condition (\ref{density_phi}) with $\mathcal{Y}=\mathbb{C}^n$. To describe the integral operators in this case, we introduce a sequence $u_{\alpha}=[u_{\alpha,j}:j\in\mathbb{N}_n],
\ \alpha\in\Lambda,$ of $\mathbb{C}^n$-valued Bochner
integrable functions on $\mathbb{R}^d$ and choose
\begin{equation}\label{lambda-alpha}
\lambda_{\alpha}(x):={\rm diag}(\overline{u_{\alpha,j}
(x)}:j\in\mathbb{N}_{n}),\ x\in\mathbb{R}^d,\ \alpha\in\Lambda.
\end{equation}
Then the integral operator (\ref{integral-operator})
reduces to
\begin{equation}\label{integral-operator1}
L_{\alpha}(f):=\int_{\mathbb{R}^d}f(x)\circ
\overline{u_{\alpha}(x)}dx,\ \ \alpha\in\Lambda.
\end{equation}
Here, $u\circ v$ denotes the Hadamard product of two
$\mathbb{C}^{n}$-valued functions $u$ and $v$,
which has been defined in (\ref{example_integration_functional}).
Associated with the Bochner integrable functions
$\varphi:\mathbb{R}^d\to \mathbf{M}_n^{+}$ and
$u_{\alpha}:\mathbb{R}^d\to\mathbb{C}^n,\ \alpha\in\Lambda,$ we define the inner product space
\begin{equation}\label{example-ORKHS1}
\mathcal{H}:=\left\{(\varphi(\cdot)^{*}u)^{\wedge}:
u\in\overline{\span}\left\{\check{u}_{\alpha}\circ\xi:
\ \alpha\in\Lambda,\xi\in\mathbb{C}^n\right\}\right\},
\end{equation}
with the inner product
\begin{equation}\label{example-inner-product}
\langle f,g \rangle_{\mathcal{H}}
:=\int_{\mathbb{R}^d}\langle\varphi(t)u(t),
v(t)\rangle_{\mathbb{C}^n}dt
\end{equation}
for $f=(\varphi(\cdot)^{*}u)^{\wedge}$
and $g=(\varphi(\cdot)^{*}v)^{\wedge}$.
The closure in \eqref{example-ORKHS1} is taken in $W$.

As a consequence of Theorem \ref{integral-perfectORKHS},
we get in the next theorem a construction of perfect
ORKHSs with respect to a family of integral operators
defined as in (\ref{integral-operator1}).

\begin{thm}\label{translation_invariant_ORKHS}
If $\varphi$ is a Bochner integrable function from
$\mathbb{R}^d$ to $\mathbf{M}_n^{+}$ and
$u_{\alpha}, \ \alpha\in\Lambda,$ is a family of $\mathbb{C}^n$-valued Bochner integrable functions, then the space $\mathcal{H}$ defined by \eqref{example-ORKHS1}
with the inner product \eqref{example-inner-product}
is a perfect ORKHS with respect to
$\mathcal{L}:=\{L_{\alpha}:\alpha\in\Lambda\}$
defined as in \eqref{integral-operator1} and
the operator reproducing kernel is determined by
$K(\alpha)\xi=(2\pi)^d[\varphi(\cdot)^{*}
(\check{u}_{\alpha}\circ\xi)]^{\wedge}$.
\end{thm}
\begin{proof}
In order to use the construction in Theorem
\ref{integral-perfectORKHS}, we first verify condition (\ref{condition}) with $\lambda_{\alpha}(x),
\ \alpha\in\Lambda,$ being defined by \eqref{lambda-alpha}.
According to inequality (\ref{phinorm}), we have
for each $\alpha\in\Lambda$ that
\begin{equation}\label{condition-verify}
\int_{\mathbb{R}^d}\|\Phi(x)\|_{\mathcal{B}(\mathbb{C}^n,W)}
\|\lambda_{\alpha}(x)\|_{\mathbf{M}_n}dx
\leq\|\varphi\|_{L^1(\mathbb{R}^d,\mathbf{M}_n)}^{1/2}
\int_{\mathbb{R}^d}\|\lambda_{\alpha}(x)\|_{\mathbf{M}_n}dx.
\end{equation}
By the definition of $\lambda_{\alpha}$, we get
$$
\int_{\mathbb{R}^d}\|\lambda_{\alpha}(x)\|_{\mathbf{M}_n}dx
\leq\|u_{\alpha}\|_{L^1(\mathbb{R}^d,\mathbb{C}^n)}.
$$
Substituting the above inequality into
(\ref{condition-verify}) and by the Bochner
integrability of $u_{\alpha}$, we have that
\begin{equation*}
\int_{\mathbb{R}^d}\|\Phi(x)\|_{\mathcal{B}(\mathbb{C}^n,W)}
\|\lambda_{\alpha}(x)\|_{\mathbf{M}_n}dx
\leq\|\varphi\|_{L^1(\mathbb{R}^d,\mathbf{M}_n)}^{1/2}
\|u_{\alpha}\|_{L^1(\mathbb{R}^d,\mathbb{C}^n)}
<\infty.
\end{equation*}
Hence, by Theorem \ref{integral-perfectORKHS} we obtain
a perfect ORKHS $\widetilde{\mathcal{H}}$ with the form
(\ref{ORKHS-integral-operator}) and the corresponding
operator reproducing kernel $K$ with the form
(\ref{L_reproducing_kernel_integral}).

We next turn to representing $\widetilde{\mathcal{H}}$
and $K$ by making use of the functions $\varphi$ and
$u_{\alpha}, \alpha\in \Lambda$. It follows for each
$\alpha\in\Lambda$ and each $\xi\in\mathbb{C}^n$ that
$$
\int_{\mathbb{R}^d}\Phi(x)\lambda_{\alpha}(x)^{*}\xi dx
=\int_{\mathbb{R}^d}\Phi(x)(u_{\alpha}(x)\circ\xi)dx
=(2\pi)^d\check{u}_{\alpha}\circ\xi.
$$
Then the closed subspace $\widetilde{W}$ in
(\ref{ORKHS-integral-operator}) of $W$ can be
represented as
\begin{equation}\label{concrete-W}
\widetilde{W}=\overline{\span}\left\{\check{u}_{\alpha}
\circ\xi:\ \alpha\in\Lambda, \xi\in\mathbb{C}^n
\right\}.
\end{equation}
There holds for each $x\in\mathbb{R}^d$ and each
$u\in \widetilde{W}$ that
$$
\langle\Phi(x)\xi,u\rangle_{\widetilde{W}}
=\int_{\mathbb{R}^d}\langle\varphi(t)e^{i(x,t)}\xi,
u(t)\rangle_{\mathbb{C}^n}dt
=\left\langle\xi,\int_{\mathbb{R}^d}e^{-i(x,t)}
\varphi(t)^{*}u(t)dt\right\rangle_{\mathbb{C}^n},
\ \xi\in\mathbb{C}^n,
$$
which leads to
$$
\Phi(x)^{*}u
=\int_{\mathbb{R}^d}e^{-i(x,t)}\varphi(t)^{*}u(t) dt
=(\varphi(\cdot)^{*}u)^{\wedge}(x).
$$
Substituting the representation (\ref{concrete-W})
of $\widetilde{W}$ and the above equation into
(\ref{ORKHS-integral-operator}), we have that
$\widetilde{\mathcal{H}}=\mathcal{H}$ and the
inner product of $\widetilde{\mathcal{H}}$ can
be represented as in (\ref{example-inner-product}).
Moreover, the operator reproducing kernel $K$ can be
represented as
$$
K(\alpha)\xi:=\Phi(\cdot)^{*}
\int_{\mathbb{R}^d}\Phi(x)\lambda_{\alpha}(x)^{*}\xi dx
=(2\pi)^d[\varphi(\cdot)^{*}(\check{u}_{\alpha}\circ\xi)
]^{\wedge},\ \alpha\in\Lambda,\ \xi\in \mathbb{C}^n.
$$
\end{proof}

The widely used Gaussian kernel on $\mathbb{R}^d$ is a translation invariant kernel with $\varphi$ being the Gaussian function. A matrix-valued Gaussian kernel on $\mathbb{R}^d$ discussed in \cite{MP05} is determined by
\begin{equation}\label{matrix-valued-Gaussian-kernel}
\mathcal{K}(x,y):=\sum_{j\in\mathbb{N}_m}e^{-\sigma_j\|x-y\|^2}
\mathbf{A}_j,
\end{equation}
where $\sigma_j\geq0,\ j\in\mathbb{N}_m$ and $\{\mathbf{A}_j:j\in\mathbb{N}_m\}$ is a sequence of real matrices in $\mathbf{M}_n^{+}$. Here, we denote by $\|\cdot\|$ the norm on $\mathbb{R}^d$ induced by the standard inner product. By letting
$$
\varphi(t):=\sum_{j\in\mathbb{N}_m}
\frac{1}{(2\sigma_j)^{d/2}}e^{-\frac{1}{4\sigma_j}\|t\|^2}
\mathbf{A}_j,
$$
we can represent $\mathcal{K}$ in the form (\ref{translation_invariant_rk_phi}). Combining Corollary \ref{density-W} with Theorem \ref{translation_invariant_ORKHS}, we conclude that the RKHS of the Gaussian kernel (\ref{matrix-valued-Gaussian-kernel}) is  a perfect ORKHS with respect to
$\mathcal{L}:=\{L_{\alpha}:\alpha\in\Lambda\}$
defined as in \eqref{integral-operator1} in terms of a family of $\mathbb{C}^n$-valued Bochner integrable functions $u_{\alpha}, \alpha\in\Lambda, $ if the linear span of $\left\{\check{u}_{\alpha}\circ\xi:
\ \alpha\in\Lambda,\xi\in\mathbb{C}^n\right\}$ is dense in $L^2_{\varphi}(\mathbb{R}^d,\mathbb{C}^n)$. Accordingly,
the operator reproducing kernel is given by
$$
(K(\alpha)\xi)(x)=\overline{L_{\alpha}(\mathcal{K}(x,\cdot)
\overline{\xi})}.
$$

We next present a perfect FRKHS, which is a closed subspace of a shift-invariant space. To this end, we
first recall the shift-invariant spaces, which have
been extensively studied in the areas of wavelet
theory \cite{D,GLT}, approximation theory
\cite{DDR,J} and sampling \cite{AG,AST,H07}. Given
$\phi\in L^2(\mathbb{R}^d)$ that satisfies
\begin{equation}\label{riesz_basis}
m\leq\sum_{j\in\mathbb{Z}^d}|\hat{\phi}(\xi+2\pi j)|^2
\leq M,\ \ \mbox{for almost every} \ \xi,
\end{equation}
for some $m,M>0$, we construct the shift-invariant
space by
$$
V^2(\phi):=\left\{\sum_{k\in\mathbb{Z}^d}c_k\phi(\cdot-k):
\ \{c_k:k\in\mathbb{Z}^d\}\in l^2(\mathbb{Z}^d)\right\}.
$$
Here, $l^2(\mathbb{Z}^d)$ is the Hilbert space of
square-summable sequences on $\mathbb{Z}^d$.
The function $\phi$ is called the generator of the
shift-invariant space $V^2(\phi)$. Clearly, $V^2(\phi)$
is a closed subspace of $L^2(\mathbb{R}^d)$ and
condition (\ref{riesz_basis}) ensures that the sequence
$\phi(\cdot-k),\ k\in\mathbb{Z}^d,$ constitutes
a Riesz basis of $V^2(\phi)$. To guarantee that
$V^2(\phi)$ is an RKHS, additional conditions need to be
imposed on the generator $\phi$. Following \cite{AG},
we first introduce a weight function $\omega$ on
$\mathbb{R}^d$, which is assumed to be continuous,
symmetric and positive, and to satisfy the condition
$0<\omega(x+y)\leq\omega(x)\omega(y)$, for any
$x,y\in\mathbb{R}^d$. Moreover, we assume that $\omega$
satisfies the growth condition
$$
\sum_{n=1}^{\infty}\frac{\log \omega(nk)}{n^2}<\infty,
\ \ \mbox{for all}\ \ k\in\mathbb{Z}^d.
$$
We also recall the Wiener amalgam spaces $W(L_w^p),
1\leq p\leq\infty$. For $1\leq p<\infty$, a measurable
function $f$ is said to belong to $W(L_w^p)$, if it satisfies
$$
\|f\|_{W(L_w^p)}^p:=\sum_{k\in\mathbb{Z}^d}
{\rm ess \ sup}\{|f(x+k)|^p\omega(k)^p:x\in[0,1]^d\}
<\infty.
$$
For $p=\infty$, a measurable function $f$ is said to belong to $W(L_w^{\infty})$ if it satisfies
$$
\|f\|_{W(L_w^{\infty})}:=\sup_{k\in\mathbb{Z}^d}
\{{\rm ess \  sup}\{|f(x+k)|\omega(k):x\in[0,1]^d\}\}
<\infty.
$$
We denote by $W_0(L_w^p)$ the subspace of continuous functions in $W(L_w^p)$. It has been proved in \cite{AG} that if $\phi\in W_0(L_w^1)$, the space $V^2(\phi)$ is an RKHS with the reproducing kernel
\begin{equation}\label{shift_invariant_rk}
\mathcal{K}(x,y)=\sum_{k\in\mathbb{Z}^d}\overline{\phi(x-k)}
\widetilde{\phi}(y-k),\ \ x,y\in\mathbb{R}^d,
\end{equation}
where $\widetilde{\phi}\in V^2(\phi)$ satisfies the condition that the sequence $\widetilde{\phi}(\cdot-k)$,
$k\in\mathbb{Z}^d$, is the dual Riesz basis of
$\phi(\cdot-k)$, $k\in\mathbb{Z}^d.$ To see the
reproducing property, for each
$$
f:=\sum_{j\in\mathbb{Z}^d}c_j\phi(\cdot-j)\in V^2(\phi),
$$
we note that
$$
\langle f,\mathcal{K}(x,\cdot)\rangle_{L^2(\mathbb{R}^d)}
=\sum_{j\in\mathbb{Z}^d}\sum_{k\in\mathbb{Z}^d}
c_j\phi(x-k)\langle\phi(\cdot-j),\widetilde{\phi}(\cdot-k)
\rangle_{L^2(\mathbb{R}^d)}.
$$
Due to the biorthogonality relation
$$
\langle\phi(\cdot-j),\widetilde{\phi}(\cdot-k)
\rangle_{L^2(\mathbb{R}^d)}=\delta_{j,k},
\ j,k\in\mathbb{Z}^d,
$$
the above equation reduces to
$$
\langle f,\mathcal{K}(x,\cdot)\rangle_{L^2(\mathbb{R}^d)}
=\sum_{j\in\mathbb{Z}^d}c_j\phi(x-j)=f(x).
$$
In order to use the construction
described in Theorem \ref{integral-perfectORKHS},
we choose a natural feature map representation for
the reproducing kernel $\mathcal{K}$. That is,
$W:=V^2(\phi)$ and $\Phi(x):=\mathcal{K}(x,\cdot),
\ x\in\mathbb{R}^d$. Moreover, we introduce a family
of measurable functions $u_{\alpha}, \alpha\in\Lambda$,
on $\mathbb{R}^d$ satisfying for each $\alpha\in\Lambda$,
\begin{equation}\label{condition_shift_invariant}
\int_{\mathbb{R}^d}|u_{\alpha}(x)|
\left(\sum_{k\in\mathbb{Z}^d}|\phi(x-k)|^2\right)^{1/2}dx
<\infty.
\end{equation}
By choosing
$\lambda_{\alpha}(x):=\overline{u_{\alpha}(x)}$, the integral operator in (\ref{integral-operator}) reduces to
\begin{equation}\label{integral-operator2}
L_{\alpha}(f):=\int_{\mathbb{R}^d}
f(x)\overline{u_{\alpha}(x)}dx,
\ \ \alpha\in\Lambda.
\end{equation}
Associated with $u_{\alpha}, \alpha\in\Lambda$, we define
the closed subspace of $V^2(\phi)$ as
\begin{equation}\label{phi_H_K_L}
\mathcal{H}:=\overline{\span}\left\{\sum_{k\in\mathbb{Z}^d}
\left(\int_{\mathbb{R}^d}u_{\alpha}(t)\overline{\phi(t-k)}dt
\right)\tilde{\phi}(\cdot-k):\alpha\in\Lambda\right\}.
\end{equation}

We construct in the next theorem a perfect FRKHS with respect to a family of integral functionals defined as in
(\ref{integral-operator2}).
\begin{thm}
If $\phi\in W_0(L_w^1)$ satisfies \eqref{riesz_basis}
and a family of measurable functions
$u_{\alpha}, \alpha\in\Lambda,$ on $\mathbb{R}^d$ satisfies \eqref{condition_shift_invariant},
then the closed subspace $\mathcal{H}$ defined by
\eqref{phi_H_K_L} of $V^2(\phi)$ is a perfect FRKHS with respect to a family
$\mathcal{F}:=\{\mathcal{L}_{\alpha}: \alpha\in\Lambda\}$ of integral functionals defined as in \eqref{integral-operator2}
and the corresponding functional reproducing kernel is
determined by
\begin{equation}\label{phi_K_L}
K(\alpha)(x):=\sum_{k\in\mathbb{Z}^d}
\left(\int_{\mathbb{R}^d}u_{\alpha}(t)\overline{\phi(t-k)}dt
\right)\tilde{\phi}(x-k),
\ \alpha\in\Lambda,\ x\in\mathbb{R}^d.
\end{equation}
\end{thm}
\begin{proof}
To apply Theorem \ref{integral-perfectORKHS} in the case that $\mathcal{Y}=\mathbb{C}$, we need to verify for each
$\alpha\in\Lambda$ that $u_{\alpha}\Phi$, as a $W$-valued
function, is Bochner integrable. By the reproducing property, we obtain for any $x\in\mathbb{R}^d$ that
\begin{equation*}
\|\Phi(x)\|_W^2=\mathcal{K}(x,x)=\sum_{k\in\mathbb{Z}^d}
\overline{\phi(x-k)}\widetilde{\phi}(x-k),
\end{equation*}
which leads to for each $\alpha\in\Lambda$,
\begin{equation}\label{equality1}
\int_{\mathbb{R}^d}\|u_{\alpha}(x)\Phi(x)\|_Wdx
=\int_{\mathbb{R}^d}|u_{\alpha}(x)|
\left(\sum_{k\in\mathbb{Z}^d}\overline{\phi(x-k)}
\widetilde{\phi}(x-k)\right)^{1/2}dx.
\end{equation}
For each $x\in\mathbb{R}^d$, we introduce two
sequences $a:=\{a_k:k\in\mathbb{Z}^d\}$ and
$\tilde{a}:=\{\tilde{a}_k:k\in\mathbb{Z}^d\}$ with
$a_k:=\phi(x+k)$ and $\tilde{a}_k:=\tilde{\phi}(x+k)$.
Note that the dual generator $\tilde{\phi}\in V^2(\phi)$
has the expansion
$\tilde{\phi}=\sum_{k\in\mathbb{Z}^d}b_k\phi(\cdot-k),$ where the coefficients $b_k, k\in\mathbb{Z}^d,$ are the Fourier coefficients of the $2\pi$-periodic function
$$
b(\xi)=\left(\sum_{j\in\mathbb{Z}^d}|\hat{\phi}
(\xi+2\pi j)|^2\right)^{-1},\ \xi\in [-\pi,\pi]^d.
$$
Set $b:=\{b_k:k\in\mathbb{Z}^d\}$. It follows for each
$k\in\mathbb{Z}^d$ that
$$
\tilde{a}_k
=\sum_{l\in\mathbb{Z}^d}b_l\phi(x+k-l)
=\sum_{l\in\mathbb{Z}^d}b_la_{k-l}
=(a\ast b)_k.
$$
Since $\phi\in W_0(L_w^1)$, we get that
$a\in l^2(\mathbb{Z}^d)$. By Lemmas 2.8 and 2.11 in \cite{AG}, we also have that
$b\in l^1(\mathbb{Z}^d):=\left\{\{c_k:k\in\mathbb{Z}^d\}:
\sum_{k\in\mathbb{Z}^d}|c_k|<+\infty\right\}$.
Hence, the above equation leads to
$\tilde{a}\in l^2(\mathbb{Z}^d)$ and
$$
\|\tilde{a}\|_{l^2(\mathbb{Z}^d)}\leq
\|a\|_{l^2(\mathbb{Z}^d)}\|b\|_{l^1(\mathbb{Z}^d)}.
$$
Then by (\ref{equality1}), we obtain that
$$
\int_{\mathbb{R}^d}\|u_{\alpha}(x)\Phi(x)\|_Wdx
\leq\|b\|_{l^1(\mathbb{Z}^d)}^{1/2}
\int_{\mathbb{R}^d}|u_{\alpha}(x)|\left(\sum_{k\in
\mathbb{Z}^d}|\phi(x-k)|^2\right)^{1/2}dx
<+\infty,
$$
which together with \eqref{condition_shift_invariant} shows that the function $u_{\alpha}\Phi$ is Bochner integrable.

By Theorem \ref{integral-perfectORKHS} in the case that
$\mathcal{Y}=\mathbb{C}$, we get the perfect FRKHS as a closed subspace of $V^2(\phi)$
\begin{equation*}
\widetilde{\mathcal{H}}
:=\overline{\span}\left\{\int_{\mathbb{R}^d}
u_{\alpha}(x)\Phi(x)dx:\alpha\in\Lambda\right\}
\end{equation*}
with the functional reproducing kernel
$$
K(\alpha):=\int_{\mathbb{R}^d}u_{\alpha}(x)\Phi(x)dx.
$$
According to (\ref{shift_invariant_rk}) we get
$\widetilde{\mathcal{H}}=\mathcal{H}$ and $K$
represented as in (\ref{phi_K_L}).
\end{proof}

To end this example, we shall give a sufficient and
necessary condition which ensures that the FRKHS
$\mathcal{H}$ appearing in (\ref{phi_H_K_L}) is
identical to $V^2(\phi)$, that is,
\begin{equation}\label{density_V2}
\overline{\span}\left\{\sum_{k\in\mathbb{Z}^d}
\left(\int_{\mathbb{R}^d}u_{\alpha}(t)\overline{\phi(t-k)}dt
\right)\tilde{\phi}(\cdot-k):\alpha\in\Lambda\right\}
=V^2(\phi).
\end{equation}
For this purpose, we introduce a sequence of functions by
letting for each $\alpha\in\Lambda$
$$
g_\alpha(\xi):=\sum_{l\in\mathbb{Z}^d}\widehat{u_{\alpha}}
(\xi+2l\pi)\overline{\hat{\phi}(\xi+2l\pi)},
\ \xi\in[-\pi,\pi]^d.
$$
\begin{thm}
If $\phi\in W_0(L_w^1)$ satisfies \eqref{riesz_basis}
and for each $\alpha\in \Lambda$,
$u_{\alpha}\in L^2(\mathbb{R}^d)$ satisfies
\eqref{condition_shift_invariant}, then the density condition \eqref{density_V2} holds if and only if there holds
\begin{equation}\label{density_V2_1}
\overline{\span}\{g_\alpha:\ \alpha\in\Lambda\}
=L^2([-\pi,\pi]^d).
\end{equation}
\end{thm}
\begin{proof}
By the definition of $g_\alpha, \alpha\in\Lambda,$ we have for each $\alpha\in\Lambda$ and $k\in\mathbb{Z}^d$ that
\begin{eqnarray*}
\int_{\mathbb{R}^d}u_{\alpha}(x)\overline{\phi(x-k)}dx
&=&\int_{\mathbb{R}^d}\widehat{u_{\alpha}}(\xi)
\overline{\hat{\phi}(\xi)}e^{i(k,\xi)}d\xi\\
&=&\sum_{l\in\mathbb{Z}^d}\int_{[-\pi,\pi]^d}
\widehat{u_{\alpha}}(\xi+2l\pi)
\overline{\hat{\phi}(\xi+2l\pi)}e^{i(k,\xi)}d\xi\\
&=&\int_{[-\pi,\pi]^d}g_\alpha(\xi)e^{i(k,\xi)}d\xi
=c_k(g_\alpha),
\end{eqnarray*}
where we denote by $c_k(h)$ the $k$-th Fourier coefficients of $h\in L^2([-\pi,\pi]^d)$. According to the representation of the functions in $V^2(\phi)$, we note that the density condition (\ref{density_V2}) holds if and only if there holds
$$
\overline{\span}\{\{c_k(g_\alpha):k\in\mathbb{Z}^d\}:
\alpha\in\Lambda\}=l^2(\mathbb{Z}^d),
$$
which is equivalent to (\ref{density_V2_1}).
\end{proof}

A functional reproducing kernel $K$ for a perfect FRKHS
$\mathcal{H}$ can be represented by the classical
reproducing kernel $\mathcal{K}$ for $\mathcal{H}$ by
\eqref{relation-kernels}. Especially with respect to the
the family $\mathcal{F}$ of the integral functionals
\begin{equation}\label{integralfun}
L_{\alpha}(f):=\int_{\mathcal{X}}f(x)\overline{u_{\alpha}(x)}
d\mu(x),\ \alpha\in\Lambda,
\end{equation}
the functional reproducing kernel $K$ has the form
\begin{equation}\label{phi_suffi_nessi}
K(\alpha)(x):=\int_{\mathcal{X}}\mathcal{K}(t,x)
u_{\alpha}(t)d\mu(t).
\end{equation}
To end this section, we shall show that under a somewhat
stronger hypothesis, the operator $K$ with the form
(\ref{phi_suffi_nessi}) is a functional reproducing kernel
if and only if $\mathcal{K}$ is a classical reproducing
kernel on $\mathcal{X}$.

\begin{prop}\label{sufficient_necessary_1}
Let $\mathcal{X}$ be a compact set, $\mu$ a finite Borel
measure on $\mathcal{X}$ and $\mathcal{K}:\mathcal{X}
\times\mathcal{X}\rightarrow \mathbb{C}$ a continuous function. If $\{u_{\alpha}:\alpha\in\Lambda\}$
satisfies for each $\alpha\in \Lambda$,
$u_{\alpha}\in L^1(\mathcal{X},\mu)$ and
$\overline{\span}\{u_{\alpha}:\alpha\in\Lambda\}
=L^1(\mathcal{X},\mu)$, then the operator $K$ defined by
\eqref{phi_suffi_nessi} is a functional reproducing kernel
with respect to a family $\mathcal{F}$ of linear functionals $L_{\alpha}, \alpha\in\Lambda$ defined by
\eqref{integralfun} if and only if $\mathcal{K}$ is a
reproducing kernel on $\mathcal{X}$.
\end{prop}
\begin{proof}
We suppose that $\mathcal{K}$ is a reproducing kernel
on $\mathcal{X}$ and $\mathcal{H}$ is the RKHS of
$\mathcal{K}$. To show that $K$, defined by
(\ref{phi_suffi_nessi}), is a functional reproducing kernel, we first introduce the feature space $W:=\mathcal{H}$ and the feature map $\Phi(x):=\mathcal{K}(x,\cdot), x\in\mathcal{X}$
of $\mathcal{K}$ and set $\lambda_{\alpha}(x):=\overline{u_{\alpha}(x)}, x\in\mathcal{X},\alpha\in\Lambda$.
We then verify the assumption (\ref{condition}) in Theorem
\ref{integral-perfectORKHS}. Based on the above notations,
we have for each $\alpha\in\Lambda$ that
$$
\int_{\mathcal{X}}|\lambda_{\alpha}(x)|
\|\Phi(x)\|_{\mathcal{H}}d\mu(x)
=\int_{\mathcal{X}}|u_{\alpha}(x)|
\sqrt{\mathcal{K}(x,x)}d\mu(x)
\leq\max_{x\in \mathcal{X}}\{\sqrt{\mathcal{K}(x,x)}\}
\int_{\mathcal{X}}|u_{\alpha}(x)|d\mu(x)
<+\infty.
$$
Hence, by Theorem \ref{integral-perfectORKHS}, we get
a functional reproducing kernel
\begin{equation*}
K(\alpha)(x)=\left\langle\int_{\mathcal{X}}u_{\alpha}(t)
\Phi(t)d\mu(t),\Phi(x)\right\rangle_{W},
\ \alpha\in\Lambda, \ x\in \mathcal{X},
\end{equation*}
with respect to a family $\mathcal{F}$ of linear functionals $L_{\alpha}, \alpha\in\Lambda$ defined by \eqref{integralfun}. By the definition of $\Phi$ and the reproducing property of $\mathcal{K}$, we conclude that $K$ has the form (\ref{phi_suffi_nessi}).

Conversely, we suppose that $K$ with the form
(\ref{phi_suffi_nessi}) is a functional reproducing
kernel with respect to $\mathcal{F}$ and show that
$\mathcal{K}$ is a classical reproducing kernel on
$\mathcal{X}$. We note from \cite{B41} that $\mathcal{K}$
is a classical reproducing kernel on $\mathcal{X}$ if and only if there holds for all $g\in L^1(\mathcal{X},\mu)$,
\begin{equation}\label{ifonlyif}
\int_{\mathcal{X}}\int_{\mathcal{X}}
g(s)\mathcal{K}(s,t)\overline{g(t)}d\mu(s)d\mu(t)
\geq 0.
\end{equation}
Hence, it suffices to verify that (\ref{ifonlyif}) holds true for all $g\in L^1(\mathcal{X},\mu)$. Observing
from the definition of the functional reproducing kernel
and the integral functionals, we have for any
$u=\sum_{j\in\mathbb{N}_n}c_ju_{\alpha_j}
\in\span\{u_{\alpha}:\alpha\in \Lambda\}$ that
$$
\int_{\mathcal{X}}\int_{\mathcal{X}}
u(s)\mathcal{K}(s,t)\overline{u(t)}d\mu(s)d\mu(t)
=\sum_{j\in\mathbb{N}_n}\sum_{k\in\mathbb{N}_n}
c_j\overline{c_k}L_{\alpha_k}(K(\alpha_j)).
$$
Since $K$ is a functional reproducing kernel with respect to $\mathcal{F}$, we get by the above equation that
$$
\int_{\mathcal{X}}\int_{\mathcal{X}}
u(s)\mathcal{K}(s,t)\overline{u(t)}d\mu(s)d\mu(t)
\geq 0
$$
for any $u\in\span\{u_{\alpha}:\alpha\in\Lambda\}$.
By the density of $u_{\alpha},\alpha\in\Lambda,$ in
$L^1(\mathcal{X},\mu)$, there exists for each
$g\in L^1(\mathcal{X},\mu)$ a sequence
$u_n\in \span\{u_{\alpha}:\alpha\in\Lambda\}$ such that
$$
\lim_{n\rightarrow \infty}
\|u_n-g\|_{L^1(\mathcal{X},\mu)}=0.
$$
By the triangular inequality there holds
\begin{eqnarray*}
&&\left|\int_{\mathcal{X}}\int_{\mathcal{X}}
\mathcal{K}(s,t)(g(s)\overline{g(t)}-u_n(s)
\overline{u_n(t)})d\mu(s)d\mu(t)\right|\\
&\leq&\max_{s,t\in X}\{|\mathcal{K}(s,t)|\}
\left(\|g\|_{L^1(\mathcal{X},\mu)}+
\|u_n\|_{L^1(\mathcal{X},\mu)}\right)
\|u_n-g\|_{L^1(\mathcal{X},\mu)},
\end{eqnarray*}
which implies that
$$
\int_{\mathcal{X}}\int_{\mathcal{X}}
g(s)\mathcal{K}(s,t)\overline{g(t)}d\mu(s)d\mu(t)
=\lim_{n\rightarrow\infty}\int_{\mathcal{X}}
\int_{\mathcal{X}}u_n(s)\mathcal{K}(s,t)
\overline{u_n(t)}d\mu(s)d\mu(t).
$$
This together with the inequalities
$$
\int_{\mathcal{X}}\int_{\mathcal{X}}
u_n(s)\mathcal{K}(s,t)\overline{u_n(t)}d\mu(s)d\mu(t)
\geq 0, \ \ \mbox{for all} \ \ n,
$$
ensures that (\ref{ifonlyif}) holds for each
$g\in L^1(\mathcal{X},\mu)$. Therefore, we have that
$\mathcal{K}$ is a reproducing kernel on $\mathcal{X}$.
\end{proof}

\section{Sampling and Reconstruction in FRKHSs}

We describe in this section a reconstruction of an element in an FRKHS from its functional values. This study is motivated by the average sampling, which aims at recovering a function from its local average values represented as integral functional values. As pointed out in the classical sampling theory in RKHSs, the local average functionals used in the reconstruction should be continuous to ensure the stability of the sampling process. In practice, only finite local average functionals are used to establish a reconstruction of a function. However, to obtain a more precise approximation of the target function, more local average functionals are demanded. Hence, the average sampling should be considered in a space which ensures the continuity of the local average functionals with respect to all the points in the domain. Such a space is just an FRKHS with respect to the family of all the local average functionals. We establish the sampling theorem in the context of general FRKHSs and present explicit examples concerning the Paley-Wiener spaces. The average sampling has been widely studied in classical RKHSs, such as the Paley-Wiener spaces and the shift-invariant spaces without the availability of FRKHSs \cite{A,S,SZ}. We point out that these spaces are FRKHSs with respect to the local average functionals and the existing sampling theorems are special cases of the main results to be presented in this section.

The general sampling and reconstruction problem in an ORKHS may be reformulated as it in a corresponding FRKHS.
Let $\mathcal{H}$ be an FRKHS with respect to a family
$\mathcal{F}:=\{L_{\alpha}:\alpha\in\Lambda\}$ of
linear functionals on $\mathcal{H}$. We shall
consider the complete reconstruction of an
element $f\in\mathcal{H}$ from its sampled data
$\{L_{\alpha_j}(f):j\in\mathbb{J}\}$, where
$\mathbb{J}$ is a countable index set and
$\{\alpha_j:j\in\mathbb{J}\}\subseteq \Lambda$.
The general sampling and reconstruction problem in
an ORKHS $\mathcal{H}$ with respect to a family
$\mathcal{L}$ of linear operators
$L_{\alpha},\ \alpha\in\Lambda,$ from $\mathcal{H}$ to $\mathcal{Y}$, where the sampled data have the form
$\{\langle L_{\alpha_j}(f),\xi_j\rangle_{\mathcal{Y}}:
j\in\mathbb{J}\},$ with $\{\alpha_j:j\in\mathbb{J}\}\subseteq\Lambda$ and $\{\xi_j:j\in\mathbb{J}\}\subseteq\mathcal{Y}$,
can be reformulated as the sampling and
reconstruction problem in an FRKHS. In fact, if we introduce the family $\mathcal{F}$ of linear functionals $\widetilde{L}_{(\alpha,\xi)},\ (\alpha,\xi)
\in\Lambda\times\mathcal{Y},$ as
$$
\widetilde{L}_{(\alpha,\xi)}(f)
=\langle L_{\alpha}(f),\xi\rangle_{\mathcal{Y}},
\ (\alpha,\xi)\in\Lambda\times\mathcal{Y},
$$
space $\mathcal{H}$ is also an FRKHS with respect to $\mathcal{F}$. From this observation, the above
sampling problem in an ORKHS can also be interpreted
as reconstructing an element $f$ in an FRKHS $\mathcal{H}$ from its functionals values
$\widetilde{L}_{(\alpha_j,\xi_j)}(f), \ j\in\mathbb{J}$. Hence, we shall consider sampling and reconstruction in FRKHSs only.

Complete reconstruction of a function requires the availability of a frame or a Riesz basis of a Hilbert space. We review the concept of frames and Riesz bases of a Hilbert space according to \cite{D,DS}. Let $\mathcal{H}$ be a separable Hilbert space with the inner product $\langle\cdot,\cdot\rangle_{\mathcal{H}}$ and let $\mathbb{J}$ be a countable index set. A sequence
$\{f_j: j\in\mathbb{J}\}\subset\mathcal{H}$ is called a frame of $\mathcal{H}$ if there exist constants $0<A\leq B<+\infty$ such that
$$
A\|f\|_{\mathcal{H}}
\leq \left(\sum_{j\in\mathbb{J}}
|\langle f,f_j\rangle_{\mathcal{H}}|^2\right)^{1/2}
\leq B\|f\|_{\mathcal{H}},
\ \mbox{for all}\ f\in\mathcal{H}.
$$
For a frame $\{f_j: j\in\mathbb{J}\}$ of
$\mathcal{H}$, there exists a dual frame
$\{g_j: j\in\mathbb{J}\}\subset\mathcal{H}$, for which
\begin{equation}\label{reconstruction-formula-H}
f=\sum_{j\in\mathbb{J}}\langle f,
g_j\rangle_{\mathcal{H}}f_j,
\ \mbox{for each}\ f\in\mathcal{H}.
\end{equation}
A frame $\{f_j: j\in\mathbb{J}\}$ is a Riesz basis of $\mathcal{H}$ if it is minimal, that is,
$$
f_j\notin\overline{\span}\{f_k:k\in\mathbb{J},k\neq j\},
\ \mbox{for each}\ j\in\mathbb{J}.
$$
If $\{f_j: j\in\mathbb{J}\}$ is a Riesz basis of $\mathcal{H}$, there exists a unique sequence
$\{g_j: j\in\mathbb{J}\}\subset \mathcal{H}$ satisfying (\ref{reconstruction-formula-H}) and we call it the dual Riesz basis of the original Riesz basis. Moreover,
the dual Riesz basis $\{g_j: j\in\mathbb{J}\}$
satisfies the biorthogonal condition that
$\langle f_j,g_k\rangle_{\mathcal{H}}=\delta_{jk}$,
for all $j,k\in\mathbb{J}$, where $\delta_{j,k}$ is the Kronecker delta notation.

We now consider reconstructing an element in an FRKHS
in terms of its sampled functional values.
Let $\mathcal{H}$ be an FRKHS with respect to a
family $\mathcal{F}$ of linear functionals
$L_{\alpha},\ \alpha\in \Lambda$ and $K$ the functional reproducing kernel for $\mathcal{H}$. As far as sampling is concerned, we assume that there exists a countable sequence $\alpha_j\in \Lambda, j\in\mathbb{J}$, such that
$\mathbb{K}:=\{K(\alpha_j):\ j\in\mathbb{J}\}$
constitutes a frame for $\mathcal{H}$. We need to
construct a dual frame of $\mathbb{K}$. To this end,
we define the frame operator
$T: \mathcal{H}\rightarrow \mathcal{H}$ for each
$f\in \mathcal{H}$ by
\begin{equation}\label{frame_operator}
Tf:=\sum_{j\in\mathbb{J}}\langle f,
K(\alpha_j)\rangle_{\mathcal{H}} K(\alpha_j).
\end{equation}
It is well-known that operator $T$ is bounded, invertible, self-adjoint and strictly positive on $\mathcal{H}$ and the sequence $\{G_j: j\in\mathbb{J}\}$ with $G_j:=T^{-1}(K(\alpha_j))$ is the dual frame of $\mathbb{K}$. We shall show that the dual frame
$\{G_j: j\in\mathbb{J}\}$ is also induced from a functional reproducing kernel. For this purpose,
we define operator $\widetilde{K}$ from $\Lambda$ to $\mathcal{H}$ by
\begin{equation}\label{dualkernel}
\widetilde{K}(\alpha):=T^{-1}(K(\alpha)),
\ \mbox{for each} \ \alpha\in \Lambda
\end{equation}
and introduce the space
$\widetilde{\mathcal{H}}:=\{f:f\in \mathcal{H}\}$
with the bilinear mapping
$\langle \cdot,\cdot\rangle_{\widetilde{\mathcal{H}}}$ defined by
\begin{equation}\label{dualinner_product}
\langle f,g\rangle_{\widetilde{\mathcal{H}}}
:=\langle Tf,g\rangle_{\mathcal{H}},
\ f,g\in\widetilde{\mathcal{H}}.
\end{equation}
Since operator $T$ is self-adjoint and strictly positive, we conclude that $\widetilde{\mathcal{H}}$ is an inner product space endowed with the inner product (\ref{dualinner_product}). The spaces $\mathcal{H}$ and $\widetilde{\mathcal{H}}$ have the same elements although their norms are different. The next lemma reveals the equivalence of their norms.

\begin{lemma}\label{norm-equivalence}
Let $\mathcal{H}$ be an FRKHS with respect to a
family $\mathcal{F}$ of linear functionals
$L_{\alpha},\ \alpha\in \Lambda$ and $K$ the
functional reproducing kernel for $\mathcal{H}$. If for a countable sequence $\alpha_j$, $j\in\mathbb{J}$, in $\Lambda$, the set $\mathbb{K}$ constitutes a frame for $\mathcal{H}$ and $T$ is the frame operator, then the norms of $\mathcal{H}$ and $\widetilde{\mathcal{H}}$, defined as above, are equivalent.
\end{lemma}
\begin{proof}
It is known \cite{D} that there exist $A,B>0$ such that for any $f\in \mathcal{H}$,
$$
A\|f\|_{\mathcal{H}}^2
\leq\langle Tf,f\rangle_{\mathcal{H}}
\leq B\|f\|_{\mathcal{H}}^2.
$$
This together with the fact
$\langle f,f\rangle_{\widetilde{\mathcal{H}}}
=\langle Tf,f\rangle_{\mathcal{H}}$
yields the norm equivalence of $\mathcal{H}$ and $\widetilde{\mathcal{H}}$.
\end{proof}

It follows from Lemma \ref{norm-equivalence} that $\widetilde{\mathcal{H}}$ is a Hilbert space since $\mathcal{H}$ is a Hilbert space. The following proposition shows that $\widetilde{\mathcal{H}}$ is also an FRKHS with respect to $\mathcal{F}$ and has $\widetilde{K}$ as its functional reproducing kernel.

\begin{prop}\label{dualLkernel}
Let $\mathcal{H}$ be an FRKHS with respect to
a family $\mathcal{F}$ of linear functionals
$L_{\alpha}, \alpha\in \Lambda,$ and $K$ the
functional reproducing kernel for $\mathcal{H}$.
If for a countable sequence $\alpha_j$, $j\in\mathbb{J}$, in $\Lambda$, the set $\mathbb{K}$ constitutes a frame
for $\mathcal{H}$ and $T$ is the frame operator, then space $\widetilde{\mathcal{H}}$ is also an FRKHS with respect to $\mathcal{F}$ and $\widetilde{K}$ is the corresponding functional reproducing kernel for $\widetilde{\mathcal{H}}$.
\end{prop}
\begin{proof}
By the norm equivalence of $\mathcal{H}$ and  $\widetilde{\mathcal{H}}$ established in Lemma \ref{norm-equivalence}, we can show that the norm of $\widetilde{\mathcal{H}}$ is compatible with $\mathcal{F}$ since so is the norm of $\mathcal{H}$. Combining the norm equivalence with the continuity of the functionals $L_{\alpha}, \alpha\in\Lambda,$ on $\mathcal{H}$, we also obtain that these functionals are continuous on $\widetilde{\mathcal{H}}$. The definition of the FRKHS ensures that $\widetilde{\mathcal{H}}$ is an FRKHS with respect to $\mathcal{F}$.

It remains to show that the operator $\widetilde{K}$ reproduces the functionals $L_{\alpha}(f)$, for $\alpha\in\Lambda$. By (\ref{dualinner_product}), the self-adjoint property of $T$ and the reproducing property of the kernel $K$, we get for each
$f\in \widetilde{\mathcal{H}}$ that
$$
\langle f,\widetilde{K}(\alpha) \rangle_{\widetilde{\mathcal{H}}}
=\langle Tf,\widetilde{K}(\alpha)\rangle_{\mathcal{H}}
=\langle f,T\widetilde{K}(\alpha)\rangle_{\mathcal{H}}
=\langle f,K(\alpha)\rangle_{\mathcal{H}}=L_{\alpha}(f),
\ \ \mbox{for all}\ \ \alpha\in\Lambda.
$$
Hence, $\widetilde{K}$ is the corresponding functional reproducing kernel for $\widetilde{\mathcal{H}}$.
\end{proof}

It is clear that
$G_j=\widetilde{K}(\alpha_j)$, $j\in\mathbb{J}.$ That is, $\widetilde{\mathbb{K}}:=\{\widetilde{K}(\alpha_j):
\ j\in\mathbb{J}\}$ is the dual frame of $\mathbb{K}$ in $\mathcal{H}$. We shall call operator $\widetilde{K}$ defined by (\ref{dualkernel}) the dual functional reproducing kernel of $K$ with respect to the set $\{\alpha_j: j\in\mathbb{J}\}$.

With the help of the dual functional reproducing kernel, we establish the complete reconstruction formula of an element in an FRKHS from its functional values.

\begin{thm}\label{complete-reconstruction-formula}
Let $\mathcal{H}$ be an FRKHS with respect to
the family $\mathcal{F}$ of linear functionals $L_{\alpha}$, $\alpha\in \Lambda$ and $K$ the
functional reproducing kernel for $\mathcal{H}$. If for a countable sequence $\alpha_j$, $j\in\mathbb{J}$, in $\Lambda$, the set $\mathbb{K}$ constitutes a frame for $\mathcal{H}$ and $\widetilde{K}$ is the dual functional reproducing kernel of $K$ with respect to the set $\{\alpha_j: j\in\mathbb{J}\}$, then $\widetilde{\mathbb{K}}$ is a dual frame of
$\mathbb{K}$ in $\mathcal{H}$ and there holds
for each $f\in\mathcal{H}$,
\begin{eqnarray}\label{reconstruction}
f=\sum_{j\in\mathbb{J}}L_{\alpha_j}(f)\widetilde{K}(\alpha_j),
\end{eqnarray}
where the reconstruction formula \eqref{reconstruction} holds also in $\widetilde{\mathcal{H}}$. Moreover, if $\mathbb{K}$ is a Riesz basis for $\mathcal{H}$, then $\widetilde{\mathbb{K}}$ is an orthonormal basis for the FRKHS $\widetilde{\mathcal{H}}$ of $\widetilde{K}$.
\end{thm}
\begin{proof}
It follows from the definition of the dual functional reproducing kernel that $\widetilde{\mathbb{K}}$ is a dual frame of $\mathbb{K}$ in $\mathcal{H}$. Then, by   (\ref{reconstruction-formula-H}) we have that
$$
f=\sum_{j\in\mathbb{J}}\langle f,K(\alpha_j)
\rangle_{\mathcal{H}}\widetilde{K}(\alpha_j).
$$
This together with the reproducing property $L_{\alpha_j}(f)=\langle f,K(\alpha_j)
\rangle_{\mathcal{H}}$, $j\in\mathbb{J},$ leads to
the reconstruction formula (\ref{reconstruction}).
Lemma \ref{norm-equivalence} guarantees that the norms of $\mathcal{H}$ and $\widetilde{\mathcal{H}}$ are equivalent. Moreover, the two spaces have the same elements. Therefore, the reconstruction formula \eqref{reconstruction} holds also in $\widetilde{\mathcal{H}}$.

In the case when $\mathbb{K}$ is a Riesz basis for $\mathcal{H}$, we verify the orthonormality of $\widetilde{\mathbb{K}}$ by using the biorthogonality
of Reisz basis $\mathbb{K}$ and its dual $\widetilde{\mathbb{K}}$. Indeed, for each $j,k\in\mathbb{J}$, we have that
$$
\langle\widetilde{K}(\alpha_j),
\widetilde{K}(\alpha_k)\rangle_{\widetilde{\mathcal{H}}}
=\langle T\widetilde{K}(\alpha_j),
\widetilde{K}(\alpha_k)\rangle_{\mathcal{H}}
=\langle K(\alpha_j), \widetilde{K}(\alpha_k)\rangle_{\mathcal{H}}
=\delta_{j,k}.
$$
\end{proof}

According to Theorem \ref{complete-reconstruction-formula}, we have to find the frame in the form $\{K(\alpha_j): j\in\mathbb{J}\}$ for an FRKHS in order to establish the complete reconstruction formula. It is helpful to have a characterization of such a frame in terms of the feature map representation (\ref{FRKHS_feature_representation_KL1}) of the corresponding functional reproducing kernel.

\begin{thm}\label{feature_Riesz}
Let $\Phi$ be a map from $\Lambda$ into a
Hilbert space $W$ satisfying
$\overline{\span}\{\Phi(\alpha):\alpha\in\Lambda\}=W.$
If $K$ is a functional reproducing kernel for an FRKHS $\mathcal{H}$, defined by \eqref{FRKHS_feature_representation_KL1},
then for a countable sequence $\alpha_j, j\in\mathbb{J},$
in $\Lambda$, the set $\mathbb{K}$ constitutes a
frame for $\mathcal{H}$ if and only if
$\{\Phi(\alpha_j): j\in\mathbb{J}\}$ forms a frame of $W$.
\end{thm}
\begin{proof}
We prove this theorem by establishing an isometric isomorphism between $\mathcal{H}$ and $W$.
To this end, we let
$$
\mathcal{K}(\alpha,\beta):=L_{\beta}(K(\alpha)),
\ \ \alpha,\beta\in \Lambda.
$$
By Corollary \ref{FRKHS_isometric-isomorphism}, we have that $f\rightarrow L_{(\cdot)}(f)$ is an isometric isomorphism between the FRKHS $\mathcal{H}$ of $K$ and the RKHS $\widetilde{\mathcal{H}}$ of $\mathcal{K}$.
Equation \eqref{FRKHS_feature_representation_KL1} shows that $\Phi$ is the feature map of the kernel $\mathcal{K}$ and $W$ is the corresponding feature space. We then see that $\langle u,\Phi(\cdot)\rangle_W\rightarrow u$ is an isometric isomorphism from $\widetilde{\mathcal{H}}$ to $W$. We define the operator $T:\mathcal{H}\rightarrow W$
as follows: For each $f\in\mathcal{H}$, $Tf\in W$ satisfies
$$
\langle Tf, \Phi(\beta)\rangle_W=L_{\beta}(f),
\ \ \mbox{for all}\ \ \beta\in\Lambda.
$$
It follows that $T$ is an isometric isomorphism
from $\mathcal{H}$ to $W$. By noting that $T(K(\alpha))=\Phi(\alpha)$, we obtain immediately
the desired result of this theorem.
\end{proof}

In general, the reconstruction formula (\ref{reconstruction}) holds only in the norm of an FRKHS. However, when an FRKHS is perfect, the series in (\ref{reconstruction}) converges both pointwise and uniformly on any set where the reproducing kernel is bounded. This is due to the point-evaluation functionals are also continuous, when an FRKHS is perfect. Motivated by this observation, it is of practical importance to consider sampling and reconstruction in a perfect FRKHS. In this case, the frame
$\{K(\alpha_j): j\in\mathbb{J}\}$ has the following characterization.

\begin{thm}\label{feature_Riesz1}
Let $W$ be a Hilbert space, $\Phi:\mathcal{X}\rightarrow W$ satisfy $\overline{\span}\{\Phi(x):x\in \mathcal{X}\}=W$ and $\Psi:\Lambda\rightarrow W$ satisfy
$\overline{\span}\{\Psi(\alpha):\alpha\in \Lambda\}=W$.
If $K$ is a functional reproducing kernel for a perfect
FRKHS $\mathcal{H}$, defined as in
\eqref{FRKHS_kernel-perfectRKHS},
then for a countable sequence
$\alpha_j$, $j\in\mathbb{J}$, in $\Lambda$, the set $\mathbb{K}$ constitutes a frame for $\mathcal{H}$ if
and only if $\{\Psi(\alpha_j): j\in\mathbb{J}\}$ forms
a frame for $W$.
\end{thm}
\begin{proof}
In order to apply Theorem \ref{feature_Riesz},
we need to represent the functional reproducing kernel
$K$ as in \eqref{FRKHS_feature_representation_KL1}.
It is known that the linear functionals on the perfect
FRKHS $\mathcal{H}$ are determined by
$$
L_{\alpha}(\langle u,\Phi(\cdot)\rangle_W)
:=\langle u,\Psi(\alpha)\rangle_W,
\ u\in W,\ \alpha\in\Lambda.
$$
By the representation \eqref{FRKHS_kernel-perfectRKHS}
of $K$, we have that
$$
L_{\beta}(K(\alpha))
=\langle\Psi(\alpha), \Psi(\beta)\rangle_W,
\ \  \mbox{for all}\ \ \alpha,\beta\in\Lambda.
$$
That is, the mapping $\Psi: \Lambda\rightarrow W$
is the feature map of $K$, defined as in \eqref{FRKHS_kernel-perfectRKHS}. Noting that
$$
\overline{\span}\{\Psi(\alpha):\alpha \in\Lambda\}=W,
$$
the desired result of this theorem follows directly from Theorem \ref{feature_Riesz}.
\end{proof}

Based upon the discussion above, we conclude that the frame $\{K(\alpha_j):j\in\mathbb{J}\}$ for an FRKHS
$\mathcal{H}$ is important for establishing the complete
reconstruction formula of functions in $\mathcal{H}$. These desirable frames can be obtained by characterizing the frame $\{\Phi(\alpha_j):j\in\mathbb{J}\}$ for the
feature space $W$. In the light of this point, we shall
consider the complete reconstruction of functions in
two concrete FRKHSs.

Below, we consider reconstructing a vector-valued
function $f:\mathbb{R}\rightarrow\mathbb{C}^n$ from
its sampled data
$$
\langle f(x_j),\xi_j\rangle_{\mathbb{C}^n},
\ (x_j,\xi_j)\in \mathbb{R}\times \mathbb{C}^n,
\ j\in\mathbb{Z}.
$$
The functions to be reconstructed are taken from the
Paley-Wiener space of $\mathbb{C}^n$-valued functions on
$\mathbb{R}$
$$
\mathcal{B}_{\pi}(\mathbb{R},\mathbb{C}^n)
:=\left\{f\in L^2(\mathbb{R},\mathbb{C}^{n}):
\ {\rm supp}\hat f\subseteq [-\pi,\pi]\right\}
$$
with the inner product
$$
\langle f,g\rangle_{\mathcal{B}_{\pi}(\mathbb{R},\mathbb{C}^n)}
:=\sum_{j\in\mathbb{N}_n}\langle f_j,g_j
\rangle_{L^{2}(\mathbb{R})}, \ \ \mbox{for}\ f:=[f_j:j\in\mathbb{N}_n], \ g:=[g_j:j\in\mathbb{N}_n].
$$
Set $\Lambda:=\mathbb{R}\times \mathbb{C}^n$ and for each
$(x,\xi)\in\Lambda,$ introduce a functional on
$\mathcal{B}_{\pi}(\mathbb{R},\mathbb{C}^n)$ as
$$
L_{(x,\xi)}(f):=\langle f(x),\xi\rangle_{\mathbb{C}^n},
\ f\in\mathcal{B}_{\pi}(\mathbb{R},\mathbb{C}^n).
$$
It was pointed out that the space
$\mathcal{B}_{\pi}(\mathbb{R},\mathbb{C}^n)$ is an FRKHS
with respect to the family $\mathcal{F}$ of linear functionals $L_{(x,\xi)}, (x,\xi)\in\Lambda$. The functional reproducing kernel $K$ for $\mathcal{B}_{\pi}(\mathbb{R},\mathbb{C}^n)$ is given by
$$
K(x,\xi):={\rm diag} (\mathcal{K}(x,\cdot):j\in\mathbb{N}_n)\xi,
\ (x,\xi)\in \Lambda,
$$
where $\mathcal{K}$ is the sinc kernel on $\mathbb{R}$ defined by
\begin{equation}\label{sincR}
\mathcal{K}(x,y):=\frac{\sin \pi(x-y)}{\pi(x-y)},
\ x,y\in\mathbb{R}.
\end{equation}
Set $W:=L^2([-\pi,\pi],\mathbb{C}^n)$ and define
$\Phi:\Lambda\rightarrow W$ by
$$
\Phi(x,\xi):=\frac{1}{\sqrt{2\pi}}e^{ix(\cdot)}\xi,
\ \ (x,\xi)\in \Lambda.
$$
It follows from
$$
\mathcal{K}(x,y)
=\left\langle\frac{1}{\sqrt{2\pi}}e^{ix(\cdot)},
\frac{1}{\sqrt{2\pi}}e^{iy(\cdot)}\right
\rangle_{L^2([-\pi,\pi])},
\ \ x,y\in\mathbb{R},
$$
that for all $(x,\xi),(y,\eta)\in \Lambda$,
\begin{eqnarray*}
L_{(y,\eta)}(K(x,\xi))
=\langle K(x,\xi)(y),\eta\rangle_{\mathbb{C}^n}
=\sum_{j\in\mathbb{N}_n}\left\langle\frac{1}{\sqrt{2\pi}}
e^{ix(\cdot)}\xi_j,\frac{1}{\sqrt{2\pi}}
e^{iy(\cdot)}\eta_j\right\rangle_{_{L^2([-\pi,\pi])}}.
\end{eqnarray*}
This leads to
$$
L_{(y,\eta)}(K(x,\xi))
=\langle\Phi(x,\xi), \Phi(y,\eta)\rangle_{W},
\ \  \mbox{for all} \ \ (x,\xi),(y,\eta)\in\Lambda.
$$
Thus, we conclude that $W$ and $\Phi$ are the feature
space and the feature map of the functional reproducing
kernel $K$, respectively. Moreover, the density condition
$\overline{\span}\{\Phi(x,\xi):\ (x,\xi)\in\Lambda\}=W$ is
satisfied.

We now turn to considering reconstructing a function
$f\in\mathcal{B}_{\pi}(\mathbb{R}, \mathbb{C}^n)$
from its functional values $L_{(x_j,\xi_j)}(f),
(x_j,\xi_j)\in\Lambda, j\in\mathbb{Z}$. This will be done by finding the sampling set
$\{(x_j,\xi_j):j \in\mathbb{Z}\}\subseteq\Lambda$
such that $\{K(x_j,\xi_j): j\in\mathbb{J}\}$ forms
a frame for $\mathcal{B}_{\pi}(\mathbb{R}, \mathbb{C}^n)$. By Theorem \ref{feature_Riesz} such a frame can be characterized by the frame for $W$ with the form  $\Phi(x_j,\xi_j), j\in\mathbb{Z}$. We shall give a positive example by choosing the sampling set $\{(x_j,\xi_j):j\in\mathbb{Z}\}$ as follows.
We present each $j\in\mathbb{Z}$ as $j:=nm+l$, where $m\in\mathbb{Z}$ and $l\in\mathbb{Z}_n:=\{0,1,\dots,n-1\}$.
We then choose the sequence $x_j, j\in\mathbb{Z},$ in $\mathbb{R}$ such that for each $l\in\mathbb{Z}_n$,
$\{\frac{1}{\sqrt{2\pi}}e^{ix_{nm+l}(\cdot)}:m\in\mathbb{Z}\}$ is a frame for $L^2([-\pi,\pi])$. We introduce $n$ vectors
$\eta_{l}:=[\eta_{l,k}:k\in\mathbb{Z}_n], l\in\mathbb{Z}_n,$ satisfying that the matrix $\mathbf{U}:=[\eta_l:l\in\mathbb{Z}_n]$ is unitary. We then choose the sequence $\xi_j,j\in\mathbb{Z},$
in $\mathbb{C}^n$ as $\xi_{nm+l}:=\eta_{l}$ for all
$l\in\mathbb{Z}_n$ and all $m\in\mathbb{Z}$. Based on the above notations, we show below that the sequence $\Phi(x_j,\xi_j), j\in\mathbb{Z},$ constitutes a frame for $W$.

\begin{thm}\label{frame-vector}
If the sequence $\{(x_j,\xi_j):j\in\mathbb{Z}\}
\subseteq\Lambda$ is defined as above, then
$\Phi(x_j,\xi_j), j\in\mathbb{Z},$ constitutes a frame
for $W$. Moreover, if for each $l\in\mathbb{Z}_n$, the
sequence $\{\frac{1}{\sqrt{2\pi}}e^{ix_{nm+l}(\cdot)}:
m\in\mathbb{Z}\}$ is a Riesz basis for $L^2([-\pi,\pi])$,
then $\Phi(x_j,\xi_j), j\in\mathbb{Z},$ forms a
Riesz basis for $W$.
\end{thm}
\begin{proof}
By the definition of frames, it suffices to verify that there exist positive constants $0<A\leq B<+\infty$ such that for all $u\in W$,
\begin{equation}\label{equivalence_A_B}
A\|u\|_{W}\leq \left(\sum_{j\in\mathbb{Z}}|
\langle u,\Phi(x_j,\xi_j)\rangle_{W}|^2\right)^{1/2}
\leq B\|u\|_{W},
\end{equation}
which implies $\Phi(x_j,\xi_j),j\in\mathbb{Z}$, is a frame for $W$. According to the description of the sequence
$\{(x_j,\xi_j): j\in\mathbb{Z}\}\subseteq\Lambda$
we have for any $u:=[u_k:k\in\mathbb{Z}_n]\in W$ that
\begin{eqnarray}\label{frame1}
\sum_{j\in\mathbb{Z}}\left|\langle u,
\Phi(x_j,\xi_j)\rangle_{W}\right|^2
=\sum_{l\in\mathbb{Z}_n}\sum_{m\in\mathbb{Z}}
\left|\left<\sum_{k\in\mathbb{Z}_n}\overline{\eta}_{l,k}u_k,
\frac{1}{\sqrt{2\pi}}e^{ix_{nm+l}(\cdot)}
\right>_{L^2([-\pi,\pi])}\right|^2.
\end{eqnarray}
For each $l\in\mathbb{Z}_n$, set
$\displaystyle{v_l:=\sum_{k\in\mathbb{Z}_n}
\overline{\eta}_{l,k}u_k}$. Since $\{\frac{1}{\sqrt{2\pi}}e^{ix_{nm+l}(\cdot)}:
m\in\mathbb{Z}\}$ is a frame for $L^2([-\pi,\pi])$,
there exist positive constants $0<A_l\leq B_l<+\infty$
such that
$$
A_l^2\|v_l\|_{L^2([-\pi,\pi])}^2
\leq\sum_{m\in\mathbb{Z}}\left|\left< v_l,
\frac{1}{\sqrt{2\pi}}e^{ix_{nm+l}(\cdot)}
\right>_{L^2([-\pi,\pi])}\right|^2
\leq B_l^2\|v_l\|_{L^2([-\pi,\pi])}^2.
$$
Substituting the above estimates for all $l\in\mathbb{Z}_n$ into (\ref{frame1}), we get that
\begin{eqnarray}\label{frame2}
A^2\sum_{l\in\mathbb{Z}_d}\|v_l\|_{L^2([-\pi,\pi])}^2
\leq\sum_{j\in\mathbb{Z}}\left|\left< u,
\Phi(x_j,\xi_j)\right>_{W}\right|^2\leq B^2
\sum_{l\in\mathbb{Z}_d}\|v_l\|_{L^2([-\pi,\pi])}^2,
\end{eqnarray}
where $A:=\min\{A_l:l\in\mathbb{Z}_n\}$ and
$B:=\max\{B_l:l\in\mathbb{Z}_n\}$. Note that there holds
\begin{eqnarray}\label{gn}
\sum_{l\in\mathbb{Z}_n}\|v_l\|_{L^2([-\pi,\pi])}^2
=\sum_{k\in\mathbb{Z}_n}\sum_{j\in\mathbb{Z}_n}
\left(\sum_{l\in\mathbb{Z}_n}
\overline{\eta}_{l,k}\eta_{l,j}\right)
\langle u_k, u_j\rangle_{L^2([-\pi,\pi])}.
\end{eqnarray}
Equation (\ref{gn}) with the unitary assumption on $\mathbf{U}$ leads to
$$
\sum_{l\in\mathbb{Z}_n}\|v_l\|_{L^2([-\pi,\pi])}^2
=\sum_{k\in\mathbb{Z}_n}\|u_k\|_{L^2([-\pi,\pi])}^2
=\|u\|^2_{W}.
$$
Hence, by (\ref{frame2}) we get the equivalent property \eqref{equivalence_A_B}.

We next prove that $\Phi(x_j,\xi_j),j\in\mathbb{Z},$
is also a Riesz basis for $W$ based on the hypothesis that
for each $l\in\mathbb{Z}_n$, $\{\frac{1}{\sqrt{2\pi}}
e^{ix_{nm+l}(\cdot)}:m\in\mathbb{Z}\}$ is a Riesz basis
for $L^2([-\pi,\pi])$. It suffices to show that if
there exist $\{c_j:j\in\mathbb{Z}\}\in l^2(\mathbb{Z})$ such that $\sum_{j\in\mathbb{Z}}c_j\Phi(x_j,\xi_j)=0$,
then $c_j=0,j\in\mathbb{Z}.$ It follows from
$\sum_{j\in\mathbb{Z}}c_j\Phi(x_j,\xi_j)=0$
that for all $j'\in\mathbb{Z}$,
\begin{eqnarray}\label{Riesz1}
\sum_{j\in\mathbb{Z}}c_j\langle\Phi(x_j,\xi_j),
\Phi(x_{j'},\xi_{j'})\rangle_{W}=0.
\end{eqnarray}
There holds for any $j:=nm+l$ and any $j':=nm'+l'$ that
$$
\langle\Phi(x_j,\xi_j),\Phi(x_{j'},\xi_{j'})\rangle_{W}
=\sum_{k\in\mathbb{Z}_n}\overline{\eta}_{l,k}\eta_{l',k}
\left\langle\frac{1}{\sqrt{2\pi}}e^{ix_{nm+l}(\cdot)},
\frac{1}{\sqrt{2\pi}}e^{ix_{nm'+l'}(\cdot)}
\right\rangle_{L^2([-\pi,\pi])}.
$$
Again, using the unitary property of ${\bf U}$, we obtain that
\begin{eqnarray}\label{Riesz2}
\langle\Phi(x_j,\xi_j),\Phi(x_{j'},\xi_{j'})\rangle_{W}
=\delta_{l,l'}\left\langle\frac{1}{\sqrt{2\pi}}
e^{ix_{nm+l}(\cdot)}, \frac{1}{\sqrt{2\pi}}
e^{ix_{nm'+l'}(\cdot)}\right\rangle_{L^2([-\pi,\pi])}.
\end{eqnarray}
Substituting (\ref{Riesz2}) into (\ref{Riesz1}), we obtain
for any $m'\in\mathbb{Z}$ and any $l'\in\mathbb{Z}_n$ that
\begin{eqnarray*}
\left\langle\sum_{m\in\mathbb{Z}}c_{nm+l'}
\frac{1}{\sqrt{2\pi}}e^{ix_{nm+l'}(\cdot)},
\frac{1}{\sqrt{2\pi}}e^{ix_{nm'+l'}(\cdot)}
\right\rangle_{L^2([-\pi,\pi])}=0.
\end{eqnarray*}
This together with the hypothesis that
$\{\frac{1}{\sqrt{2\pi}}e^{ix_{nm+l'}(\cdot)}:
m\in \mathbb{Z}\}$ is a Riesz basis for $L^2([-\pi,\pi])$
leads to
$$
\frac{1}{\sqrt{2\pi}}\sum_{m\in\mathbb{Z}}
c_{nm+l'}e^{ix_{nm+l'}(\cdot)}=0.
$$
Since for any $l'\in\mathbb{Z}_N$,
$\{\frac{1}{\sqrt{2\pi}}e^{ix_{nm+l'}(\cdot)}:
m\in\mathbb{Z}\}$ is a Riesz basis for $L^2([-\pi,\pi])$, we conclude that $c_{nm+l'}=0$, for all $m\in\mathbb{Z}$ and $l'\in\mathbb{Z}_n,$ which completes the proof.
\end{proof}

Sampling sets $\{(x_j,\xi_j): j\in\mathbb{Z}\}$ such that $\Phi(x_j,\xi_j),j\in\mathbb{Z},$ constitute a Riesz basis for $W$ were characterized in \cite{AI95,AI08} in terms of the generating matrix function. The authors considered the matrix Sturm-Liouville problems and applied the boundary control theory for hyperbolic dynamical systems in producing a wide class of matrix functions generating sampling sets. Such sampling sets can provide us the desired frames for
$\mathcal{B}_{\pi}(\mathbb{R}, \mathbb{C}^n)$
having the form $\{K(x_j,\xi_j): j\in\mathbb{Z}\}$.
By constructing the frame operator (\ref{frame_operator}) and obtaining the dual functional reproducing kernel $\widetilde{K}$ by (\ref{dualkernel}), we can
build the reconstruction formula of $f\in\mathcal{B}_{\pi}(\mathbb{R},\mathbb{C}^n)$ as
$$
f=\sum_{j\in\mathbb{Z}}\langle f(x_j),
\xi_j\rangle_{\mathbb{C}^n}\widetilde{K}(x_j,\xi_j).
$$

We next consider recovering a scalar-valued function
$f$ on $\mathbb{R}$ from the sampled data determined by  integral functionals. As a special case, the local average
sampling is devoted to reconstructing a function
$f$ from its local average values in the form
\begin{equation*}\label{averagefunctional}
L_{x_j}(f):=\int_{\mathbb{R}}f(t)u_{x_j}(t)dt,
\ x_j\in\mathbb{R},\  j\in\mathbb{Z},
\end{equation*}
where $u_x$, $x\in\mathbb{R}$, are nonnegative
functions that satisfy
\begin{equation}\label{average-functions}
\int_{\mathbb{R}}u_x(t)dt=1\ \mbox{and}\ \
\supp u_x\subseteq [x-\delta, x+\delta],
\end{equation}
with $\delta$ being a positive constant. We call $u_x$, $x\in\mathbb{R}$, the average functions. We shall restrict
ourselves to investigating the local average sampling
in the framework of FRKHSs. The FRKHS to be considered is
the Paley-Wiener space of scalar-valued functions on
$\mathbb{R}$
$$
\mathcal{B}_{\pi}:=\{f\in L^2(\mathbb{R}):
{\rm supp}\hat f\subseteq [-\pi,\pi]\}.
$$
The space $\mathcal{B}_{\pi}$ is a standard RKHS,
in which the ideal sampling and the local average
sampling are usually considered. The sinc kernel for
$\mathcal{B}_{\pi}$, defined by (\ref{sincR}), is a
translation invariant reproducing kernel defined by
(\ref{translation_invariant_rk_phi}) with
$\varphi:=\frac{1}{2\pi}\chi_{[-\pi,\pi]}$.
As pointed out in the last section, the space
$W:=L^2([-\pi,\pi])$ and the mapping
$\Phi:\mathbb{R}\rightarrow W$, defined by
$\Phi(x):=\frac{1}{\sqrt{2\pi}}e^{ix(\cdot)}$,
are the feature space and the feature map of the sinc kernel, respectively. It is clear that the function $\varphi$ and the average functions
$u_{x},\ x\in\mathbb{R},$ are all integrable. Hence, by Theorem \ref{translation_invariant_ORKHS} we obtain that the closed subspace
$$
\mathcal{H}:=\left\{f\in \mathcal{B}_{\pi}:
\check{f}\in\overline{\span}\left\{u_{x}^{\vee}
\chi_{[-\pi,\pi]}:\ x\in\mathbb{R}\right\}\right\}
$$
of $\mathcal{B}_{\pi}$ is a perfect FRKHS with respect to the family $\mathcal{F}$ of the local average functionals
$L_{x},x\in\mathbb{R}$. The closure in the representation of $\mathcal{H}$ is taken in $W$. Moreover, the functional
reproducing kernel is given by
$$
K(x)(y)=[u_{x}^{\vee}\chi_{[-\pi,\pi]}]^{\wedge}(y),
\ x,y\in\mathbb{R}.
$$
For each $x\in\mathbb{R}$, set
$$
\Psi(x):=\sqrt{2\pi}u_{x}^{\vee}\chi_{[-\pi,\pi]}.
$$
Then $K$ can be represented as in
\eqref{FRKHS_kernel-perfectRKHS}.

To relate our new theory to the existing average sampling theorems in the literature, we consider recovering functions in $\mathcal{B}_{\pi}$, rather than the closed subspace $\mathcal{H}$, from the local average values. Hence, we recall Corollary \ref{density-W}, which states that $\mathcal{B}_{\pi}$ itself is a perfect
FRKHS if and only if there holds the density
$\overline{\span}\{\Psi(x):x\in\mathbb{R}\}=W$.
We restrict ourselves to this special case. Since $\mathcal{B}_{\pi}$ is a perfect FRKHS, with the help of Theorem \ref{feature_Riesz1} we shall try to construct  the sampling set $\{x_j:j\in\mathbb{Z}\}$ such that the sequence $\Psi(x_j), j\in\mathbb{Z},$ constitutes a frame for $W$. The next theorem gives a sufficient condition for
$\Psi(x_j), j\in\mathbb{Z},$ to be a frame for $W$.
For any $u\in W$, we let $R_{u}$ and $I_{u}$ denote
its real part and imaginary part, respectively.

\begin{thm}\label{frame-Paley}
Let $\delta$ be a positive constant. Suppose that the
sequence of nonnegative functions $u_x,x\in\mathbb{R},$ satisfy condition \eqref{average-functions}.
If there exist positive constants $A,B$ depending only on
the sampling set $\{x_j:j\in\mathbb{Z}\}$ and
$\delta$ such that for any sequence $t_j, j\in\mathbb{Z},$ with $t_j\in [x_j-\delta,x_j+\delta]$,
$\{e^{it_j(\cdot)}:j\in\mathbb{Z}\}$ forms a frame for
$W$ with the frame bounds $A,B$, then
$\{\Psi(x_j): j\in\mathbb{Z}\}$ constitutes a
frame for $W$.
\end{thm}
\begin{proof}
To show that $\Psi(x_j), j\in\mathbb{Z},$ constitutes a frame for $W,$ we need to estimate the quantities $Q_j(u):=\langle u,\Psi(x_j)\rangle_{W}$, $j\in\mathbb{Z},$ for any $u\in W$. It follows for any $u\in W$ that
\begin{eqnarray}\label{I}
|Q_j(u)|^2
=\left|\int_{-\pi}^{\pi}u(t)\left(\frac{1}{2\pi}
\int_{\mathbb{R}}u_{x_j}(s)e^{-ist}ds\right)dt\right|^2
=\frac{1}{(2\pi)^2}\left|\int_{\mathbb{R}}
\langle u, e^{is(\cdot)}\rangle_{W}u_{x_j}(s)ds\right|^2.
\end{eqnarray}
We decompose $u$ into the odd and the even parts.
That is, $u=u_{o}+u_{e}$, where
$$
u_{o}(x)=\frac{u(x)-u(-x)}{2},\ u_{e}(x)=\frac{u(x)+u(-x)}{2},\ x\in[-\pi,\pi].
$$
Substituting the decomposition $u=u_{o}+u_{e}$ into
\eqref{I}, we get that
\begin{equation}\label{I1}
|Q_j(u)|^2
=\frac{1}{(2\pi)^2}\Big{|}\int_{\mathbb{R}}F_u(s)u_{x_j}(s)ds
+i\int_{\mathbb{R}}G_u(s)u_{x_j}(s)ds\Big{|}^2,
\end{equation}
where
$$
F_u(s):=\langle R_{u_e}+iI_{u_o},
e^{is(\cdot)}\rangle_{W} \ \mbox{and}\
G_u(s):=\langle I_{u_e}-iR_{u_o}, e^{is(\cdot)}\rangle_{W}, \ \  s\in\mathbb{R}.
$$
Note that $F_u$ and $G_u$ are both real and continuous
functions. By the mean value theorem and the hypothesis
on $u_x,x\in\mathbb{R}$, there exist
$s_j,\tilde{s}_j \in [x_j-\delta,x_j+\delta]$ such that
$$
\int_{\mathbb{R}}F_u(s)u_{x_j}(s)ds=F_u(s_j)\ \mbox{and}\
\int_{\mathbb{R}}G_u(s)u_{x_j}(s)ds=G_u(\tilde{s}_j).
$$
Substituting the above equations into (\ref{I1}), we obtain that
\begin{eqnarray}\label{I2}
|Q_j(u)|^2=\frac{1}{(2\pi)^2}[F_u^2(s_j)+G_u^2(\tilde{s}_j)].
\end{eqnarray}
By the hypothesis of this theorem, both
$\{e^{is_j(\cdot)}:j\in\mathbb{Z}\}$ and
$\{e^{i\tilde{s}_j(\cdot)}:j\in\mathbb{Z}\}$ are
frames for $W$ with the frame bounds $A,B$. Hence, by the definition of $F_u$ and $G_u$, we get that
$$
A^2\|R_{u_e}+iI_{u_o}\|_{W}^2
\leq\sum_{j\in\mathbb{Z}}F_u^2(s_j)
\leq B^2\|R_{u_e}+iI_{u_o}\|_{W}^2
$$
and
$$
A^2\|I_{u_e}-iR_{u_o}\|_{W}^2
\leq\sum_{j\in\mathbb{Z}}G_u^2(\tilde{s}_j)
\leq B^2\|I_{u_e}-iR_{u_o}\|_{W}^2.
$$
By (\ref{I2}) and noting that
$$
\|u\|_{W}^2=\|R_{u_e}+iI_{u_o}\|_{W}^2+\|I_{u_e}-iR_{u_o}\|_{W}^2,
$$
we conclude that
$$
\frac{1}{2\pi}A\|u\|_{W}
\leq\left(\sum_{j\in\mathbb{Z}}|Q_j(u)|^2\right)^{1/2}
\leq\frac{1}{2\pi}B\|u\|_{W},
$$
which completes the proof.
\end{proof}

Observing from Theorem \ref{frame-Paley}, we have that
the frames for $W$ with the form
$\{e^{it_j(\cdot)}: j\in\mathbb{Z}\}$ is important for the study of the frames with the form
$\{\Psi(x_j): j\in\mathbb{Z}\}$. For a complete characterization of the former, one can see \cite{C,Y}.
By making use of the sampling set $\{x_j:j\in\mathbb{Z}\}$, we obtain the frame
$K(x_j),\ j\in\mathbb{Z},$ for $\mathcal{B}_{\pi}$.  Accordingly, we get the complete reconstruction formula for $f\in\mathcal{B}_{\pi}$ as
$$
f=\sum_{j\in\mathbb{Z}}\left(\int_{\mathbb{R}}f(t)u_{x_j}(t)
dt\right)\widetilde{K}(x_j),
$$
where $\widetilde{K}$ is the dual functional reproducing kernel of $K$ with respect to the set $\{x_j:j\in\mathbb{Z}\}$.

As applications of Theorem \ref{frame-Paley}, we present two specific examples for choosing the sampling set $\{x_j:j\in\mathbb{Z}\}$ and the average functions $u_{x},x\in\mathbb{R}$.
\begin{example}
{\rm Let $x_j=j, j\in\mathbb{Z},$ and $0<\delta<1/4$.
The well-known Kadec's $1/4$-theorem shows that for
any sequence $\{t_j:j\in\mathbb{Z}\}$ satisfying
$t_j\in[x_j-\delta,x_j+\delta], j\in\mathbb{Z},$
the sequence $\{e^{it_j(\cdot)}:j\in\mathbb{Z}\}$
forms a frame for $L^2([-\pi,\pi])$ with the bounds
$$
2\pi(\cos(\delta \pi)-\sin(\delta\pi))^2\ \mbox{and}\ 2\pi(2-\cos(\delta \pi)+\sin(\delta\pi))^2.
$$
That is, the hypothesis of Theorem \ref{frame-Paley} is satisfied. Hence, we have by Theorem \ref{frame-Paley}
that $\{\Psi(x_j): j\in\mathbb{Z}\}$ constitutes a
frame for $W$ and then we can establish the
complete reconstruction formula of functions
in $\mathcal{B}_{\pi}$ as in Theorem
\ref{complete-reconstruction-formula}. In fact, the
local average sampling theorem in this case was
established in \cite{SZ} for $\mathcal{B}_{\pi}$ without the concept of FRKHSs. A generalized Kadec's $1/4$-theorem can help us extend this result. Specifically, we suppose that the sampling set $\{x_j:j\in\mathbb{Z}\}$ satisfies that $\{e^{ix_j(\cdot)}: j\in\mathbb{Z}\}$ is a frame for $W$ with bounds $A,B$ and $0<\delta<1/4$ such that
\begin{equation*}\label{Kadec}
1-\cos(\delta\pi)+\sin(\delta\pi)<\sqrt{\frac{A}{B}}.
\end{equation*}
The generalized Kadec's $1/4$-theorem states that
for any $\{t_j:j\in\mathbb{Z}\}$ with
$t_j\in[x_j-\delta,x_j+\delta], j\in\mathbb{Z},$
$\{e^{it_j(\cdot)}:j\in\mathbb{Z}\}$ forms a frame for
$W$ with bounds
$$
A\left(1-\sqrt{\frac{B}{A}}(1-\cos(\delta\pi)
-\sin(\delta\pi))\right)^2\ \mbox{and}\
B(2-\cos(\delta \pi)+\sin(\delta\pi))^2.
$$
Hence, the sampling set $\{x_j:\ j\in\mathbb{Z}\}$ and
$\delta$ in this case also satisfy the hypothesis of
Theorem \ref{frame-Paley} and thus,
$\{\Psi(x_j): j\in\mathbb{Z}\}$ constitutes a frame for $W$.}
\end{example}

\begin{example}
{\rm We review an example considered in \cite{SZ} and reinterpret it by Theorem \ref{frame-Paley}. Specifically, we suppose that $\alpha, L$ are positive constants and $0<\epsilon<1$. It is known \cite{DS} that for any sampling set $\{x_j:\ j\in\mathbb{Z}\}$ such that
$|x_j-x_k|\geq\alpha, j\neq k$ and $|x_j-j\epsilon|\leq L$, the sequence $\{e^{ix_j(\cdot)}:j\in\mathbb{Z}\}$ forms a frame for $W$ with bounds $A, B$ depending only on
$\alpha, L, \epsilon.$ If we choose the sampling set
$\{x_j:j\in\mathbb{Z}\}$ satisfying the above condition
with $\alpha, L, \epsilon$ and $0<\delta<\alpha/2$, then
we can verify that for any sequence $t_j, j\in\mathbb{Z},$
with $t_j\in [x_j-\delta,x_j+\delta]$, there holds
$$
|t_j-t_k|\geq\alpha-2\delta, j\neq k,\ |t_j-j\epsilon|
\leq L+\delta.
$$
Hence, we conclude that
$\{e^{it_j(\cdot)}: j\in\mathbb{Z}\}$ forms a frame for $W$ with the bounds $A,B$ depending only on
$\alpha, L, \epsilon$ and $\delta$. That is, the hypothesis of Theorem \ref{frame-Paley} is satisfied. Hence, we get that $\{\Psi(x_j):j\in\mathbb{Z}\}$ constitutes a frame for $W$. The corresponding complete reconstruction formula of functions
in $\mathcal{B}_{\pi}$ can be built as in Theorem
\ref{complete-reconstruction-formula}. Such a local average sampling theorem was also established in \cite{SZ}.}
\end{example}

To end the discussion about the local average sampling in the Paley-Wiener space $\mathcal{B}_{\pi}$, we consider a special case that $u_x(t)=u(t-x),\ t\in\mathbb{R}$, with
$u\in L^1(\mathbb{R})$ and present a sufficient condition
for $\{\Psi(x_j): j\in\mathbb{J}\}$ to be a frame for $W$.

\begin{thm}
Let $u\in L^1(\mathbb{R})$ satisfy $|\check{u}(t)|\geq c$
for a positive constant $c$ and any $t\in[-\pi,\pi]$ and
$u_x:=u(\cdot-x),\ x\in\mathbb{R}$. If $\{e^{ix_j(\cdot)}:
j\in\mathbb{Z}\}$ is a frame for $W$, then $\{\Psi(x_j):
j\in\mathbb{Z}\}$ also constitutes a frame for $W$.
\end{thm}
\begin{proof}
It follows for any $x\in\mathbb{R}$ that
$u_{x}^{\vee}=e^{ix(\cdot)}\check{u}$. Then we have for any $w\in W$ that
$$
\sum_{j\in\mathbb{Z}}|\langle w,\Psi(x_j)\rangle_{W}|^2
=\sum_{j\in\mathbb{Z}}|\langle w \overline{\check{u}}
\chi_{[-\pi,\pi]},e^{ix_j(\cdot)}\rangle_{W}|^2.
$$
Since $\{e^{ix_j(\cdot)}:j\in\mathbb{Z}\}$ is a frame
for $W$, there exist $0<A\leq B<\infty$ such that
\begin{equation}\label{equality11}
A^2\|w \overline{\check{u}}\chi_{[-\pi,\pi]}\|_{W}^2
\leq \sum_{j\in\mathbb{Z}}|\langle w,
\Psi(x_j)\rangle_{W}|^2
\leq B^2\|w \overline{\check{u}}\chi_{[-\pi,\pi]}\|_{W}^2.
\end{equation}
By the assumptions on function $u$, we have that
$$
c\|w\|_{W}
\leq \|w\overline{\check{u}}\chi_{[-\pi,\pi]}\|_{W}
\leq \frac{1}{2\pi}\|u\|_{L^1(\mathbb{R})}\|w\|_{W}.
$$
Substituting the above estimate into (\ref{equality11}),
we obtain for any $w\in W$ that
\begin{equation*}
A'\|w\|_{W}
\leq \left(\sum_{j\in\mathbb{Z}}|\langle w,\Psi(x_j)
\rangle_{W}|^2\right)^{1/2}
\leq B'\|w\|_{W},
\end{equation*}
with $A':=Ac$ and  $B':=\frac{B}{2\pi}\|u\|_{L^1(\mathbb{R})}$.
That is, the sequence $\{\Psi(x_j): j\in\mathbb{Z}\}$
constitutes a frame for $W$.
\end{proof}

To close this section, we remark that
taking into account the continuity of the sampling process, FRKHSs are the right spaces for reconstruction using functional-valued data. The complete reconstruction formula in such spaces can be obtained by constructing the frames in terms of the corresponding functional reproducing kernels. Characterizing the desired frames by features provides us a convenient approach for finding them. The specific examples discussed in this section include not only the existing sampling theorems concerning the functional-valued data but also new ones. Hence, we point out that FRKHSs indeed provide an ideal framework for systematically study sampling and reconstruction with respect to functional-valued data.

\section{Regularized Learning in ORKHSs}

Regularization, a widely used approach in solving an ill-posed problem of learning the unknown target function from finite samples, has been a focus of attention in the field of machine learning \cite{CS,MP05,SS,SC,V}. We study in this section the regularized learning from finite operator-valued data in the framework of ORKHSs. The resulting representer theorem shows that as far as the operator-valued data are concerned, learning in ORKHSs has the advantages over RKHSs.

We begin with recalling the regularized learning schemes for learning a function in an RKHS from finite point-evaluation data. Suppose that $\mathcal{X}$ is the input space and the output space $\mathcal{Y}$ is a Hilbert space. Set $\mathcal{Y}^m:=\mathcal{Y}\times\cdots\times\mathcal{Y}$ ($m$ times). Let $Q:\mathcal{Y}^m\times\mathcal{Y}^m
\to\mathbb{R}_{+}:=[0,+\infty]$ be a prescribed loss function, $\phi:\mathbb{R}_{+}\to\mathbb{R}_{+}$ a regularizer and $\lambda$ a positive regularization parameter. The regularized learning of a function in a $\mathcal{Y}$-valued RKHS $\mathcal{H}$ on $\mathcal{X}$
with the reproducing kernel $\mathcal{K}$ from its finite samples $\{(x_j,\xi_j):j\in\mathbb{N}_m\}
\subset\mathcal{X}\times\mathcal{Y}$ may be formulated as the following minimization problem:
\begin{equation}\label{regularization}
\inf_{f\in\mathcal{H}}\left\{Q(f(\mathbf{x}), \mathbf{\xi})+\lambda\phi(\|f\|_{\mathcal{H}})\right\},
\end{equation}
where $\mathbf{x}:=[x_j:j\in\mathbb{N}_m], f(\mathbf{x}):=[f(x_j):j\in\mathbb{N}_m]$ and $\mathbf{\xi}:=[\xi_j:j\in\mathbb{N}_m]$.
Many popular learning schemes such as regularization networks and support vector machines correspond to the minimization problem (\ref{regularization}) with different choices of $Q$
and $\phi$. Specifically, set $\phi(t):=t^2,
\ t\in\mathbb{R}_{+}$. If the loss function $Q$ is
chosen as
$$
Q(f(\mathbf{x}), \mathbf{\xi})
:=\sum_{j\in\mathbb{N}_m}\|f(x_j)-\xi_j\|_{\mathcal{Y}}^2,
$$
the corresponding algorithm is called the regularization networks. We next choose the loss function as
$$
Q(f(\mathbf{x}), \mathbf{\xi}):=\sum_{j\in\mathbb{N}_m}
\max(\|f(x_j)-\xi_j\|_{\mathcal{Y}}-\varepsilon,0),
$$
where $\varepsilon$ is a positive constant. In this case, the algorithm coincides with the support vector machine regression.

The success of the regularized learning lies on the well-known representer theorem. This remarkable
result shows that for certain choice of loss functions and regularizers the solution to the minimization problem (\ref{regularization}) can be represented by a linear combination of the kernel sections $\mathcal{K}(x_j,\cdot),\ j\in\mathbb{N}_m$. Then the original minimization problem in a potentially infinite-dimensional Hilbert space can be converted into one about finitely many coefficients emerging in the linear combination of the kernel sections. Moreover, since the reproducing kernel is used to measure the similarity between inputs, the representer theorem provides us with a reasonable output by making use of the input similarities.

The representer theorem for the regularization networks in the scalar case dates from \cite{KW} and was generalized for non-quadratic loss functions and nondecreasing regularizers \cite{AMP,CO,SHS}. As the multi-task learning and learning in Banach spaces received considerable attention recently, the representer theorems in vector-valued RKHSs and RKBSs were also established \cite{MP05,ZXZ,ZZ12,ZZ13}. We state  in the following theorem a general result for the solution to the minimization problem (\ref{regularization}). To this end, we suppose that the loss function $Q$ is continuous with respect to each of its first $m$ variables under the weak topology on $\mathcal{Y}$ and convex on $\mathcal{Y}^m$. Moreover, the regularizer $\phi$ is assumed to be continuous, nondecreasing, strictly convex and satisfy $\displaystyle{\lim_{t\to+\infty}\phi(t)=+\infty}$.

\begin{thm}\label{representer-theorem}
Let $\mathcal{H}$ be a $\mathcal{Y}$-valued RKHS on $\mathcal{X}$ with the reproducing kernel $\mathcal{K}$. If $Q$ and $\phi$ satisfy the above hypothesis then there exists a unique minimizer $f_0$ of \eqref{regularization} and it has the form
\begin{equation}\label{representer-theorem-formula}
f_0=\sum_{j\in\mathbb{N}_m}\mathcal{K}(x_j,\cdot)\eta_j
\end{equation}
for some sequence $\eta_j,j\in\mathbb{N}_m,$ in
$\mathcal{Y}$.
\end{thm}
Substituting the representation (\ref{representer-theorem-formula}) of the minimizer into (\ref{regularization}), the original minimization problem in an RKHS is converted into the minimization problem about the parameters $\eta_j,j\in\mathbb{N}_m.$ Especially in some occasions, with the help of the representer theorem, we can convert the minimization problem into a system of equations about the parameters. For the regularization networks, if $\mathcal{K}(x_j,\cdot),j\in\mathbb{N}_m,$
are linearly independent in $\mathcal{H}$, then
the parameters $\eta_j,j\in\mathbb{N}_m,$ in (\ref{representer-theorem-formula}) satisfy
the linear equations
\begin{equation*}\label{linearsystem}
\sum_{k\in\mathbb{N}_m}\mathcal{K}(x_k,x_j)\eta_k
+\lambda\eta_j=\xi_j,\ j\in\mathbb{N}_m.
\end{equation*}

As the point-evaluation data are not readily available in practical applications, it is desired to study the regularized learning from finite samples which are in general the linear functional values or linear operator values. Most of the past work mainly focused on such learning problems in RKHSs and RKBSs, where the point-evaluation functionals are continuous \cite{V,YC,ZZ12}. A basic assumption of these work is that the finite linear functionals are continuous on a prescribed space, which guarantees the stability of the finite sampling process. Under this hypothesis the representer theorem was obtained. Specifically, in RKHSs, one can represent the solution to the regularization problem by a linear combination of the dual elements of the given linear functionals.

In general, learning a function in a usual RKHS from its non-point-evaluation data is not appropriate. Two reasons account for this consideration. First of all, although the representer theorem still exists for the regularized learning from functional-valued data in an RKHS, it represents the solution by making use of the dual elements of the continuous linear functionals other than the reproducing kernel. This will bring difficulties to the computation of the solution. Secondly, when only finite linear functionals values are used to learn a target element, the underlying RKHS may ensure the continuity of the finite linear functionals. However, to improve the accuracy of approximation, more functional values should be used in the regularized learning scheme. An RKHS may not have the ability to guarantee the continuity of a family of linear functionals. The disadvantages of RKHSs lead us to consider regularized learning from operator-valued data in the ORKHSs setting.

Suppose that $\mathcal{H}$ is an ORKHS with respect to the set $\mathcal{L}:=\{L_{\alpha}:\alpha\in \Lambda\}$ of
linear operators from $\mathcal{H}$ to $\mathcal{Y}$.
Let $\Lambda_m:=[\alpha_j:j\in\mathbb{N}_m]\in \Lambda^m$, $\mathbf{\xi}:=[\xi_j:j\in\mathbb{N}_m]\in\mathcal{Y}^m$ and for each $f\in\mathcal{H}$ set $L_{\Lambda_m}(f):=[L_{\alpha_j}(f): j\in\mathbb{N}_m]$. We consider the regularized learning algorithm with respect to the values of linear operators as
\begin{equation}\label{general-regularization}
\inf_{f\in\mathcal{H}}\left\{Q(L_{\Lambda_m}(f), \mathbf{\xi})+\lambda\phi(\|f\|_{\mathcal{H}})\right\},
\end{equation}
where $Q,\phi$ and $\lambda$ are defined as in (\ref{regularization}). The following theorem gives the existence and the uniqueness of the minimizer of (\ref{general-regularization}) and the corresponding representer theorem. We note that the theorem can be proved by the similar arguments to Theorem \ref{representer-theorem}. However, we will give an alternative proof by making use of the isometric isomorphism between $\mathcal{H}$ and a vector-valued RKHS.

\begin{thm}\label{representer_theorem_ORKHS}
Suppose that $\mathcal{H}$ is an ORKHS with respect to
the set $\mathcal{L}:=\{L_{\alpha}:\alpha\in \Lambda\}$
of linear operators from $\mathcal{H}$ to $\mathcal{Y}$ and $K$ is the operator reproducing kernel for $\mathcal{H}$. If the loss function $Q$ and the regularizer $\phi$ satisfy the hypothesis in Theorem \ref{representer-theorem} then there exists a unique minimizer $f_0$ of \eqref{general-regularization} and there exist $\eta_j\in\mathcal{Y},\ j\in\mathbb{N}_m,$ such that
\begin{equation}\label{representer_theorem_ORKHS-formula}
f_0=\sum_{j\in\mathbb{N}_m}K(\alpha_j)\eta_j.
\end{equation}
\end{thm}
\begin{proof}
We introduce the operator $\mathcal{K}:\Lambda\times \Lambda\rightarrow \mathcal{B}(\mathcal{Y},\mathcal{Y})$ by
\begin{equation}\label{kernel_representation1}
\mathcal{K}(\alpha,\beta)\xi:=L_{\beta}(K(\alpha)\xi),
\ \alpha,\beta\in \Lambda,\xi\in\mathcal{Y}.
\end{equation}
Then by Theorem \ref{isometric-isomorphism1} we have that $\mathcal{K}$ is a reproducing kernel on $\Lambda$ and its vector-valued RKHS $\mathcal{H}_{\mathcal{K}}$, composed by functions with the form $\tilde{f}(\alpha):=
L_{\alpha}(f), \alpha\in\Lambda, f\in\mathcal{H},$
is isometrically isomorphic to $\mathcal{H}.$ For
each $\tilde{f}\in\mathcal{H}_{\mathcal{K}}$ set
$\tilde{f}(\Lambda_m):=[L_{\alpha_j}(f):j\in\mathbb{N}_m]$.
Due to the isometrically isomorphic between $\mathcal{H}$ and $\mathcal{H}_{\mathcal{K}}$ we have that $f_0$ is the minimizer of (\ref{general-regularization})
if and only if $\tilde{f}_0$ with
$\tilde{f}_0(\alpha):=L_{\alpha}(f_0),\ \alpha\in\Lambda$, is the minimizer of
\begin{equation*}
\inf_{f\in\mathcal{H}_{\mathcal{K}}}
\left\{Q(\tilde{f}(\Lambda_m),\mathbf{\xi})
+\lambda\phi(\|f\|_{\mathcal{H}_{K}})\right\}.
\end{equation*}
By Theorem \ref{representer-theorem} we get that the unique minimizer $\tilde{f}_0$ of the above optimization problem has the form
$$
\widetilde{f}_0=\sum_{j\in\mathbb{N}_m}
\mathcal{K}(\alpha_j,\cdot)\eta_j
$$
for some $\eta_j\in\mathcal{Y}, j\in\mathbb{N}_m.$ Hence, we conclude that there exists a unique minimizer of \eqref{general-regularization}. Moreover, since there holds for each $\beta\in\Lambda$,
$$
\widetilde{f}_0(\beta)
=\sum_{j\in\mathbb{N}_m}\mathcal{K}(\alpha_j,\beta)\eta_j
=\sum_{j\in\mathbb{N}_m}L_{\beta}(K(\alpha_j)\eta_j)
=L_{\beta}\left(\sum_{j\in\mathbb{N}_m}
K(\alpha_j)\eta_j\right),
$$
we get the representation of $f_0$ as in (\ref{representer_theorem_ORKHS-formula}).
\end{proof}

As a special case, we consider the regularization networks (\ref{general-regularization}), where the regularizer $\phi$ is chosen as $\phi(t):=t^2,\ t\in\mathbb{R}_{+}$ and the loss function $Q$ is chosen as
$$
Q(L_{\Lambda_m}(f), \mathbf{\xi})
:=\sum_{j\in\mathbb{N}_m}\|L_{\alpha_j}(f)
-\xi_j\|_{\mathcal{Y}}^2.
$$
It is clear that the loss function $Q$ and the regularizer $\phi$ satisfies the hypothesis in Theorem \ref{representer_theorem_ORKHS}. We then conclude by Theorem \ref{representer_theorem_ORKHS} that
there exists a unique minimizer $f_0$ of the regularization networks. Let $\mathcal{K}$ be the reproducing kernel defined by (\ref{kernel_representation1}). Observing from the
proof of Theorem \ref{representer_theorem_ORKHS},
we have that if $K(\alpha_j), j\in\mathbb{N}_m$ are linearly independent in $\mathcal{H}$, then the
parameters $\eta_j,j\in\mathbb{N}_m$, in (\ref{representer-theorem-formula}) satisfy the
linear equations
\begin{equation*}
\sum_{k\in\mathbb{N}_m}\mathcal{K}(\alpha_k,\alpha_j)\eta_k
+\lambda\eta_j=\xi_j,\ j\in\mathbb{N}_m.
\end{equation*}
By the definition of $\mathcal{K}$, the above equations are equivalent to
\begin{equation}\label{linearsystem1}
\sum_{k\in\mathbb{N}_m}L_{\alpha_j}(K(\alpha_k)\eta_k)
+\lambda\eta_j=\xi_j,\ j\in\mathbb{N}_m.
\end{equation}

Observing from representation (\ref{representer_theorem_ORKHS-formula}), we note that for the regularized learning from operator-valued data in an ORKHS, the target element can be obtained by the operator reproducing kernel in a desired manner. That is, it can be viewed as a linear combination of the kernel sections. As shown in the system of equations (\ref{linearsystem1}), the finite coefficients in the representation (\ref{representer_theorem_ORKHS-formula}) can be obtained by solving a minimization problem in a finite-dimensional space, which reduces to a linear system in some special cases. Furthermore, the desired result holds for all the finite set of linear operators in the family $\mathcal{L}:=\{L_{\alpha}:\alpha\in \Lambda\}$. All these facts show that ORKHSs and operator reproducing kernels provide a right framework for investigating learning from operator-valued data.

\section{Stability of Numerical Reconstruction Algorithms}

When we use a numerical algorithm to reconstruct an element from non-point-evaluation data (its operator values), stability of the algorithm is crucial.
In this section, we study stability of a numerical reconstruction algorithm using operator values in an ORKHS. As will be pointed out, the continuity of linear operators, used to obtain the non-point-evaluation data, on an ORKHS is necessary for the algorithm to
be stable.

Numerical reconstruction algorithms are often established for understanding a target element. There are two stages in the reconstruction. The first one is sampling used to obtain a finite set of functional-valued or operator-valued data processed digitally on a computer. The second one is the reconstruction of an approximation of the target element from the resulting sampled data. With the number of the sampled data increasing, the numerical reconstruction algorithm is expected to provide a more and more accurate  approximation for the target element. As an admissible numerical reconstruction algorithm, it needs to be stable. Such a property will ensure that the effect of the noise, emerging in both sampling and reconstruction, are not amplified through the underlying numerical reconstruction algorithm.

We now define the notion of the stability. Suppose that $\mathcal{H}$ is an ORKHS with respect to a family $\mathcal{L}:=\{L_{\alpha}:\alpha\in\Lambda\}$
of linear operators from $\mathcal{H}$ to a Hilbert
space $\mathcal{Y}$ and   $\widetilde{\Lambda}:=\{\alpha_j:j\in\mathbb{J}\}$ is a subset of $\Lambda$ with $\mathbb{J}$ being a countable index set. With respect to each finite set $\mathbb{J}_m$ of $m$ elements in $\mathbb{J}$, we set   $\widetilde{\Lambda}_m:=\{\alpha_j:j\in\mathbb{J}_m\}$ and define the sampling operator $\mathcal{I}_{\widetilde{\Lambda}_m}:\mathcal{H}\to \mathcal{Y}^m$ by
\begin{equation}\label{Sampling-operator-finite}
\mathcal{I}_{\widetilde{\Lambda}_m}(f)
:=[L_{\alpha_j}(f):j\in\mathbb{J}_m].
\end{equation}
We denote by $\mathcal{A}_m$ the reconstruction operator from $\mathcal{Y}^m$ to $\mathcal{H}$. The sampling operator and the reconstruction operator together describe a numerical reconstruction algorithm. Namely, $\mathcal{A}_m(\mathcal{I}_{\widetilde{\Lambda}_m}(f))$ describes a numerical algorithm for constructing an approximation of $f$. We say the numerical reconstruction algorithm is stable if there exists a positive constant $c$ such that for all $f\in\mathcal{H}$, all $m\in\mathbb{N}$ and all finite subset $\widetilde{\Lambda}_m$ of $\widetilde{\Lambda}$,
$$
\|\mathcal{A}_m(\mathcal{I}_{\widetilde{\Lambda}_m}(f))
\|_{\mathcal{H}}\leq c\|f\|_{\mathcal{H}}.
$$

We present in the following theorem a sufficient condition for the stability of a numerical reconstruction algorithm. To this end, we define the Hilbert space of square-summable sequence in $\mathcal{Y}$ on $\mathbb{J}$ by
$$
l^2(\mathbb{J},\mathcal{Y})
:=\left\{\{\xi_j:j\in\mathbb{J}\}:\sum_{j\in\mathbb{J}}
\|\xi_j\|^2_{\mathcal{Y}}<+\infty\right\}.
$$
With respect to $\widetilde{\Lambda}$ we define the sampling operator $\mathcal{I}_{\widetilde{\Lambda}}:\mathcal{H}
\rightarrow l^2(\mathbb{J},\mathcal{Y})$ for $f\in\mathcal{H}$ by $\mathcal{I}_{\widetilde{\Lambda}}(f)
:=\{L_{\alpha_j}(f):j\in\mathbb{J}\}$.
\begin{thm}\label{stable}
Let $\mathcal{H}$ be an ORKHS with respect to a family $\mathcal{L}:=\{L_{\alpha}:\alpha\in\Lambda\}$ of linear operators from $\mathcal{H}$ to a Hilbert
space $\mathcal{Y}$. Let $\widetilde{\Lambda}:=\{\alpha_j:j\in\mathbb{J}\}$ be a countable subset  of $\Lambda$. Suppose that the operator $\mathcal{I}_{\widetilde{\Lambda}}$ is continuous on $\mathcal{H}$ and the operators $\mathcal{I}_{\widetilde{\Lambda}_m}$, $\mathcal{A}_m$
are defined as above. If there exist a positive
constant $c$ such that for all $f\in\mathcal{H}$,
all $m\in\mathbb{N}$ and all finite subset $\widetilde{\Lambda}_m$ of $\widetilde{\Lambda}$,
$$
\|\mathcal{A}_m(\mathcal{I}_{\widetilde{\Lambda}_m}(f))
\|_{\mathcal{H}}\leq c
\|\mathcal{I}_{\widetilde{\Lambda}_m}(f)\|_{\mathcal{Y}^m},
$$
then the numerical reconstruction algorithm is stable.
\end{thm}
\begin{proof}
It follows from the continuity of operator  $\mathcal{I}_{\widetilde{\Lambda}}$ that there exists
a positive constant $c'$ such that for all $f\in\mathcal{H}$,
\begin{equation}\label{Sampling-operator-continuity}
\|\mathcal{I}_{\widetilde{\Lambda}}(f)
\|_{l^2(\mathbb{J},\mathcal{Y})}
\leq c'\|f\|_{\mathcal{H}}.
\end{equation}
By the definition of sampling operators and the norm of the space $l^2(\mathbb{J},\mathcal{Y})$, we get for all $f\in\mathcal{H}$, all $m\in\mathbb{N}$ and all finite subset $\widetilde{\Lambda}_m$ that
$$
\|\mathcal{I}_{\widetilde{\Lambda}_m}(f)
\|_{\mathcal{Y}^m}=\left(\sum_{j\in\mathbb{J}_m}
\|L_{\alpha_j}(f)\|_{\mathcal{Y}}^2\right)^{1/2}
\leq \|\mathcal{I}_{\widetilde{\Lambda}}(f)
\|_{l^2(\mathbb{J},\mathcal{Y})}.
$$
This combined with (\ref{Sampling-operator-continuity}) leads to the estimate
$$
\|\mathcal{I}_{\widetilde{\Lambda}_m}(f)
\|_{\mathcal{Y}^m}\leq c'\|f\|_{\mathcal{H}}.
$$
Substituting the above inequality into the inequality in the assumption, we have for all $f\in\mathcal{H}$, all $m\in\mathbb{N}$ and all finite subset $\widetilde{\Lambda}_m$ that
$$
\|\mathcal{A}_m(\mathcal{I}_{\widetilde{\Lambda}_m}(f))
\|_{\mathcal{H}}\leq cc'\|f\|_{\mathcal{H}}.
$$
This together with the definition of the stability of the numerical reconstruction algorithm proves the desired result.
\end{proof}

According to Theorem \ref{stable}, to ensure the stability of the numerical reconstruction algorithm, we need both the family $\{\mathcal{I}_{\widetilde{\Lambda}_m}:
\widetilde{\Lambda}_m\subseteq\widetilde{\Lambda}\}$ of sampling operators and the family
$\{\mathcal{A}_m: m\in\mathbb{N}\}$
of reconstruction operators to be uniformly bounded. We note that the continuity of the sampling operators on an ORKHS is necessary for the uniform boundedness of the family of the sampling operators. For the reconstruction operators, we shall show below that two classes of commonly used operators are uniformly bounded.

The first numerical reconstruction algorithm that we consider is the truncated reconstruction of elements in an FRKHS from a countable set of functional-valued data discussed in section 6. Let $\mathcal{H}$ be an FRKHS with respect to a family $\mathcal{L}$ of linear functionals
$\{L_{\alpha}:\alpha\in\Lambda\}$ and $K$ the functional reproducing kernel for $\mathcal{H}$. Suppose that  $\widetilde{\Lambda}:=\{\alpha_j: j\in\mathbb{J}\}$ is a countable set of $\Lambda$ such that
$\{K(\alpha_j): j\in\mathbb{J}\}$ constitutes a Riesz basis for $\mathcal{H}$. We denote by $\widetilde{K}$ the dual functional reproducing kernel of $K$ with respect to the set $\{\alpha_j: j\in\mathbb{J}\}$. Since only finite sampled data are used in practice, we consider the reconstruction operator $\mathcal{A}_m$ from $\mathbb{C}^m$ to $\mathcal{H}$ defined by
$$
\mathcal{A}_m(\mathcal{I}_{\widetilde{\Lambda}_m}(f))
:=\sum_{j\in\mathbb{J}_m}L_{\alpha_j}(f)
\widetilde{K}(\alpha_j),
$$
with $\mathbb{J}_m$ being a finite subset of $\mathbb{J}$ and $\widetilde{\Lambda}_m:=\{\alpha_j:j\in\mathbb{J}_m\}$.
The stability of the resulting numerical reconstruction algorithm for this case is established below.

\begin{thm}
Let $\mathcal{H}$ be an FRKHS with respect to
the family $\mathcal{F}$ of linear functionals $L_{\alpha}$, $\alpha\in\Lambda$ and $K$ the
functional reproducing kernel for $\mathcal{H}$. If for a countable subset $\widetilde{\Lambda}:=\{\alpha_j: j\in\mathbb{J}\}$ of $\Lambda$, the set $\{K(\alpha_j):  j\in\mathbb{J}\}$ constitutes a Riesz basis for $\mathcal{H}$, then the numerical reconstruction algorithm defined as above is stable.
\end{thm}
\begin{proof}
We prove this result by employing Theorem \ref{stable} with $\mathcal{Y}:=\mathbb{C}$. To this end, we verify the hypothesis of Theorem \ref{stable} for this special case.
Since $\{K(\alpha_j): j\in\mathbb{J}\}$ constitutes a Riesz basis for $\mathcal{H}$, there exist positive constants $A$ and $B$ such that
\begin{equation}\label{inequalityAB}
A\|f\|_{\mathcal{H}}\leq \left(\sum_{j\in\mathbb{J}}
|\langle f,K(\alpha_j)\rangle_{\mathcal{H}}|^2
\right)^{1/2}\leq B\|f\|_{\mathcal{H}},
\ \mbox{for all}\ f\in\mathcal{H}.
\end{equation}
By the first inequality in \eqref{inequalityAB} and the definition of
$\mathcal{A}_m$, we have that
$$
\|\mathcal{A}_m(\mathcal{I}_{\widetilde{\Lambda}_m}(f))
\|_{\mathcal{H}}\leq \frac{1}{A}
\left(\sum_{k\in\mathbb{J}}
\sum_{j\in\mathbb{J}_m}|L_{\alpha_j}(f)|^2
\left|\left< \widetilde{K}(\alpha_j),K(\alpha_k)
\right>_{\mathcal{H}}\right|^2\right)^{1/2}.
$$
Using the biorthogonal condition of the sequence pair  $\{\widetilde{K}(\alpha_j), K(\alpha_j): j\in\mathbb{J}\}$ in the right hand side of the inequality above yields that for all $f\in\mathcal{H}$, all $m\in\mathbb{N}$ and all finite subset $\widetilde{\Lambda}_m$,
$$
\|\mathcal{A}_m(\mathcal{I}_{\widetilde{\Lambda}_m}(f))
\|_{\mathcal{H}}\leq  \frac{1}{A} \left(\sum_{j\in\mathbb{J}_m}|L_{\alpha_j}(f)|^2\right)^{1/2}
=\frac{1}{A}\|\mathcal{I}_{\widetilde{\Lambda}_m}(f)
\|_{\mathbb{C}^m}.
$$
Employing the reproducing property of $K$ in the second inequality in (\ref{inequalityAB}), we get that
$$
\|\mathcal{I}_{\widetilde{\Lambda}}(f)
\|_{l^2(\mathbb{J}, \mathbb{C})}\leq B\|f\|_{\mathcal{H}},
\ \mbox{for all}\ f\in\mathcal{H},
$$
yielding the continuity of operator $\mathcal{I}_{\widetilde{\Lambda}}$.
By Theorem \ref{stable}, we conclude the desired result.
\end{proof}

We next consider the regularization network, discussed in section 7, which infers an approximation of a target element by solving the optimization problem
\begin{equation}\label{regularization-network}
\inf_{g\in\mathcal{H}}\left\{\sum_{j\in\mathbb{N}_m}
\|L_{\alpha_j}(g)-\xi_j\|_{\mathcal{Y}}^2
+\lambda\|g\|_{\mathcal{H}}^2\right\}.
\end{equation}
Suppose that $\mathcal{H}$ is an ORKHS with respect to a family $\mathcal{L}$ of linear operators
$\{L_{\alpha}:\alpha\in\Lambda\}$ from $\mathcal{H}$ to $\mathcal{Y}$ and $K$ is the operator reproducing kernel for $\mathcal{H}$. Let $\widetilde{\Lambda}:=\{\alpha_j:j\in\mathbb{J}\}$ be a countable subset of $\Lambda$. For each finite subset $\widetilde{\Lambda}_m:=\{\alpha_j:j\in\mathbb{J}_m\}$, the sampling operator $\mathcal{I}_{\widetilde{\Lambda}_m}:\mathcal{H}
\rightarrow \mathcal{Y}^m$ is defined as in (\ref{Sampling-operator-finite}). For each $f\in\mathcal{H}$, we let $f_{\widetilde{\Lambda}_m}$ be the unique minimizer of (\ref{regularization-network}) with $[\xi_j:j\in\mathbb{J}_m]
:=\mathcal{I}_{\widetilde{\Lambda}_m}(f).$ Accordingly,  the reconstruction operator
$\mathcal{A}_m:\mathcal{Y}^m\rightarrow\mathcal{H}$
is defined by
$$
\mathcal{A}_m(\mathcal{I}_{\widetilde{\Lambda}_m}(f))
:=f_{\widetilde{\Lambda}_m}.
$$

The next result concerns the stability of the corresponding numerical reconstruction algorithm.
To this end, we introduce another operator by
considering the following regularization problem
\begin{equation}\label{Tikhonov-regularization}
\inf_{g\in\mathcal{H}}\left\{\sum_{j\in\mathbb{J}}
\|L_{\alpha_j}(g)-\xi_j\|^2_{\mathcal{Y}}+
\lambda\|g\|_{\mathcal{H}}^2\right\}.
\end{equation}
It follows from theory of the Tikhonov regularization \cite{K} that if the operator $\mathcal{I}_{\widetilde{\Lambda}}:\mathcal{H}
\rightarrow l^2(\mathbb{J},\mathcal{Y})$ is continuous on $\mathcal{H}$, there exists for any $\xi:=\{\xi_j:j\in\mathbb{J}\}\in l^2(\mathbb{J},\mathcal{Y})$ a unique minimizer $g_0$ of the optimization problem (\ref{Tikhonov-regularization}). Furthermore, the minimizer $g_0$ is given by the unique solution of the equation
$\lambda g_0+\mathcal{I}_{\widetilde{\Lambda}}^*
\mathcal{I}_{\widetilde{\Lambda}}(g_0)=
\mathcal{I}_{\widetilde{\Lambda}}^*(\xi)$. We define the operator $\mathcal{A}_{\widetilde{\Lambda}}:
l^2(\mathbb{J},\mathcal{Y})\rightarrow\mathcal{H}$ by
\begin{equation*}\label{Tikhonov-regularization-operator}
\mathcal{A}_{\widetilde{\Lambda}}
:=(\lambda \mathcal{I}+\mathcal{I}_{\widetilde{\Lambda}}^*
\mathcal{I}_{\widetilde{\Lambda}})^{-1}
\mathcal{I}_{\widetilde{\Lambda}}^*,
\end{equation*}
where $\mathcal{I}$ denotes the identity operator on $\mathcal{H}$.
\begin{thm}
Suppose that $\mathcal{H}$ is an ORKHS with respect to
the set $\mathcal{L}:=\{L_{\alpha}:\alpha\in \Lambda\}$
of linear operators from $\mathcal{H}$ to $\mathcal{Y}$ and $K$ is the operator reproducing kernel for $\mathcal{H}$. If the countable subset $\widetilde{\Lambda}:=\{\alpha_j:j\in\mathbb{J}\}$ of $\Lambda$ is chosen to satisfy the condition that there exists a positive constant $c$ such that for all $\xi:=\{\xi_j:j\in\mathbb{J}\}\in l^2(\mathbb{J},\mathcal{Y})$ with $\|\xi\|_{l^2(\mathbb{J},\mathcal{Y})}=1$,
\begin{equation}\label{inequality-c}
\left\|\sum_{j\in\mathbb{J}}K(\alpha_j)\xi_j\right
\|_{\mathcal{H}}\leq c,
\end{equation}
then the corresponding numerical reconstruction algorithm defined as above is stable.
\end{thm}
\begin{proof}
Again, we prove this result by employing Theorem \ref{stable}. It follows for each $f\in\mathcal{H}$ that
$$
\|\mathcal{I}_{\widetilde{\Lambda}}(f)\|_{l^2(\mathbb{J}, \mathcal{Y})}=\sup_{\|\xi\|_{l^2(\mathbb{J}, \mathcal{Y})}=1}\left|\langle
\mathcal{I}_{\widetilde{\Lambda}}(f),
\xi\rangle_{l^2(\mathbb{J},\mathcal{Y})}\right|
=\sup_{\|\xi\|_{l^2(\mathbb{J}, \mathcal{Y})}=1}\left|\sum_{j\in\mathbb{J}}\langle
L_{\alpha_j}(f),\xi_j\rangle_{\mathcal{Y}}\right|.
$$
Using the reproducing property of $K$, we get that
$$
\|\mathcal{I}_{\widetilde{\Lambda}}(f)\|_{l^2(\mathbb{J}, \mathcal{Y})}=\sup_{\|\xi\|_{l^2(\mathbb{J}, \mathcal{Y})}=1}\left|\langle f,
\sum_{j\in\mathbb{J}}K(\alpha_j)\xi_j
\rangle_{\mathcal{H}}\right|.
$$
This together with inequality (\ref{inequality-c}) implies that for all $f\in\mathcal{H}$,
$$
\|\mathcal{I}_{\widetilde{\Lambda}}(f)\|_{l^2(\mathbb{J}, \mathcal{Y})}\leq c\|f\|_{\mathcal{H}}.
$$
That is, the operator $\mathcal{I}_{\widetilde{\Lambda}}$ is continuous on $\mathcal{H}$.

It follows for each $g\in\mathcal{H}$ that
$$
\langle\lambda g+\mathcal{I}_{\widetilde{\Lambda}}^*
\mathcal{I}_{\widetilde{\Lambda}}(g),
g\rangle_{\mathcal{H}}=\lambda\|g\|_{\mathcal{H}}^2+
\|\mathcal{I}_{\widetilde{\Lambda}}(g)\|_{\mathcal{H}}^2
\geq\lambda\|g\|_{\mathcal{H}}^2.
$$
By the Lax-Milgram theorem \cite{Yo}, we have that $(\lambda \mathcal{I}+\mathcal{I}_{\widetilde{\Lambda}}^*
\mathcal{I}_{\widetilde{\Lambda}})^{-1}$ is continuous on $\mathcal{H}$. According to the definition of $\mathcal{A}_{\widetilde{\Lambda}}$, there exists a positive constant $c'$ such that for all $f\in\mathcal{H}$,
\begin{equation}\label{Tikhonov-regularization-inequality}
\|\mathcal{A}_{\widetilde{\Lambda}}
(\mathcal{I}_{\widetilde{\Lambda}}(f))
\|_{\mathcal{H}}\leq c'
\|\mathcal{I}_{\widetilde{\Lambda}}(f)
\|_{l^2(\mathbb{J},\mathcal{Y})}.
\end{equation}
For each finite subset $\widetilde{\Lambda}_m$ and each $f\in\mathcal{H}$, we choose $\tilde{f}\in\mathcal{H}$ to satisfy that $L_{\alpha_j}(\tilde{f})=L_{\alpha_j}(f)$ for $j\in\mathbb{J}_m$ and $L_{\alpha_j}(\tilde{f})=0$ for $j\notin\mathbb{J}_m$. By the definition of $\mathcal{A}_m$ and $\mathcal{A}_{\widetilde{\Lambda}}$, we get that  $\mathcal{A}_m(\mathcal{I}_{\widetilde{\Lambda}_m}(f))
=\mathcal{A}_{\widetilde{\Lambda}}
(\mathcal{I}_{\widetilde{\Lambda}}(\tilde{f})).$
This combined with (\ref{Tikhonov-regularization-inequality}) leads to that there holds for all $f\in\mathcal{H}$, all $m\in\mathbb{N}$ and all finite subset $\widetilde{\Lambda}_m$,
$$
\|\mathcal{A}_m(\mathcal{I}_{\widetilde{\Lambda}_m}(f))
\|_{\mathcal{H}}=\|\mathcal{A}_{\widetilde{\Lambda}}
(\mathcal{I}_{\widetilde{\Lambda}}(\tilde{f}))
\|_{\mathcal{H}}\leq c'
\|\mathcal{I}_{\widetilde{\Lambda}}(\tilde{f})
\|_{l^2(\mathbb{J},\mathcal{Y})}=c'
\|\mathcal{I}_{\widetilde{\Lambda}_m}(f)
\|_{\mathcal{Y}^m}.
$$
Hence, the stability of the corresponding numerical reconstruction algorithm follows directly from Theorem \ref{stable}.
\end{proof}

\section{Conclusion}

We have introduced the notion of the ORKHS and established the bijective correspondence between an ORKHS and its operator reproducing kernel. Particular attention has been paid to the perfect ORKHS, which relates to two families of continuous linear operators with one being the family of the point-evaluation operators. The characterization and the specific examples for the perfect ORKHSs with respect to the integral operators have been provided. To demonstrate the usefulness of ORKHSs, we have discussed in the ORKHSs setting the sampling theorems and the regularized learning schemes from non-point-evaluation data. We have considered stability of numerical reconstruction algorithms in ORKHSs. This work provides 
an appropriate mathematical foundation for kernel methods when non-point-evaluation data are used.

\bigskip

\noindent{\bf Acknowledgment:} This research is supported in part by the United States National Science Foundation under grant DMS-152233, by Guangdong Provincial Government of China through the ``Computational Science Innovative Research Team'' program and by the Natural Science Foundation of China under grants 11471013 and 11301208.

\end{document}